\documentclass[12pt]{amsart}

\usepackage{color}
\usepackage{amsthm}
\usepackage{amsxtra}
\usepackage{amssymb}
\usepackage{graphicx}  

\newcommand{\hph}{{\widehat \varphi}}
\newcommand{\hG}{{\widehat G}}
\newcommand{\Op}{{\operatorname{Op}^{{w}}_h}}
\newcommand{\Opt}{{\operatorname{Op}^{{w}}_{\tilde h}}}
\newcommand{\be}{\begin{equation}}
\newcommand{\ee}{\end{equation}}

\newcommand{\ra}{\rangle}
\newcommand{\la}{\langle}

\newcommand{\cB}{{\mathcal B}}
\newcommand{\CI}{{\mathcal C}^\infty }
\newcommand{\CIc}{{\mathcal C}^\infty_{\rm{c}} }

\newcommand{\cE}{{\mathcal E}}

\newcommand{\cH}{\mathcal H}
\newcommand{\cK}{\mathcal K}
\newcommand{\cM}{{\mathcal M}}
\newcommand{\cO}{{\mathcal O}}
\newcommand{\Oo}{{\mathcal O}} 
\newcommand{\cR}{{\mathcal R}}
\newcommand{\cS}{{\mathcal S}}
\newcommand{\cT}{{\mathcal T}}
\newcommand{\UU}{{\mathcal U}}
\newcommand{\cU}{{\mathcal U}}
\newcommand{\VV}{{\mathcal V}}
\newcommand{\cV}{{\mathcal V}}
\newcommand{\cW}{{\mathcal W}}

\newcommand{\CC}{{\mathbb C}}
\newcommand{\NN}{{\mathbb N}}

\newcommand{\RR}{{\mathbb R}}
\newcommand{\IR}{{\mathbb R}}
\newcommand{\TT}{{\mathcal T}}
\newcommand{\SP}{{\mathbb S}}

\newcommand{\tA}{\widetilde A}
\newcommand{\tD}{\widetilde D}
\newcommand{\tF}{\widetilde F}
\renewcommand{\th}{\tilde h}
\newcommand{\tM}{\widetilde M}
\newcommand{\tcM}{\widetilde{\mathcal M}}
\newcommand{\tS}{\widetilde S}
\newcommand{\tT}{\widetilde T}
\newcommand{\tU}{\widetilde U}
\newcommand{\tVV}{\widetilde {\mathcal V}}
\newcommand{\tkappa}{\widetilde \kappa}
\newcommand{\tPi}{\widetilde \Pi}
\newcommand{\trh}{{\tilde \rho}}
\newcommand{\trho}{{\tilde \rho}}
\newcommand{\tv}{{\tilde \varphi}}
\newcommand{\tvarphi}{{\tilde \varphi}}
\newcommand{\ST}{{\widetilde S_{\frac12}}}
\newcommand{\PT}{{\widetilde \Psi_{\frac12}}}

\newcommand{\defeq}{\stackrel{\rm{def}}{=}}

\newcommand{\vol}{\operatorname{vol}}
\newcommand{\rank}{\operatorname{rank}}
\newcommand{\diag}{\operatorname{diag}}

\newcommand{\supp}{\operatorname{supp}}

\newcommand{\coker}{\operatorname{coker}}
\renewcommand{\dim}{\operatorname{dim}}
\newcommand{\WF}{\operatorname{WF}}
\newcommand{\WFh}{\operatorname{WF}_h}

\newcommand{\loc}{\operatorname{loc}}

\newcommand{\tr}{\operatorname{tr}}
\newcommand{\Spec}{\operatorname{Spec}}

\newcommand{\rest}{\!\!\restriction}

\renewcommand{\Re}{\mathop{\rm Re}\nolimits}
\renewcommand{\Im}{\mathop{\rm Im}\nolimits}
\newcommand{\ad}{\operatorname{ad}}

\newcommand{\neigh}{\operatorname{neigh}}
\newcommand{\esss}{\operatorname{ess-supp}}

\theoremstyle{plain}

\newtheorem{thm}{Theorem}
\newtheorem{prop}{Proposition}[section]

\newtheorem{lem}[prop]{Lemma}

\theoremstyle{definition}

\newtheorem{defn}[prop]{Definition} 

\numberwithin{equation}{section}

\def\bbbone{{\mathchoice {1\mskip-4mu {\rm{l}}} {1\mskip-4mu {\rm{l}}}
{ 1\mskip-4.5mu {\rm{l}}} { 1\mskip-5mu {\rm{l}}}}}

\def\squarebox#1{\hbox to #1{\hfill\vbox to #1{\vfill}}} 
\newcommand{\norm}[1]{\Vert#1\Vert}
\newcommand{\stopthm}{\hfill\hfill\vbox{\hrule\hbox{\vrule\squarebox 
                 {.667em}\vrule}\hrule}\smallskip}

\newcommand{\Id}{{ I}}

\newcommand{\eps}{{\epsilon}}
\newcommand{\vareps}{{\varepsilon}}
\newcommand{\gz}{{g_0}}

\usepackage{amsxtra}

\ifx\pdfoutput\undefined
  \DeclareGraphicsExtensions{.pstex, .eps}
\else
  \ifx\pdfoutput\relax
    \DeclareGraphicsExtensions{.pstex, .eps}
  \else
    \ifnum\pdfoutput>0
      \DeclareGraphicsExtensions{.pdf}
    \else
      \DeclareGraphicsExtensions{.pstex, .eps}
    \fi
  \fi
\fi

\title
{
Fractal Weyl law for open quantum chaotic maps}
\author[S. Nonnenmacher]
{St\'ephane Nonnenmacher}
\author[J. Sj\"ostrand]
{Johannes Sj\"ostrand}
\author[M. Zworski]
{Maciej Zworski}
\address{Institut de Physique Th\'eorique\\
CEA/DSM/PhT, Unit\'e de recherche associ\'ee au CNRS\\
CEA-Saclay\\
91191 Gif-sur-Yvette, France}
\email{snonnenmacher@cea.fr}
\address{Institut de Math\'ematiques de Bourgogne, UFR Science
et Techniques, 9 Avenue Alain Savary -- B.P. 47870, 21078 Dijon 
CEDEX, France} 
\email{jo7567sj@u-bourgogne.fr}
\address{Mathematics Department, University of California \\
Evans Hall, Berkeley, CA 94720, USA}
\email{zworski@math.berkeley.edu}

\begin{document}    

\maketitle   


\section{Introduction}
\label{int}
In this paper we study semiclassical quantizations of Poincar\'e
maps arising in scattering problems with hyperbolic classical 
flows which have topologically one dimensional trapped sets. 
The main application is the proof of a fractal Weyl upper bound
for the number of resonances/scattering poles in small domains. 

The reduction of open scattering problems with hyperbolic classical 
flows to quantizations of open maps has been recently described by the 
authors in \cite{NSZ1}. In this introduction we show how
the main result of the current paper applies to the case of scattering by several
convex obstacles, and explain ideas of the proof in that particular setting.

\begin{center}
\includegraphics[scale=0.4]{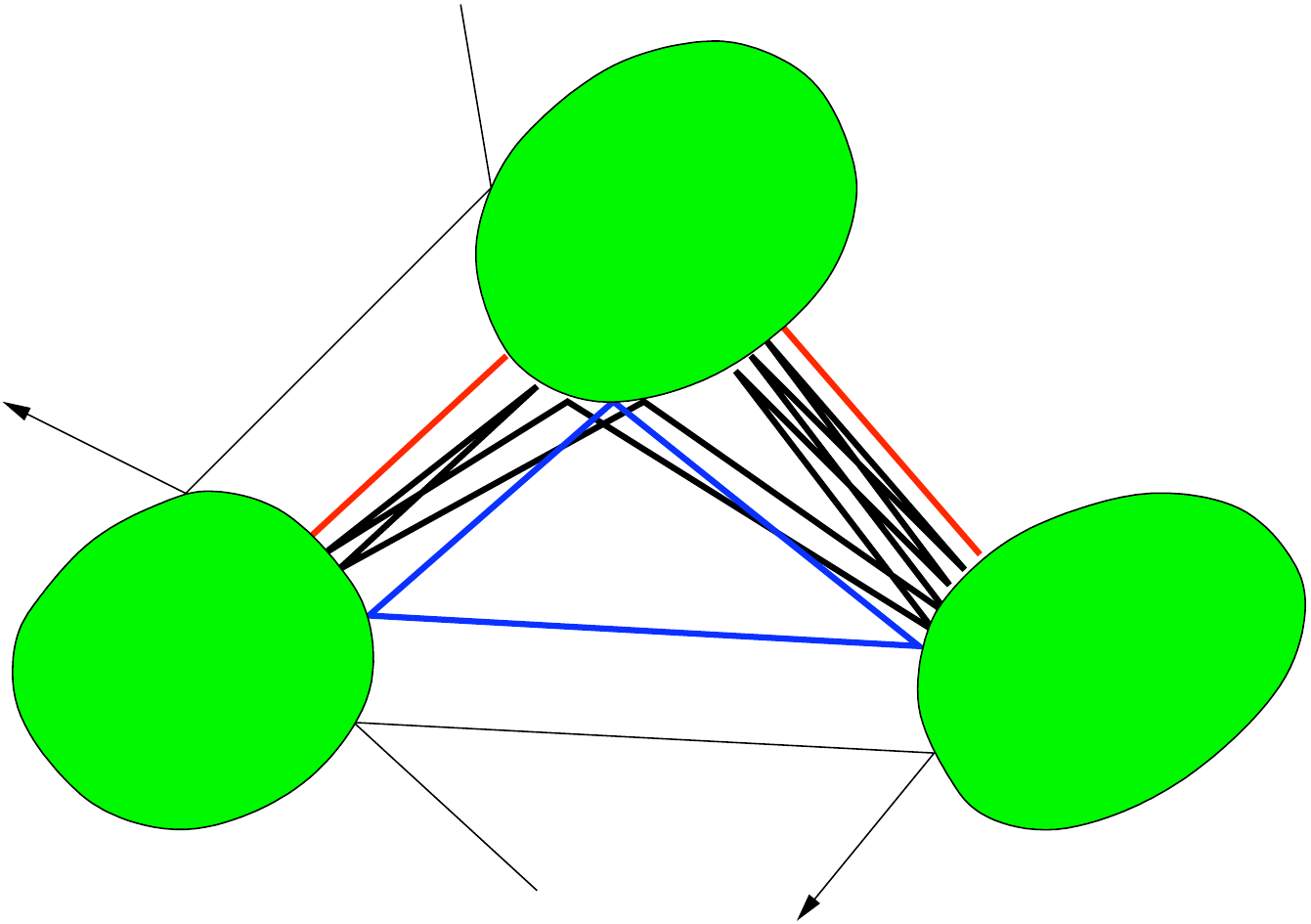}
\end{center}

Let $ {\mathcal O} = \bigcup_{j=1}^J {\mathcal O}_j $ where $ {\mathcal O}_j \Subset \RR^n $
are open, strictly convex, have smooth boundaries, and satisfy the {\em Ikawa condition}:
\begin{equation}
\label{eq:Ik}
  \overline {\mathcal O}_k \cap {\text{convex hull}}( 
\overline {\mathcal O}_j \cup \overline {\mathcal O}_\ell ) 
= \emptyset \,, \ \ j \neq k \neq \ell \,. 
\end{equation}
The classical flow on  $ ( \RR^n \setminus {\mathcal O} ) \times \SP^{n-1} $ (with the 
second factor responsible for the direction) is defined by free motion
away from the obstacles, and 
normal reflection on the obstacles -- see the figure above and 
also \S \ref{damb} for a precise definition. An important
dynamical object is the {\em trapped set}, $ K $, consisting of 
$ ( x , \xi) \in ( \RR^n \setminus {\mathcal O} ) \times \SP^{n-1} $ which do not 
escape to infinity under forward or backward flow.

The high frequency waves on $ \RR^n \setminus {\mathcal O} $ 
are given as solutions of the Helmholtz equation with Dirichlet boundary 
conditions:
\[ (- \Delta - \lambda^2)u=0,\quad u\in  
 H^2( {\RR}^n \setminus {\mathcal O} ) \cap H^1_{  0}( {\RR}^n
 \setminus {\mathcal O} ) ,\quad \lambda\in\IR\,.
\]
The scattering resonances are defined as poles of the meromorphic continuation of
\[  R ( \lambda ) = ( - \Delta - \lambda^2 )^{-1} \; : \; L^2_{\rm{comp}} 
( {\RR}^n \setminus {\mathcal O} ) \longrightarrow
L_{\rm{loc}}^2 ( {\RR}^n \setminus {\mathcal O} ) \]
to the complex plane for $ n $ odd and to the logarithmic plane
when $ n $ is even.

The multiplicity of a (non-zero) resonance $\lambda$ is given by 
\begin{equation}
\label{eq:mul}
m_R ( \lambda ) = \text{rank} \; \oint_{\gamma} R ( \zeta ) d \zeta \,, 
\ \ \gamma : t \mapsto \lambda + \epsilon e^{ 2 \pi i t}  \,,  \ 0 \leq t < 1 \,, \ \ 
0 < \epsilon \ll 1 \,. 
\end{equation}

Our general result for hyperbolic quantum monodromy operators (see
Theorem~\ref{th:bound} in \S\ref{upbd}) leads to the following 
result for scattering by several convex obstacles. 
\begin{thm}
\label{th:1}
Let $ {\mathcal O} = \bigcup_{j=1}^J {\mathcal O}_j $ be a union of
strictly convex smooth obstacles satisfying \eqref{eq:Ik}. Then for any fixed $ 
\alpha > 0 $,
\begin{equation}
\label{eq:t1} \sum_{\substack{- \alpha < \Im \lambda\\ r \leq |\lambda | \leq r + 1}}
m_R ( \lambda ) = {\mathcal O} (  r^{\mu+0 } ) \,, 
\ \ r \longrightarrow \infty \,, 
\end{equation}
where  $ 2 \mu + 1 $ is the box dimension of the 
{\em trapped set}.

If the trapped set is of
pure dimension (see \S \ref{upbd}) then the bound is $ {\mathcal O} (r^\mu ) $.
In the case of  $n = 2 $ the trapped set is always of pure dimension, which
is its Hausdorff dimension.
\end{thm}

We should stress that even a weaker bound, 
\[ \sum_{\substack{- \alpha < \Im \lambda \\ 1 \leq |\lambda | \leq r}}
 m_R ( \lambda ) =  {\mathcal O} (r^{\mu + 1 + 0} ) \,, \]
corresponding the
standard Weyl estimate $ {\mathcal O} (r^n ) $ for frequencies 
of a bounded domain, 
was not known previously. Despite various positive indications
which will be described below no lower bound is known in this setting.

The study of counting of scattering resonances was initiated in 
physics by Regge \cite{Regg} and in mathematics  by 
Melrose \cite{Mel3} who proved a global bound in odd dimensions:
\[  \sum_{ |\lambda| \leq r   } 
m_R ( \lambda ) = {\mathcal O} (  r^{n } ) \,. \]
This bound is optimal for the sphere and for obstacles with certain 
elliptic trapped trajectories but the existence of a general 
lower bound remains open -- see
\cite{St1},\cite{St2} and references given there.
For even dimensions an analogous bound was established by Vodev \cite{Vod}

The fractal bound \eqref{eq:t1} for obstacles 
was predicted by the second author in \cite{SjDuke} where 
fractal upper bounds for the number of resonances were established for
a wide class of semiclassical operators 
with analytic coefficients (such as $ - h^2 \Delta + V $ with 
$ V $ equal to a restriction of a suitable holomorphic function in a complex neighbourhood
of $ \RR^n $); promising numerics were obtained in \cite{Lin02} for
the case of a three bump potential, and in \cite{Ramil+09} for the
H\'enon--Heiles Hamiltonian. The bound of the type \eqref{eq:t1} was first
proved for resonances associated to hyperbolic quotients by 
Schottky groups without parabolic elements 
(that is for the zeros of the Selberg zeta functions) in \cite{GLZ}, and
for a general class of semiclassical problems in \cite{SZ10}. 
Theorem \ref{th:bound} below provides a new proof of the result in \cite{SZ10}
in the case of topologically one dimensional trapped sets. The new proof is
simpler by avoiding the complicated second microlocalization procedure
of \cite[\S 5]{SZ10}. The reduction to Poincar\'e sections obtained using the
Schr\"odinger propagator \cite{NSZ1} replaces that step.
The only 
rigorous fractal lower bound was obtained in a special toy model of
open quantum map in  
\cite{NZ1}.  For some classes of hyperbolic surfaces lower bounds
involving the dimension were obtained in \cite{JN}. 

\begin{figure}
\begin{center}
\includegraphics[width=5in]{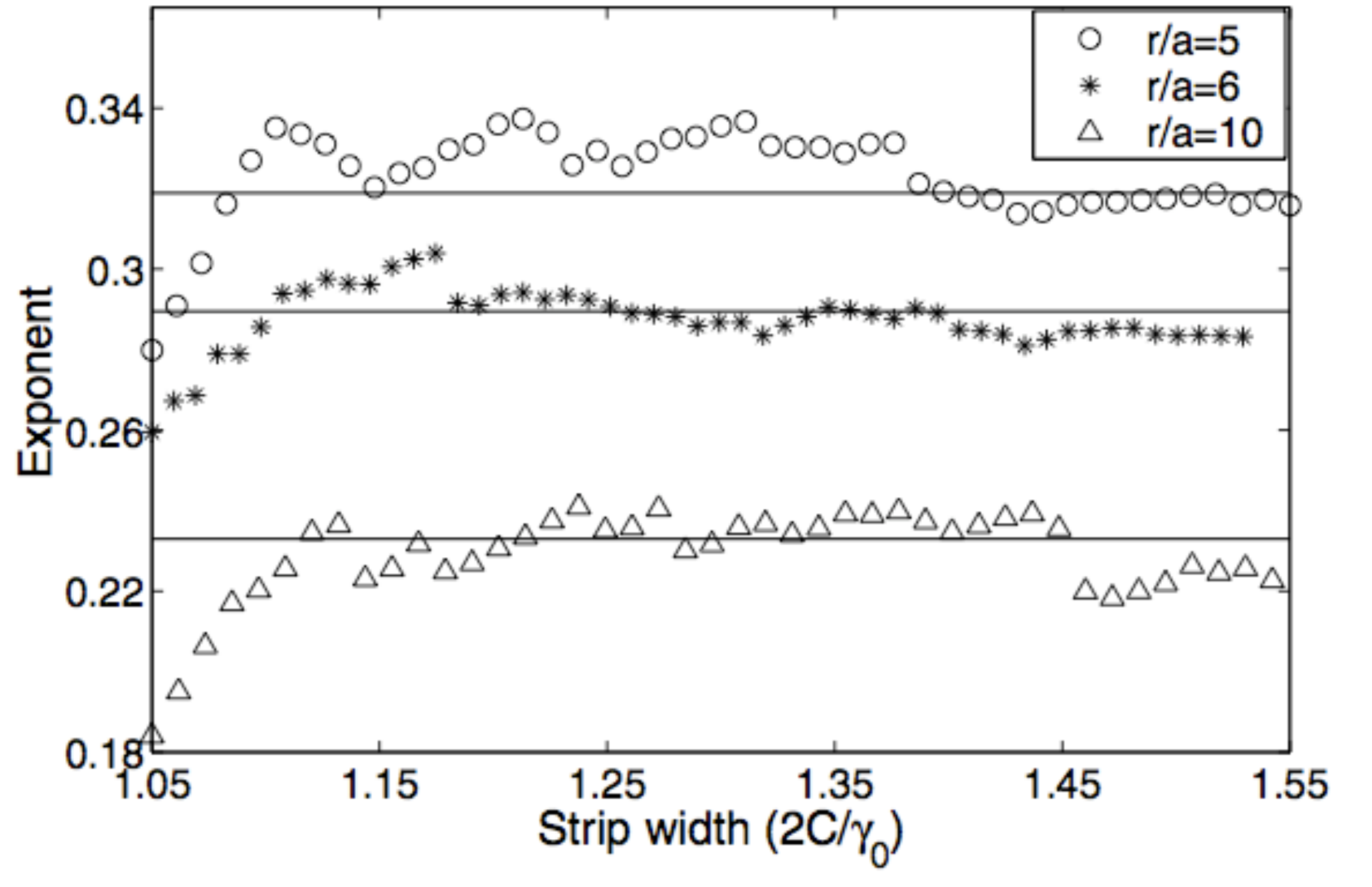} 
\caption{\label{f:lsz} This figure, taken from \cite{LSZ}, 
shows resonances computed using the semiclassical zeta 
function for scattering by three symmetrically placed discs of radii $ a $ and distances
$ r $ between centers. The horizontal axis represents $ 2 \alpha / \gamma_0 $ where
$ \gamma_0 $ is a classical rate of decay (the imaginary parts of resonances tend
to cluster at $ \Im \lambda \sim - \gamma_0 /2 $ --- see \cite{LSZ}) and the 
vertical axis corresponds to the best fit for the slope of 
$ \log N ( \alpha, r ) / \log r $, that is the exponent in the fractal Weyl law 
\eqref{eq:gl}. The three lines correspond to the three values of the Hausdorff 
dimension $ \mu $.}
\end{center}
\end{figure}

In the case of three discs in the plane, results of numerical experiments 
based on semiclassical zeta function calculations \cite{GaRi}
were presented in \cite{LSZ}. They suggest that a global version of the
Weyl law might be valid:
\begin{equation}
\label{eq:gl}
N( \alpha, r ) \defeq \sum_{ \Im \lambda > - \alpha, \, \,  |\lambda | \leq r   } 
m_R ( \lambda ) \sim   C ( \alpha )  r^{\mu +1 }  \,, 
\ \ r \longrightarrow \infty \,,
\end{equation}
see Fig.~\ref{f:lsz}. A similar study for the scattering by four
hard spheres centered on a tetrahedron in three dimensions was recently conducted in
\cite{EberMainWunn10}, and lead to a reasonable agreement with the
above fractal Weyl law, at least for $\alpha$ large enough. We stress however that the method of
calculation based on the zeta function, although widely accepted 
in the physics literature, does not have a rigorous justification
and may well be inaccurate. 
Experimental validity 
of the fractal Weyl laws is now investigated in the setting of 
microwave cavities \cite{PUSZ} --- see Fig.~\ref{f:mar}. 
The theoretical model is precisely the one for which 
Theorem \ref{th:1} holds. The fractal Weyl law has been considered (and numerically
checked) for various open chaotic quantum maps like
the open kicked rotator \cite{SchoTwor04} and the open baker's map
\cite{NZ1,Pedr+09}. Theorems \ref{t:2gr} and \ref{t:2gr2} below lead to a rigorous
fractal Weyl upper bound in this setting of open quantum maps with a
hyperbolic trapped set. 
Fractal Weyl laws have also been proposed in other types of chaotic scattering
systems, like dielectric cavities \cite{WierMain08}, as well as for
resonances associated with classical dynamical systems \cite{StrainZwo04,Chr05,ErmaShepe10}.

\begin{figure}
\begin{center}
\includegraphics[width=5in]{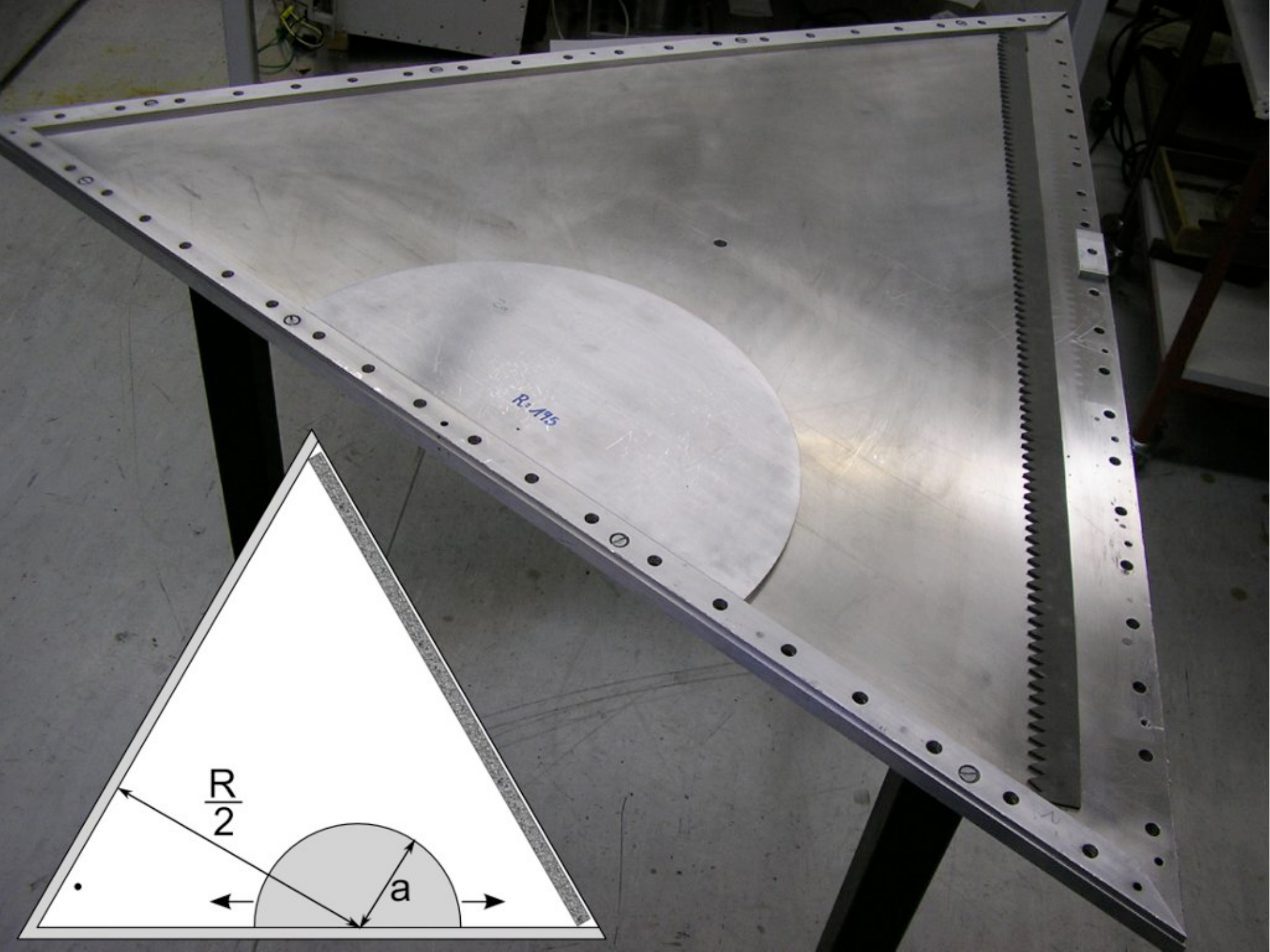} 
\caption{\label{f:mar} The experimental set-up of the Marburg
quantum chaos group {\tt http://www.physik.uni-marburg.de} for
the five disc, symmetry reduced, system. The hard walls 
correspond to the Dirichlet boundary condition, that is 
to odd solutions (by reflection) of the full problem. The
absorbing barrier, which produces negligible reflection at the 
considered range of frequencies, models escape to infinity.}
\end{center}
\end{figure}

The proof of Theorem \ref{th:1} uses Theorem \ref{th:bound} 
below, which holds for general {\em hyperbolic quantum monodromy 
operators} defined in \S \ref{2gr}. Here we will sketch how 
these operators appear in the framework of scattering by several convex bodies.
In the case of two obstacles they were already used in the precise
study of resonances conducted by Ch.~G\'erard \cite{Ge}. The detailed
analysis will be presented in \S \ref{assc}.

To connect this setting with the general semiclassical point 
of view, we write 
$$  z =\frac{i}{h} (h^2\lambda^2  - 1) \,,  \ \ h \sim |\Re
\lambda|^{-1} \,, \ \ z \sim 0 \,, $$ 
and consider the problem 
\[  P ( z )u = 0\,,\quad  P(z) \defeq   \left( \frac  i  h ( -h^2 \Delta - 1 ) - z \right) u ( x ) \,,  \ \ u \rest_{\partial \Oo} = 0 \,, \ \ 
\text{$ u $ outgoing.} \]
The precise meaning of ``outgoing'' will be recalled in \S \ref{assc}. 
This rescaling means that investigating resonances in $\{ r \leq |\lambda | 
\leq r + 1 ,\, \Im \lambda > - \alpha \} $, corresponds to investigating
the poles of the meromorphic continuation of $ P ( z )^{-1} $ in a
fixed size neighbourhood of $ 0 $ (in these notations the resonances are situated on the
half-plane $\Re z>0$).

The study of this meromorphic continuation can be reduced to the boundary 
through the following well-known construction. To each obstacle $\Oo_j$ we
associate a Poisson operator $ H_j ( z ) \; : \; 
\CI ( \partial \Oo_j ) \rightarrow \CI ( \RR^n \setminus \Oo_j ) $
defined by
\[  P ( z ) H_j ( z ) v ( x ) = 0\,, \ \ x \in 
\RR^n \setminus {\mathcal O}_j  \,, \ \ H_j( z ) v \rest_{\partial \Oo_j} 
=  v \,, \ \ 
\text{$ H_j ( z ) v  $ outgoing.} \]
Besides, let $ \gamma_j : \CI ( \RR^n ) \rightarrow \CI ( \partial \Oo_j ) $
be the restriction operators,  $ \gamma_j u \defeq  u \rest_{\partial \Oo_j} $.
As described in detail in \S\ref{bvp} the study of the resolvent $ P ( z)^{-1} $ 
can be reduced to the study of 
\[ ( \Id - \cM ( z,h ) )^{-1}  \; : \; \bigoplus_{j=1}^J \CI ( \partial \Oo_j ) 
\longrightarrow \bigoplus_{j=1}^J \CI ( \partial \Oo_j ) \,, \]
where 
\begin{equation}
\label{eq:MM}  
\left( \cM   (z,h ) \right)_{ij} = \left\{ \begin{array}{ll} 
- \gamma_i H_j ( z ) & i \neq j \,, \\
\ \ \ \ 0 & i = j \,. \end{array} \right. \end{equation}
The 
structure of the operators $ \gamma_i H_j ( z ) $
is quite complicated due to diffractive phenomena.  In the
semiclassical/large
frequency regime and for complex values of $ z $, 
the operators $ H_j ( z ) $  have been analysed by G\'erard
\cite[Appendix]{Ge}  ($ \Im z > - C $) and by Stefanov-Vodev 
\cite[Appendix]{StVo} ($\Im z > - C \log ( 1/h ) $). We refer to these
papers and to \cite[Chapter 24]{Hor2} and \cite{MT} for more information 
about propagation of singularities for boundary value problems and for
more references. 

Using the propagation of singularities results obtained from the 
parametrix  (see  \S \ref{sep}) the issue of invertibility of $ (\Id - {\mathcal M} ( z , h ))  $ can 
be microlocalized to a neighbourhood of the trapped set, where the
structure of $\cM( z , h )$ is described using $ h$-Fourier integral operators --- see
\S \ref{fio} for the definition of these objects.

\begin{figure}
\begin{center}
\includegraphics[width=0.45\textwidth]{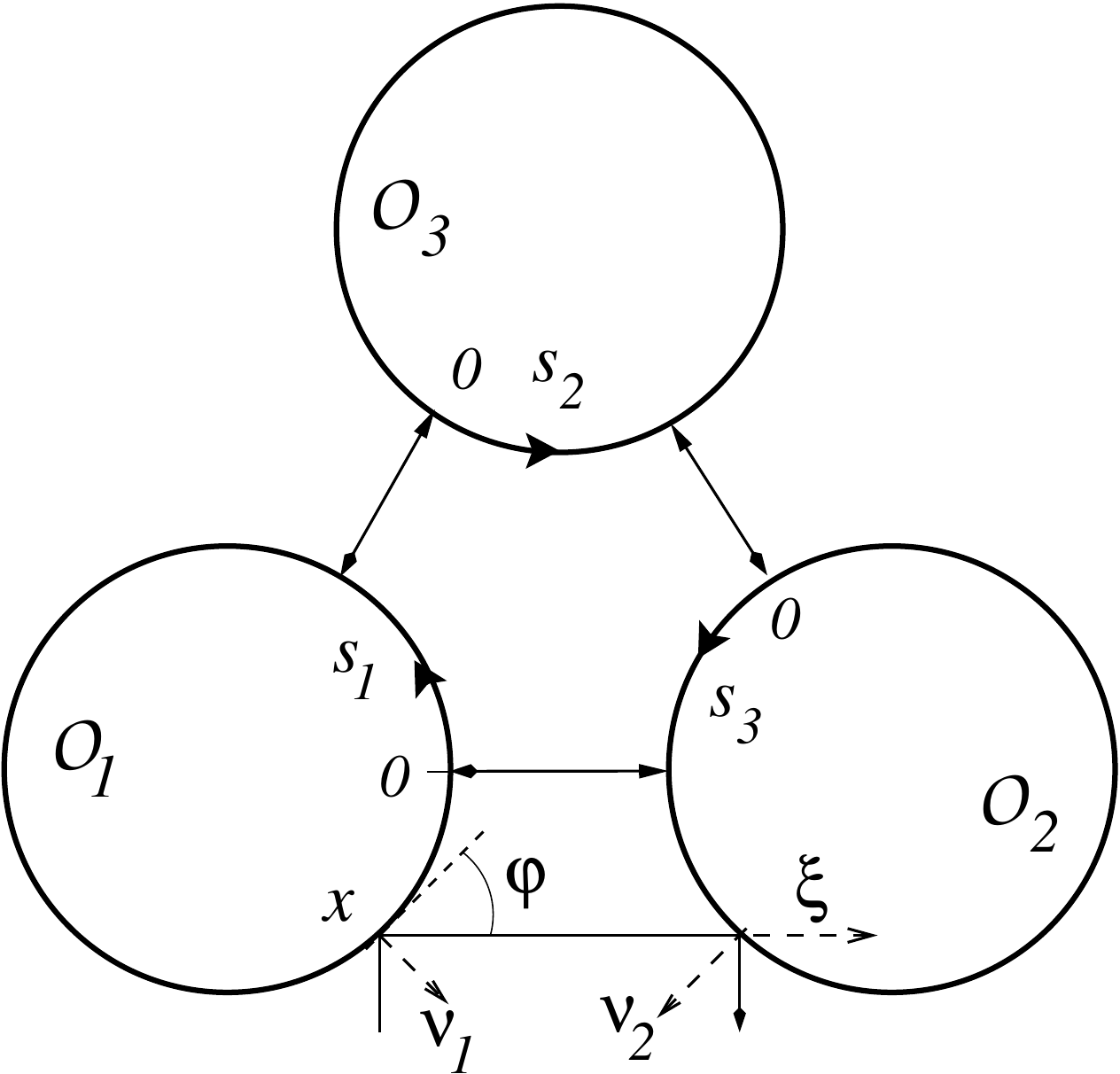}
\includegraphics[width=0.8\textwidth]{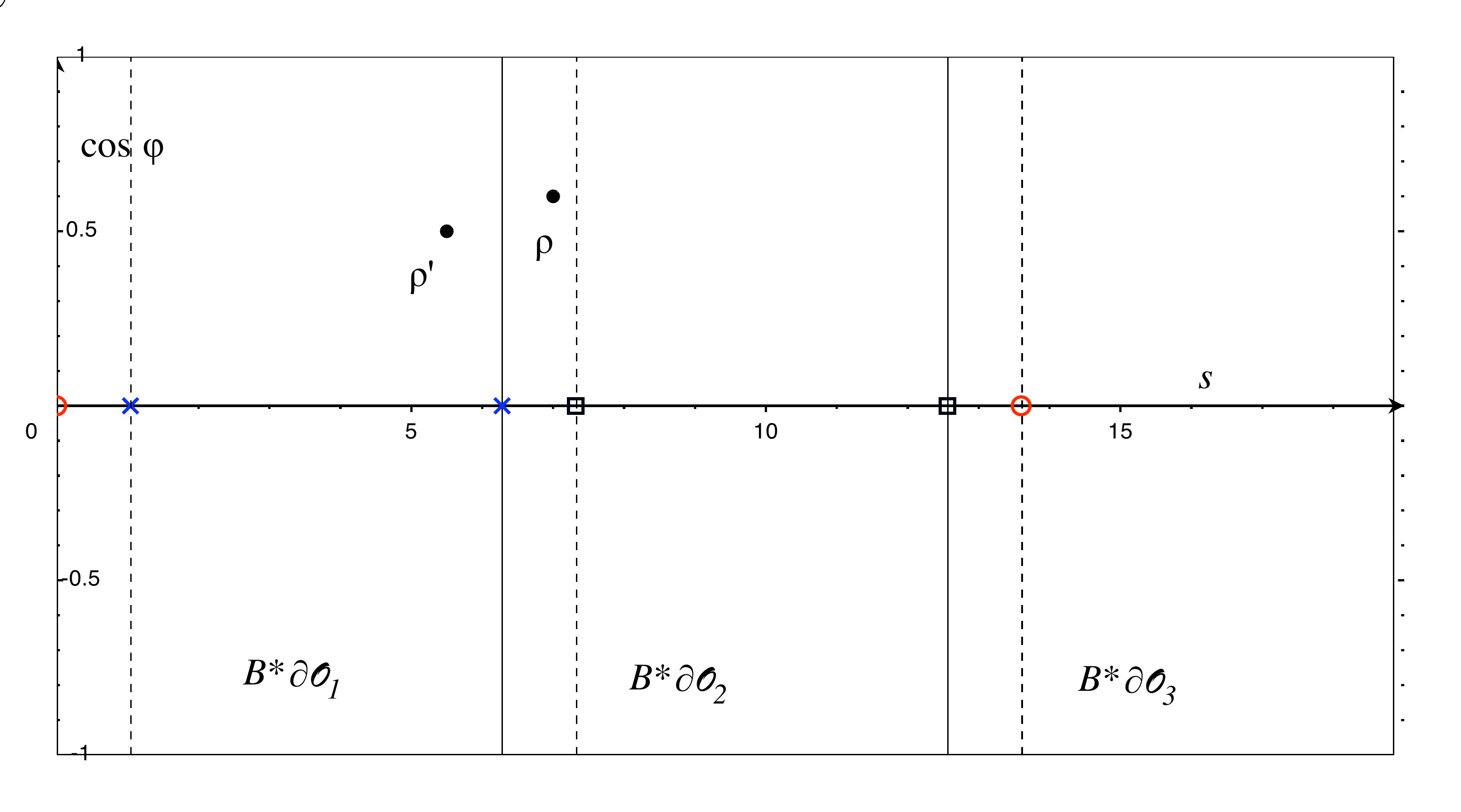}
\caption{\label{ott} (after \cite{Ott}) Reduction on the boundary
for the symmetric three disk scattering problem (the distance $D$
between the centers is $3$, the radius of the disks is $1$). 
Left: the trajectories hitting the obstacles can be
parametrized by position of the impact along the circles $\partial \Oo_i$,
(length coordinate $s_i\in[0,2\pi)$) and the angle between the velocity after impact and 
the tangent to the circle (momentum coordinate $\cos \varphi \in
[-1,1]$). We show three short periodic orbits, and a transient orbit.
Right: reduced phase space
$B^*\partial\Oo=\sqcup_{i=1}^3B^*\partial\Oo_i$ of the obstacles
(we concatenate the three length coordinates into $s\in [0,6\pi)$).
Each of the 3 periodic orbits is represented by 2 points on the
horizontal axis (circles, crosses, squares), while the transient orbit
is represented by the successive points $\rho$,
$\rho'=F(\rho)$. Vertical lines delimit the partial phase space used in Fig.~\ref{GGK}.}
\end{center}
\end{figure}

At the classical level, the reduction of the flow to the boundary of the
obstacles proceeds using the standard construction of the billiard
map  on the reduced phase space 
\[ 
B^* \partial \Oo =
\bigsqcup_{j=1}^J B^* \partial \Oo_j \,,
\]
where $ B^* \partial \Oo_k $ is the co-ball bundle over 
the boundary of $ \Oo_k $ defined using the induced Euclidean 
metric (see Fig.~\ref{ott} in the two dimensional case). 
Strictly speaking, the billiard map is not defined on the whole
reduced phase space (e.g. in Fig.~\ref{ott} it is not defined on the
point $\rho'$); one can describe it as a 
symplectic relation $F \subset B^* \partial \Oo \times B^* \partial \Oo$,
union of the relations $ F_{ij} \subset B^* \partial \Oo_i \times 
B^* \partial \Oo_j $, $ i \neq j $ encoding the trajectory segments going from
$\Oo_j$ to $\Oo_i$:
\begin{gather}
\label{eq:bill}
\begin{gathered}
 ( \rho' , \rho) \in F_{ij} \subset B^* \partial \Oo_i \times B^* \partial \Oo_j 
 \\ \Longleftrightarrow \\ \exists \; t > 0 \,, \xi \in \SP^{n-1} \,, \
x  \in {\partial \Oo_j} \,, \  x + t \xi  \in 
{\partial \Oo_i}  \,,  \ \langle \nu_j ( x ) , \xi \rangle > 0 \,, \\ 
 \langle \nu_i ( x + t \xi ) , \xi \rangle <  0 \,, \
\pi_j ( x, \xi ) = \rho \,, \ \ \pi_i ( x + t  \xi , \xi) = \rho' \,.
\end{gathered}
\end{gather}
(here $ \pi_k: S_{\partial \Oo_k} ^*( \RR^n) 
\rightarrow B^* { \partial {\Oo_k}} $ is the natural orthogonal
projection.)

To the relation $ F $ we can associate various {\em trapped} sets:
\begin{equation}
\label{eq:defT}
  \cT_\pm \defeq \bigcap_{k=0}^\infty F^{\pm k} ( B^*\partial\Oo) \,, 
\qquad {\mathcal T} \defeq \cT_+ \cap \cT_- \,.
\end{equation}
Notice that $\cT$ is directly connected with the trapped set $K$ for
the scattering flow, defined in the beginning of this introduction:
\[  \cT \cap B^* \partial \Oo_j = \pi_j ( K \cap S^*_{\partial \Oo_j} 
( \RR^n ) ) \,.
\]
The sets $ \cT_\pm $ and $\cT $ for the 2D scattering problem of
Fig.~\ref{ott} are shown in
Fig.~\ref{GGK}.

\begin{figure}
\begin{center}
\includegraphics[scale=0.34]{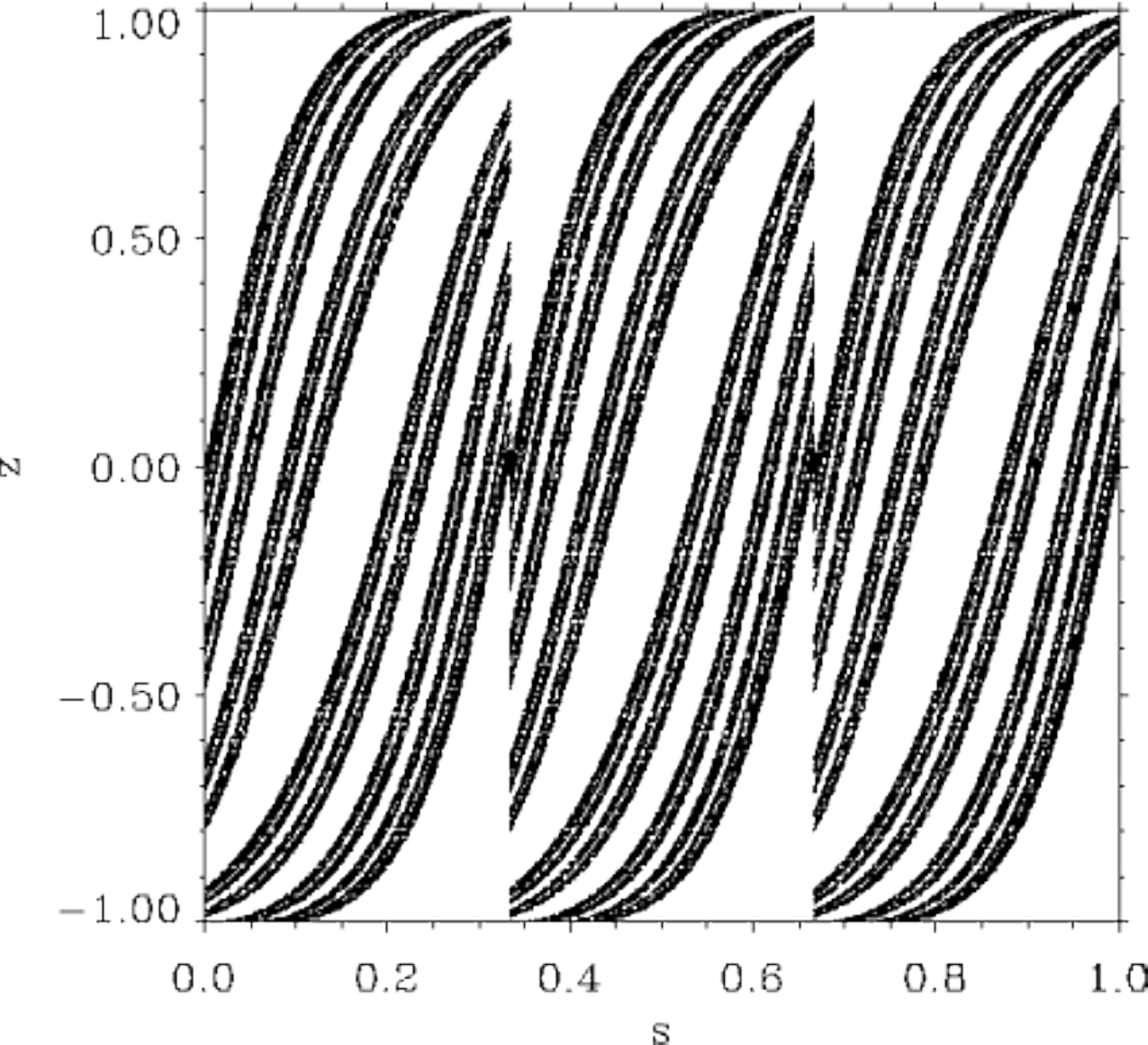} \hspace{0.05in}   \includegraphics[scale=0.34]{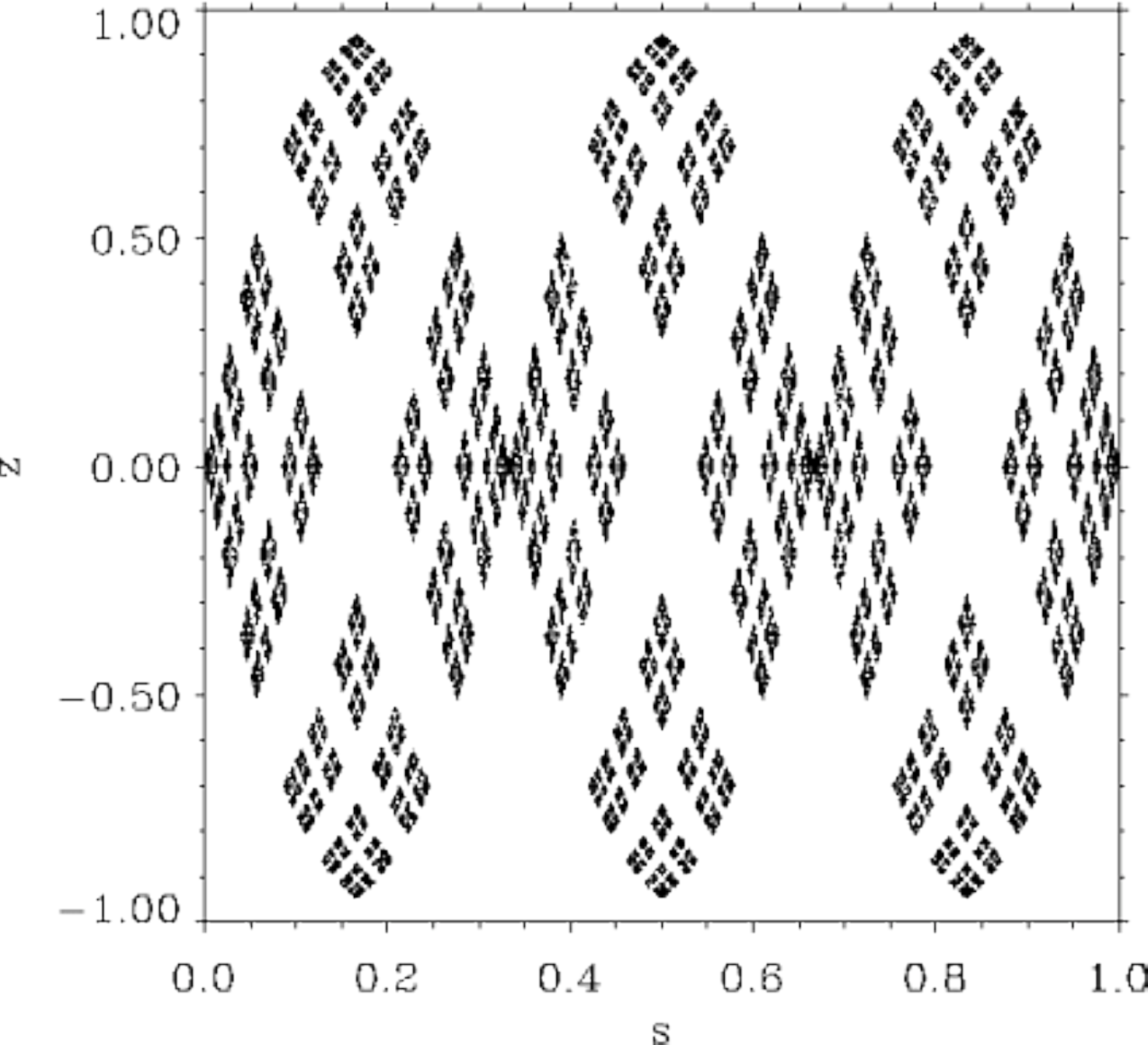}  \hspace{0.05in}    \includegraphics[scale=0.34]{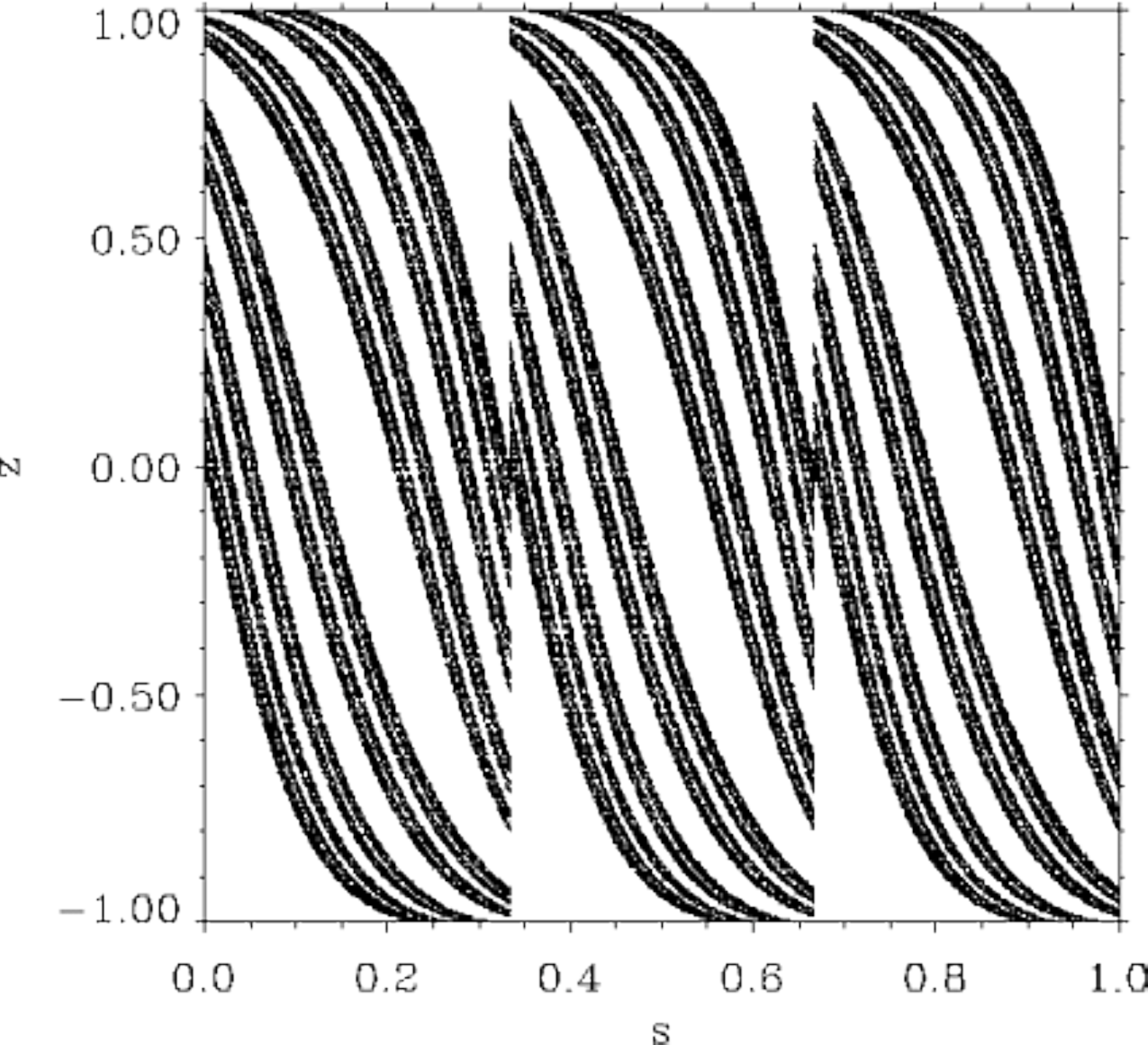}
\caption{\label{GGK} The unstable manifold, $ \TT_+ $, 
the trapped set, $ \TT $, and the stable manifold $ \TT_- $
and for the system of Fig.~\ref{ott}. 
Compared with the notations of Fig.~\ref{ott}, the horizontal
axis is given by the union of three segments of $\partial\cO_j$, parametrized by a
single coordinate $s=s_i/\pi+(i-1)/3$, $s_i\in[0,\pi/3]$, $i=1,2,3$
(this partial phase space is delimited by the vertical lines in
Fig.~\ref{ott}, right).
The sets $\TT_\pm$ are smooth along the unstable/stable
direction and fractal transversely, while $ \TT $ has the
structure of the
product of two Cantor sets. The left strip corresponds to the one
indicated in Fig.~\ref{ott}. (the figures are from \cite{Ott}).}
\end{center}
\end{figure}
The strict convexity of the obstacles entails that the
trapped rays are uniformly unstable: in the dynamical systems
terminology, $\cT$ is an invariant hyperbolic set for the relation $F$
(see \S\ref{2gr}). 

The boundary of $ B^* \partial \Oo_k $, which consists of covectors of length one, 
corresponds to rays which are tangent to $ \partial \Oo_k $ (the 
glancing rays), and these
produce complicated effects in the operators $ H_j ( z )$. However,
Ikawa's condition \eqref{eq:Ik} guarantees that none of these
rays belongs to the trapped set:
\begin{equation}\label{e:trapped-no-glancing}
\cT \cap S^* \partial \Oo = \emptyset\,.
\end{equation} 

A the quantum level, the operator $\cM(z,h)$ has the properties of a
Fourier integral operator associated with the billiard map away from
the glancing rays, but its structure near these rays is more
complicated due to diffraction effects. Fortunately, the property
\eqref{e:trapped-no-glancing} implies that, as far as the poles of
$(I-\cM(z,h))^{-1}$ are concerned, these annoying glancing rays
are irrelevant. Indeed, 
we will show in \S\ref{assc} that the operator $\cM(z,h)$ can be replaced
by a truncated operator of the form
\begin{equation}
\label{eq:Mzh}
  M ( z, h  ) = \Pi_h  e^{-g^w}\, \cM ( z, h ) \,e^{g^w}\,\Pi_h 
+ \Oo_{L^2 ( \partial \Oo) \rightarrow L^2 ( \partial \Oo) }  ( h^N)\big) \Pi_h 
 \,,
\end{equation}
where $g(x,\xi)$ is an appropriate escape function, and $ \Pi_h$ is
an orthogonal projector of finite rank comparable with $h^{1-n}$,
microlocally corresponding to a compact neighbourhood of $\cT$ not containing the glancing rays.
The operator $M ( z, h  )$ is thus ``nice'': it is of the same type as the {\em hyperbolic quantum monodromy
operators} constructed in \cite{NSZ1} to study semiclassical
scattering problems such as $-h^2\Delta + V(x)$, with $V\in
C_c^\infty(\IR^n)$. The study of the poles of $(I-M(z,h))^{-1}$ can
then be pursued in parallel for both types of problems (see \S\ref{gp}).


\medskip

For simplicity let us assume that the trapped set $\cT$ is of pure
dimension $\dim {\mathcal T } =2\mu$.
The fractal Weyl upper bound \eqref{eq:t1} corresponds to the 
statement that, for any fixed $r>0$, there exists $C_r>0$ such that
\begin{equation}
\label{eq:arg}
-  \frac 1 { 2 \pi i } \tr \int_{ \{|z|=r\} } 
( I - M ( z , h ) )^{-1}\, \partial_z M ( z , h )\, dz \leq C_r\, h^{ - \mu }  \,. 
\end{equation}
To prove it we want to further modify the monodromy operator $ M ( z ,
h ) $, and replace it by 
\be\label{e:tilde-M}
\widetilde M ( z , h ) = \Pi_W \big( e^{-G^w} M ( z , h ) e^{G^w} + \Oo (\epsilon ) \big)  \Pi_W 
\,, \qquad  0 < \epsilon \ll 1\,.
\ee
Here $G$ is a finer escape function, and $ \widetilde \Pi_W $ is a
finer finite rank projector,
now associated with a much thinner neighbourhood of $\cT$,
of diameter $ \sim h^{\frac12} $, and therefore of volume comparable to 
$ h^{\frac 12 ( 2(n-1) - 2 \mu )} $ (see the definition of the dimension
in \eqref{eq:dimm} below). 
This projector, and the escape function $G$, will be constructed in
\S\ref{gp} using symbols in an exotic pseudodifferential class, and the
associated symbol calculus.
The rank of an operator microlocalized to a set in $T^*\RR^{n-1}$ of volume
$ v $
is estimated using the uncertainty principle by $  \Oo ( h^{-n+1}
v ) $. Hence, when $ v \sim h^{\frac 12 ( 2(n-1) - 2 \mu )} $ we obtain
the bound $ \Oo ( h^{- \mu } ) $ for the rank of $\Pi_W$. This in turn shows that 
\[ 
\log |\det ( I -  \widetilde M ( z , h ) )| =\Oo ( h^{-\mu} ) \,. 
\]
We also show that at some fixed point
$z_0$, the above expression is bounded from {\it below} by $-Ch^{-\mu}$. Jensen's inequality then 
leads to \eqref{eq:arg}.

In the case where $\cT$ is not of pure dimension, one just needs to replace
$\mu$ by $\mu\!+\!0$ in the above estimates.

\medskip

The paper is organized as follows. In \S \ref{2gr} we give a general
definition of hyperbolic open quantum maps and quantum monodromy
operators. This definition will be compatible with
both the case of several convex obstacles, and the case of semiclassical potential
scattering, after the reduction of \cite{NSZ1}. 
The technical preliminaries are given in \S \ref{pr}
where various facts about semiclassical microlocal 
analysis are presented. In particular, we investigate the properties
of exotic symbol classes, which will be necessary to 
analyze the weight (escape function) $G$ used to conjugate the
monodromy operator in \eqref{e:tilde-M}, and to construct the projection
$\Pi_W$. The weight $G$
is constructed in \S \ref{cef}, using our dynamical assumptions.
In \S\ref{gp} we construct the projection $\Pi_W$, and prove the
resulting fractal upper bounds on the number of resonances in the
general framework of hyperbolic quantum monodromy operators.
Finally, \S \ref{assc} is devoted to a detailed study of the
obstacle scattering problem: we show that the reduction to the
boundary operator $\cM(z)$, combined
with propagation of singularities, leads to a hyperbolic quantum monodromy
operator \eqref{eq:Mzh}, to which the results of the earlier sections can be applied.

\section{Hyperbolic open quantum maps}
\label{2gr}

In this section we start from a hyperbolic open map, and quantize it
into an open quantum map. The quantum monodromy operators constructed
in \cite{NSZ1} have the same structure, but also holomorphically
depend on a complex parameter $z$.

Let $ Y_j \Subset \RR^d $, $ j =1 , \cdots J $,
be open contractible sets, and let 
\[  Y \defeq  \bigsqcup_{j=1}^J Y_j \subset \bigsqcup_{j=1}^J \RR^d 
\,, \]
be their disjoint union. We also define a local phase space,
\[ \UU \defeq 
\bigsqcup_{j=1}^J U_j  \subset \bigsqcup_{j=1}^J T^* \RR^d 
\,, \ \ U_j\quad \text{open},\quad U_j \Subset T^* Y_j \,.  
\]
Let $ \TT \Subset \UU $ be a compact subset  and suppose that 
\[  f \; : \; \TT \longrightarrow \TT  \]
is an invertible transformation which satisfies
\be\label{e:extension}
f \subset F \subset \UU \times \UU \,, \quad \text{($f$ is identified
  with its graph),}
\ee
where $ F $ is 
a {\em smooth Lagrangian relation} with boundary, and 
$ \partial F \cap f = \emptyset $. We assume that $ F $ is locally 
the restriction of a smooth symplectomorphism (see below for a more precise 
statement). 
In particular, $ F $ is at most single valued:
\[ ( \rho', \rho ) , ( \rho'' , \rho ) \in F  \; \Longrightarrow \;
\rho' = \rho'' \,, \]
and similarly for $ F^{-1} $. By a slight abuse of notation, we will
sometimes replace the graph notation,  $(\rho',\rho)\in F$, by the
map notation, $\rho'=F(\rho)$.
As a result of \eqref{e:extension}, for
any $ \rho \in \TT $ the tangent map
$ df_\rho : 
T_\rho \, \UU \longrightarrow T_{f(\rho)} \, \UU $ is well defined,
and so is its inverse. 

Such a relation $F$ can be considered as an 
{\em open canonical transformation} on $\UU$. Here, open means that the
map $F$ ($F^{-1}$) is only defined on a subset
$\pi_R(F)$ ($\pi_L(F)$, respectively)
of $\UU$; the complement $\UU\setminus \pi_R(F)$ can be thought of as
a hole through which particles escape to infinity.

We now make a stringent dynamical assumption on the transformation
$f$, by assuming it to be {\em hyperbolic}
(equivalently, one says that $\cT$ is a hyperbolic set for $F$).
This means that, at every $ \rho \in \TT $,
the tangent space $ T_\rho \, \UU $ decomposes as 
\be\label{eq:aa}
\begin{split}
i)\ & T_\rho \, \UU =  E^+ _ \rho \oplus
E^- _ \rho \,, \quad  \dim E^\pm _ \rho  = d  \,, \quad\text{with the properties}\\
ii)\ &  d f_\rho ( E^\pm_\rho ) = E^\pm _{ f ( \rho ) }\,, \\
iii)\ & \exists \; 0 < \theta < 1 \,,
\   \forall \,  v \in E^\mp_ \rho  \,, \ \forall n \geq 0 \,,
\quad  \| df^{\pm n}_\rho ( v ) \| \leq 
C \theta^n \| v \| \,.
\end{split}
\ee
This decomposition is assumed to be continuous in $ \rho $, and the
parameters $\theta,C$ can be chosen independent of $\rho$. It is then
a standard fact that \cite[\S 6.4,\S 19.1]{KaHa}
\be
\label{eq:aaa}
\begin{split}
iv)\ & \TT  \ni \rho \longmapsto E^{\pm}_ \rho  \subset
T_\rho (\UU ) \
\text{ is H\"older-continuous} \\
v)\ &\text{any }\rho\in \TT \
\text{admits local stable($-$)/unstable($+$) manifolds } W^{\pm}_{\loc}(\rho), \\
\ & \text{ tangent to } E^\pm_\rho
\end{split}
\ee

Let us now describe the relation $ F $ more precisely.
It is a disjoint union of symplectomorphisms
\[  
F_{ik}
 \; : \; \tD_{ ik}\subset U_{k} \longrightarrow \; 
F_{ik} ( \tD_{ik } ) 
=\tA_{ik} 
\subset U_i \,, \]
where $ \tA_{ik} \subset U_i $ and $ \tD_{ik } 
\subset U_k $ are open neighbourhoods of the arrival and departure 
subsets of $ \TT $,
\[ \begin{split}
&  A_{ik } \defeq \{ \rho \in \TT_i \; : \;  f^{-1} ( \rho ) \in \TT_k \} 
= \TT_i \cap f ( \TT_k ) \,, \ \ \TT_j \defeq \TT \cap U_j \,, \\ 
& D_{ ik} \defeq \{ \rho \in \TT_k \; : \; f( \rho ) \in \TT_i \} = 
\TT_k \cap f ^{-1} (\TT_i ) \,.  \end{split} \]
(see Fig.~\ref{f:3disks-AD} for a plot of the sets $\tA_{ik}$ and $\tD_{ik}$ for
the scattering by three disks, and Fig.~\ref{GGK}, center, for the
trapped set). We also 
write
\[    \tD_{k} \defeq \bigcup_j \tD_{jk} \,, \ \
\tD \defeq \bigsqcup_{k} \tD_k \,, \]
and similarly for the other sets above. Notice that we have $\tD=\pi_R(F)$, $\tA=\pi_L(F)$.

\medskip

On the quantum level we associate to $ F$ {\em hyperbolic open quantum
maps}  defined as follows:

\begin{defn}\label{def:HQMO}
{\em A {\rm hyperbolic open quantum map} $ M=M(h) $, is an $h$-Fourier 
integral operator quantizing a smooth Lagrangian relation $F$ of the
type described above,
\[  
M \in I_{0+} ( Y \times Y , F' ) \,, 
\]
(here $F'$ is the twisting of the relation $F$ so that $ F' $ becomes
Lagrangian in $ T^* (Y \times Y ) $ -- see \S \ref{fio} for the definition of this class of
operators). In particular, $M$ is microlocalized in the interior
of $F'$: its semiclassical wavefront set (see \eqref{e:WFh}) satisfies
\[ 
\WFh ( M ) \cap \partial F' = \emptyset \,. 
\]
We also assume that there exists
some  $a_M\in \CIc ( T^*Y ) $, with $ \supp a_M  $ contained 
in a compact neighbourhood $\cW$ of $ \TT $, $\cW\Subset \pi_R(F)$, and $a_M(\rho)\equiv 1$ in a
smaller neighbourhood $\cW'$ of $\TT$, such that
\begin{equation}
\label{eq:Mlo}  M ( I - A_M ) = \Oo_{ L^2 \to L^2 }
( h^{N_0} ) \,,\qquad A_M = \Op(a_M)\,,
\end{equation}
with $N_0\gg 1$,  independent of $h$. This means that $M(h)$ is very
small,  microlocally outside $\cW$. 

Assume $\Pi_h$ is an orthogonal projector of finite dimension comparable with
$h^{-d}$, with $\Pi_h$ equal to the identity
microlocally near $\cW$.
We will then also call open quantum map (or truncated open quantum
map) the finite rank operator 
$$
\tM(h)=\Pi_h \,M(h)\,\Pi_h = M(h) + \Oo_{L^2\to L^2}(h^{N_0})\,.
$$
A {\rm hyperbolic quantum monodromy operator} is a family of hyperbolic open quantum
maps $\{M(z,h)\}$ (or their finite rank version $\tM(z,h)$) associated
with the same relation $F$, which depend holomorphically on
$z$ in 
\begin{equation}
\label{eq:omegah} 
\Omega=\Omega(h) \defeq [ - R ( h ) , R_1 ] + i [ -R_1 , R_1 ] \,, \ \
R ( h ) \stackrel{h \rightarrow 0 }{ \longrightarrow } \infty \,, \ \
R_1 > 0 \,, 
\end{equation}
as operators $L^2\to
L^2$.
Furthermore, there
exists a decay rate $\tau_M>0$ such that
\be\label{e:decayM}
\|M(z,h)\|\leq C\,e^{\tau_M\Re z}\,,\quad h<h_0,\ z\in\Omega(h) \,. 
\ee
We will also consider truncated monodromy operator $\tM(z,h)$.
The cutoff $A_M$, the projector $\Pi_h$ and the estimates
\eqref{eq:Mlo} are assumed uniform with respect to $z\in \Omega(h)$.}
\end{defn}



This long definition is taylored to include the monodromy operator $\tM(z,h)$ constructed for open
hyperbolic flows with topologically one dimensional trapped sets,
through a Grushin reduction --- see \cite{NSZ1}. That construction shows that the
exponent $ N_0$ in \eqref{eq:Mlo} can be taken arbitrary large, and
that we can take domain in \eqref{eq:omegah} with 
 $ R (h) = C \log(1/h)$, for some $C>0$, and $ R_1 $ large but fixed.

For the
scattering by several convex bodies described in the 
introduction, we will be able (using propagation of singularities) 
to transform the boundary operator
$\cM(z,h)$ of \eqref{eq:MM} 
into a monodromy operator $\tM(z,h)$ of the form above, with arbitrary $ N_0 $, 
see \S \ref{assc}. 


\section{Preliminaries}
\label{pr}

The general preliminary material and notation for this paper
are the same as in \cite[\S 2]{NSZ1}.
We specifically present properties of exotic symbols and weights necessary to
construct the escape function $G$ and the projector $\Pi_W$ in
\eqref{e:tilde-M}, and analyze their interaction with Fourier integral
operators. Some of the material is taken 
directly from 
\cite[\S 2]{NSZ1}, \cite[\S 3]{SZ10}, and some developed specifically for our needs.

\subsection{Semiclassical pseudodifferential calculus}
\label{sa}

We recall the following class of symbols on $T^*\RR^d$ (here
$m,k\in\RR$, $\delta\in [0,1/2]$):
\[ \begin{split} 
S^{m,k}_\delta ( T^* \RR^d ) =\Big\{ & a \in \CI( T^* \RR^d \times (0, 1]
  ) :\ \forall \alpha,\beta\in\NN^n,\\
& \ |\partial_x ^{ \alpha } \partial _\xi^\beta a ( x, \xi ;h ) | \leq
C_\alpha h^{-m-\delta ( | \alpha| + |\beta |) }
\langle \xi \rangle^{k-|\beta| } \Big\}  \,, 
\end{split} \]
where we use the standard notation  $\la \bullet\ra =
(1+\bullet^2)^{1/2}$.

The Weyl quantization $a^w(x,hD)$ of such a symbol is defined as
follows: for any wavefunction $u$ in the Schwartz space $\cS(\RR^d)$, 
\be\label{eq:weyl}
\begin{split}
a^w u(x) & = a^w(x,hD) u(x)  =  [\Op (a) u ] ( x) \\
& \defeq \frac1{ ( 2 \pi h )^d } 
  \int \int  a \Big( \frac{x + y  }{2}  , \xi \Big) 
e^{ i \la x -  y, \xi \ra / h } u ( y ) dy d \xi \,, 
\end{split} 
\ee
see \cite[Chapter 7]{DiSj} for a detailed discussion of 
semiclassical quantization, and \cite[Appendix]{SjR}, 
\cite[Appendix D.2]{EZB} for the semiclassical calculus for
the symbol classes given above, and its implementation on 
manifolds. When $ \delta = 0 $ or $ m  = k = 0 $ we will generally
omit to indicate those indices.
We denote by
$ \Psi_{\delta}^{m,k} ( \RR^d) $, $\Psi^{m,k} ( \RR^d )$, or 
$ \Psi ( \RR^d ) $  
the corresponding  classes of pseudodifferential operators.

For a given symbol $ a \in S ( T^* \RR^d ) $ we follow \cite{SZ8} and 
say that its {\em essential support} is contained a given compact
set $ K \Subset T^*\RR^d $, 
$$
  {\text{ess-supp}}_h\; a \subset K \Subset T^*\RR^d\,,  
$$
if and only if   
$$
 \forall \, \chi \in S ( T^*\RR^d ) \,, \ \supp \chi \cap K=\emptyset 
\ \Longrightarrow
 \ \chi \, a \in h^\infty \cS ( T^* \RR^d) \,.
$$
Here $\cS$ denotes the Schwartz space. The essential support of $a$ is then the intersection of all such $K$'s.

For $ A \in \Psi ( \RR^d) $,  $  A = \Op ( a ) $, we call
\begin{equation}
\label{eq:WFA}
    \WFh ( A) =  \text{ess-supp}_h\; a 
\end{equation}
the semiclassical wavefront set of $A$. 
(In this paper we are concerned with a purely semiclassical 
theory and will only need to deal with {\em compact} subsets of $ T^* \RR^d $.)

Let $ u=u(h) $, $ \| u ( h ) \|_{ L^2} = {\mathcal O} ( h^{-N} ) $
(for some fixed $ N$) be a wavefunction microlocalized in 
a compact set in $ T^* \RR^d $, in the sense that for some
$ \chi \in \CIc ( T^* \RR^d ) $, one has $ u = \chi^w u + 
\Oo_{\cS} ( h^\infty ) $. The semiclassical wavefront set of $u$ is
then defined as:
\be\label{eq:defWF}
\WFh ( u ) =   \complement \big\{  ( x, \xi )\in T^*\RR^d
\; : \; \exists \, a \in S ( T^* \RR^d ) \,, \  \ a ( x, \xi ) =1
\,, \ \| a^w \, u \|_{L^2} = \Oo ( h^\infty) \big\}
 \,.
\ee

For future reference we record the following simple consequence of
this definition:

\begin{lem}
\label{l:WFr}
If $ u ( h ) = {\mathcal O}_{L^2 ( \RR^n ) } ( h^{-N} ) $ is a
wavefunction microlocalized in a compact subset of $ T^* \RR^n $ 
and 
\[  v ( h ) ( x' ) \defeq u ( h ) ( 0 , x' ) \,, \ \ ( x_1 , x' )  \in
\RR^n \]
then $ v ( h ) = {\mathcal O}_{L^2 ( \RR^n ) } ( h^{-N-1/2} ) $ and 
$ v ( h ) $ is microlocalized in a compact subset of $ T^* \RR^{n-1}
$.
In addition we have 
\be
\label{eq:WFr} \WFh ( v ) \subset \{ ( x', \xi' ) \in T^* \RR^{n-1} \; : \; 
\exists \, \xi_1 \in \RR \,, \ \ ( 0 , x', \xi_1, \xi' ) \in \WFh ( u )
\} \,. \ee
\end{lem}

We point out that unlike in the case of classical wave front 
sets we do not have to make the assumption that $ \xi' \neq 0 $
when $ ( x',0,\xi_1, \xi' ) \in \WFh ( u ) $. 

\begin{proof}
Let $ \chi \in \CIc ( T^* \RR^n ) $ such that $ u = \chi^w u +
{\mathcal O}_{\mathcal S} ( h^\infty ) $ -- it exists by the
assumption 
that $ \chi $ is microlocalized in a compact set.

By choosing $ \psi \in \CI ( \RR ) $ such that $ \psi ( \xi_1 ) = 1
$ for $ ( x, \xi ) \in \supp \chi $ we  have 
$ \psi ( h D_{x_1} ) u = u + {\mathcal O}_{\mathcal S} ( h^\infty )
$ and we can simply replace $ u $ by $ \psi ( h D_{x_1} ) $.
Then 
\[ \begin{split}  | v (x') |^2 &  = \frac{ 1 } {\sqrt { 2 \pi h } }
\left|    \int_\RR \psi (
\xi_1  ) ({\mathcal F}_h )_{x_1 \mapsto \xi_1 } u ( \xi_1 , x' )
d\xi_1  \right|^2 
 \leq \frac{C_\psi }{ \sqrt h } \| 
 ({\mathcal F}_h )_{x_1 \mapsto \xi_1 } u ( \bullet , x' ) \|^2_{L^2(   \RR) } \,,
\end{split} \]
where $ {\mathcal F}_h $ is the unitary semiclassical Fourier transform 
(see \cite[Chapter 2]{EZB}).
Integrating in $ x'$ gives the bound $ \| v \|_{L^2 ( \RR^{n-1} ) } =
{\mathcal O} ( h^{-N-1/2} ) $. 

Similar arguments prove the remaining statements in the lemma.
\end{proof}

Semiclassical Sobolev spaces, $ H^s_h ( X) $ are defined using 
the norms
\begin{gather}
\label{eq:ssn}
\begin{gathered}
\| u \|_{H_s ( \RR^n ) } = \| ( \Id - h^2 \Delta_{\RR^n} )^{s/2} u
\|_{L^2 ( \RR^n) } \,,  \ \ X = \RR^n \,, \\
\| u \|_{H_s (X ) } = \| ( \Id - h^2 \Delta_{g} )^{s/2} u
\|_{L^2 ( X) } \,,  \ \ X \ \text{ a compact  manifold,} 
\end{gathered}
\end{gather}
for any choice of Riemmanian metric $ g $.

\subsection{$S_{\frac12} $ spaces with two parameters}
\label{s2p}

We now refine the symbol classes $S^{m , k }_{\frac12}$, by
introducing a second small parameter, $\tilde h\in(0,1]$, independent
of $h$.
Following \cite[\S 3.3]{SZ10} we define the symbol classes:
\begin{equation}
\label{eq:sthf}
a \in S^{m , {\widetilde m}, k }_{\frac12} ( T^* \RR^d ) 
\ \Longleftrightarrow \ |\partial_x^\alpha \partial^\beta_\xi 
a (  x, \xi ) | \leq C_{\alpha \beta} h^{-m } \tilde h^{-{\widetilde m}} 
\Big( \frac{\tilde h }{ h } 
\Big)^{\frac12 ( |\alpha | + |\beta| ) } \langle \xi \rangle^{k-|\beta|}
 \,,\end{equation}
where in the notation we suppress the dependence of $ a $ on $ h 
$ and $ \tilde h $.
When working on $ \RR^d $ or in fixed local coordinates, we will
use the simpler classes
\begin{equation}
\label{eq:sthfs}
a \in \widetilde{S} ( T^* \RR^d ) 
\ \Longleftrightarrow \ |\partial^\alpha
a | \leq C_{\alpha} \,,\qquad
a \in \ST ( T^* \RR^d ) 
\ \Longleftrightarrow \ |\partial^\alpha
a | \leq C_{\alpha} ( \th / h )^{\frac12  |\alpha | } \,. 
\end{equation}
We denote the corresponding classes of operators by $ \Psi_{\frac12}^{m , {\widetilde m}, k } 
(\RR^d ) $ or $ \PT $.

We recall \cite[Lemma 3.6]{SZ10} which provides explicit 
error estimates on remainders in the product formula: 
\begin{lem}
\label{l:Sjnew}
Suppose that 
$ a, b \in \ST $, 
and that $ c^w = a^w \circ b^w $. 
Then for any integer $N>0$ we expand
\begin{equation}
\label{eq:weylc}  c ( x, \xi) = \sum_{k=0}^N \frac{1}{k!} \left( 
\frac{i h}{2} \sigma ( D_x , D_\xi; D_y , D_\eta) \right)^k a ( x , \xi) 
b( y , \eta) \rest_{ x = y , \xi = \eta} + e_N ( x, \xi ) \,.
\end{equation}
The remainder $e_N$ is bounded as follows: for some integer $ M $ independent of $N$,
\begin{equation}
\label{eq:Sjnew1}
\begin{split}
& | \partial^{\alpha} e_N | \leq C_N h^{N+1}
\times \\
& \ \ 
\sum_{ \alpha_1 + \alpha_2 = \alpha} 
 \sup_{ T^*\RR^d \times T^* \RR^d } 
\sup_{ \beta\in \NN^{4d} \,,
|\beta | \leq M  }
\left|
( h^{\frac12} 
\partial_{( x, \xi; y , \eta) } )^{\beta } 
(i  \sigma ( D) / 2) ^{N+1} \partial^{\alpha_1} a ( x , \xi)  
\partial^{\alpha_2} b ( y, \eta ) 
\right| \,,
\end{split} 
\end{equation}
where 
$$ \sigma ( D) = 
 \sigma ( D_x , D_\xi; D_y, D_\eta ) 
\defeq \langle D_\xi , D_y \rangle - \langle D_\eta , D_x \rangle \,, $$
is the symplectic form on $ T^* \RR^d \times T^* \RR^d$. 
\end{lem}
Notice that, due to the growth of the derivatives of $a,b$, the expression 
\eqref{eq:weylc} is really an expansion in powers of $\th$ rather than
$h$.
On the other hand, if $ a \in 
\ST ( T^* \RR^d)  $ and $ b $ is in the more regular class $S ( T^* \RR^d ) $, then 
\begin{gather*}
  c ( x, \xi) =
 \sum_{k=0}^N \frac{1}{k!} \left( 
i h \sigma ( D_x , D_\xi; D_y , D_\eta) \right)^k a ( x , \xi) 
b( y , \eta) \rest_{ x = y , \xi = \eta} + {\mathcal O} 
( h^{\frac{N+1}2 } \tilde h^{\frac{N+1}2} ) 
\,. 
\end{gather*}
We also recall \cite[Lemma 3.5]{SZ10} which is an easy adaptation of
the semiclassical Beals's lemma --- see \cite[Chapter 8]{DiSj} and 
\cite[\S 9.5]{EZB}. Because of a small modification of the statement
we present the reduction to the standard case:
\begin{lem}
\label{l:Beals}
Suppose that $ A \; : \; {\mathcal S} ( \RR^d ) \rightarrow 
{\mathcal S}' ( \RR^d ) $.
Then $ A = \Op ( a ) $ with $ a \in \ST $ 
 if and only if, 
for 
any $N\geq 0$ and any sequence $ \{ \ell_j \}_{j=1}^N $ of smooth functions
which are linear {\em outside a compact subset of}
$ T^* \RR^d $, 
\begin{equation}
\label{eq:Beals}
\| \ad_{\Op (\ell_1)} \circ \cdots \ad_{\Op (\ell_
N ) } A u \|_{L^2 ( \RR^d ) }
\leq C h^{   N/2} \tilde h^{ N/2} \| u \|_{ L^2 ( \RR^d )} \,, 
\end{equation}
for any $ u \in {\mathcal S} ( \RR^d ) $.
\end{lem}
\begin{proof} 
We use the standard rescaling to eliminate $ h$:
$$  ( \tilde x , \tilde \xi ) = ( \tilde h / h )^{\frac12} ( x , \xi ) \,,$$
and implement it through the following unitary operator on $ L^2 ( \RR^d ) $:
\begin{equation}
\label{eq:resc}
 U_{ h / \th } u ( \tilde x ) = ( \th / h )^{\frac{d}{4}}
u ( ( h / \tilde h )^{\frac12}  \tilde x ) =  ( \th / h
)^{\frac{d}{4}}u(x) \,.
\end{equation}
One can easily check that
\[ \Op(a)= U_{ h / \th }^{-1} \Opt(\tilde a) U_{ h / \th} \,, 
\quad\text{where}\quad
\tilde a ( \tilde x , \tilde \xi ) = a (
( h / \tilde h)^{\frac12} ( \tilde x , \tilde \xi ) ) \,. \]
Notice that the symbol $\tilde a\in \widetilde{S}(T^*\RR^d)$ if $a\in \ST(T^*\RR^d)$.
In the rescaled coordinates, the condition \eqref{eq:Beals} concerns 
$\th$-pseudodifferential operators: it reads
\be\label{e:Beals-rescaled}
\| \ad_{ \Opt(\tilde \ell_1)} \circ \cdots \circ 
 \ad_{\Opt(\tilde \ell_N) } \Opt(\tilde a) u \|_{ L^2 } \leq 
C \tilde h ^N \| u \|_{ L^2 } \,. 
\ee
Let us prove the statement in the case where the
$\ell_j $'s are linear: then
\[  \tilde \ell_j = ( \tilde h / h )^{\frac12} 
\ell_j \]
are also linear, and Beals's lemma for $\th$-pseudodifferential
operators \cite[Prop. 8.3]{DiSj} states that \eqref{e:Beals-rescaled}
for any $N$ is equivalent with
$\tilde a\in \widetilde{S}(1)$.

We finally want to show that, if $\tilde a\in \ST$, then \eqref{e:Beals-rescaled} holds for $ \ell_j$'s
which are compactly supported. Actually,
for $ \ell_j \in S ( 1 ) $ we may use Lemma \ref{l:Sjnew} to see that 
\[  \ad_{ \Opt(\tilde \ell_j)} \Opt(\tilde a) = 
\Oo_{ L^2 \rightarrow L^2 } ( h^{\frac 12} 
\tilde h^{\frac 12} ) \,, \]
which implies \eqref{eq:Beals}.
\end{proof}

The rescaling \eqref{eq:resc} can be used to obtain analogues
of other standard results. Here is one which we will need below.
\begin{lem}
\label{l:stan1}
Suppose $ a \in \ST $ and that $ \sup_{T^*\RR^d} | a | > c > 0 $, with 
$ c $ independent of $ h $ and $ \tilde h $. 
Then 
\[  \| \Op(a)\|_{ L^2 \rightarrow L^2 } \leq \sup_{ T^* \RR^d } | a | 
 + \Oo ( \tilde h  ) \,. \]
\end{lem}
\begin{proof}
We first apply the rescaling \eqref{eq:resc}, so that $ \Op(a) $
is unitarily equivalent to $ \Opt(\tilde a)$,
where $ \tilde a \in \widetilde{S} ( T^*\RR^d ) $, 
and $ \sup | \tilde a | = \sup | a | $. We then note that
$$
\Opt(\tilde a^w)^*\, \Opt(\tilde a) 
= \Opt( |\tilde a|^2 )+ \Oo_{L^2 \rightarrow L^2 } ( \th ) 
\,, 
$$
and that the sharp G{\aa}rding inequality (see for instance 
\cite[7.12]{DiSj} or \cite[Theorem 4.21]{EZB}) shows that
\[ ( \sup |a|)^2 -   \Opt(\tilde a)^* \Opt(\tilde a) \geq - C \tilde h \,, \]
from which the lemma follows.
\end{proof}

\subsection{Exponentiation and quantization}
\label{eaq}

As in \cite{SZ10}, we will need to consider 
operators of the form $ \exp(G^w ( x, h D )) $, where $ G \in S^{0,0,-\infty}_{\frac 12}(T^*\RR^d) $. 
To understand the properties of the conjugated operators, 
\[ \exp( {-G^w ( x ,h D) } )\, P\, \exp( { G^w ( x, h D) } ) \,, \]
we will use 
a special case of a result of
Bony and Chemin \cite[Th\'eor\`eme
6.4]{BoCh} --- see \cite[Appendix]{SZ10} or \cite[\S 9.6]{EZB}

To state it we need to recall a more general class of 
pseudodifferential operators defined using {\em order functions}.
A function $ m:T^*\RR^d\to \RR_+$ is an order function in 
the sense of \cite{DiSj} iff for some $N\geq 0$ and $C>0$, we have
\begin{equation}
\label{eq:orderm}
\forall \rho,\rho'\in T^*\RR^d,\qquad  m (\rho ) \leq C m ( \rho' )
\langle  \rho - \rho' \rangle ^N \,. 
\end{equation}
The class of symbols  corresponding to $ m $, denoted by $ \tS ( m ) $, is defined as
\[  a \in \tS ( m ) \ \Longleftrightarrow \ |\partial^\alpha_x \partial_\xi^\beta
a ( x , \xi ) | \leq C_{\alpha \beta} \, m ( x , \xi ) \]
(in this notation $\tS(1)$ is the class of symbols we had called $\widetilde{S}(T^*\RR^d)$).
If $ m_1$ and $ m_2 $ are order functions in the sense of \eqref{eq:orderm},
and $ a_j \in S (m_j ) $ then (we put $ h=1 $ here),
\[ 
a_1^w ( x, D )\, a_2^w ( x , D ) = b^w ( x , D) \,, \ \ b \in \tS ( m_1
m_2 ) \,,
\]
with $ b $ given by the usual formula
\begin{gather}
\label{eq:usual}
\begin{gathered} \begin{split} b ( x , \xi ) & = a_1 \; \sharp \; a_2 ( x, \xi ) \\
& \defeq
\exp ( i \sigma ( D_{x^1}, D_{\xi^1} ; D_{x^2 } , D_{\xi^2} )/2  ) \,
a_1 ( x^1 , \xi^1 ) \, a_2 ( x^2 , \xi^2 ) \rest_{ x^1 = x^2 = x , 
\xi^1 = \xi^2 = \xi } \,. \end{split} \end{gathered}
\end{gather}
Note that here we do not have a small parameter $ h $, so $a_1\sharp
a_2$ cannot be expanded as a power series. The value
of the following proposition lies in the calculus based on order functions. 
A special case of \cite[Th\'eor\`eme 6.4]{BoCh}, see \cite[Appendix]{SZ10}, gives 
\begin{prop}
\label{p:bbc}
Let $m $ be an order function in the sense of \eqref{eq:orderm}, and
suppose that $ g \in \CI ( T^* \RR^n; \RR  )  $ satisfies, uniformly
in $(x,\xi)\in T^*\RR^d$,
\begin{equation}
\label{eq:bc1}
 g ( x , \xi ) - \log m ( x , \xi )  = \Oo ( 1 ) \,, \ \
\partial_x^\alpha \partial_\xi^\beta g ( x, \xi ) = \Oo ( 1 ) \,,
 \ \ |\alpha | + |\beta| \geq 1 \,.
\end{equation}
Then for any $t\in\RR$, 
\begin{equation}
\label{eq:bc2} 
\exp ( t g^w ( x , D ) ) = B_t^w ( x , D ) \,, \quad B_t \in \tS ( m^t )\,.
\end{equation}
Here $ \exp ( t g^w ( x , D) ) u $, $ u \in \cS(\RR^d) $, is constructed 
by solving 
\[ \partial_t u ( t ) = g^w ( x , D ) u ( t) \,, \quad   u ( 0 ) = u \,. \]
The estimates on $ B_t \in \tS( m^t ) $ depend 
{\em only} on the constants in \eqref{eq:bc1} and in \eqref{eq:orderm}.
In particular they are independent of the support of $ g $.
\end{prop}
Since $ m^t $ is the order function $ \exp ( t \log m ( x , \xi ) ) $, 
we can say that, on the level of order functions, quantization 
commutes with exponentiation.

\medskip

This proposition will be used after applying the rescaling
\eqref{eq:resc} to the above formalism.
For the class $ \ST $, the order functions are defined by demanding
that for some $N\geq 0$,
\be
\label{eq:ord12}
   m ( \rho ) \leq C \, m ( \rho' ) \Big\langle \frac{ \rho - \rho' }
{ ( h / \th )^{\frac12} } \Big\rangle^N \,.
\ee
The corresponding class is defined by 
$$
a\in \ST(m)\ \Longleftrightarrow \ |\partial^\alpha
a ( \rho) | \leq C_{\alpha } ( \th / h )^{|\alpha|/2}\, m ( \rho )\,.
$$ 
We will consider order functions satisfying
\begin{equation}
\label{eq:ord}  m \in \ST ( m ) \,, \quad \frac{1} m \in 
\ST \left( \frac 1 m \right) \,.
\end{equation}
This is equivalent to the fact that the function 
\be\label{eq:G=logm}
G(x,\xi)\defeq\log m(x,\xi)
\ee
satisfies
\begin{gather}
\label{eq:ord1}
\begin{gathered}
   \frac{ \exp {G ( \rho )} } { \exp {G ( \rho' )} } 
 \leq C \Big \langle \frac{ \rho - \rho' }
{ ( h/\tilde h)^{\frac12} } \Big\rangle^N \,, \quad
\partial^\alpha G = \Oo ( (h/\tilde h )^{-|\alpha|/2} ) \,, 
\ | \alpha| \geq 1 \,. \end{gathered}
\end{gather}
Using the rescaling \eqref{eq:resc},
we see that Proposition \ref{p:bbc} implies that 
\begin{gather}
\label{eq:imG}
\begin{gathered}
\exp ( G^w ( x, h D) ) = B^w ( x, h D) \,, \ \
B \in \ST ( m ) \,, \\
 a \in \ST ( m ) \ \Longleftrightarrow 
\ \Op(a) = e^{G^w ( x, h D ) }\, \Op ( a_0 ) \,, \ \ a_0 \in \ST \,. 
\end{gathered}
\end{gather}
For future reference we also note the following fact: for 
$ A \in \Psi_\delta(\RR^d) $,
\be
\label{eq:FF}  A - e^{- G^w ( x, h D )}\, A \, e^{ G^w ( x, h D ) }= 
h^{ \frac 12 ( 1 - 2 \delta ) } \tilde h^{\frac12}\, a^w_1 ( x, h D ) \,,
\ \ a_1 \in \ST \,. 
\ee
The following lemma will also be useful when applying these weights to Fourier
integral operators.

\begin{lem}\label{l:G-support}
Let $U\Subset T^*\RR^d$, and let $\chi\in\CIc(U)$. Take $G_1,\,G_2$
two weight functions as in \eqref{eq:G=logm}, such that 
\[ (G_1-G_2)\rest_{U}=0 \,.\]
 Then, 
\begin{gather*}
e^{G_1^w(x,hD)}\chi^w = e^{G_2^w(x,hD)}\chi^w +
\Oo_{\cS'\to\cS}(h^\infty)\,,\\
\chi^w\,e^{G_1^w(x,hD)} = \chi^w\,e^{G_2^w(x,hD)} +
\Oo_{\cS'\to\cS}(h^\infty)\,.
\end{gather*}
\end{lem}
\begin{proof} We just give the proof of the first identity, the second
  being very similar.
Let us differentiate the operator
$e^{-tG_2^w(x,hD)}\,e^{tG_1^w(x,hD)}\chi^w$:
$$
\frac{d}{dt}e^{-tG_2^w(x,hD)}\,e^{tG_1^w(x,hD)}\,\chi^w =
e^{-tG_2^w(x,hD)}(G_1^w - G_2^w)\,e^{tG_1^w(x,hD)}\,\chi^w
$$
For each $t\in [0,1]$, the operator
$e^{tG_1^w(x,hD)}\,\chi^w\in\PT(m_1^t)$ is bounded on $L^2$, and 
has its semiclassical wavefront set contained in $\supp\chi$. Due to the
support property of $G_1-G_2$, we get
$$
(G_1^w - G_2^w)\,e^{tG_1^w(x,hD)}\,\chi^w=\Oo_{\cS'\to \cS}(h^\infty)\,,
$$
and the same estimate holds once we apply $e^{-tG_2^w(x,hD)}$ on the
left. As a result, 
$$
e^{-G_2^w(x,hD)}\,e^{G_1^w(x,hD)}\,\chi^w -\chi^w=\Oo_{\cS'\to \cS}(h^\infty)\,,
$$
from which the statement follows.
\end{proof}

\subsection{Fourier integral operators}
\label{fio}

We now follow \cite[Chapter 10]{EZB},\cite{SZ8} and review some 
aspects of the theory of semiclassical
Fourier integral operators. Since we will deal with operators in 
$ \PT $, the rescaling to $\tilde h$-semiclassical calculus, as in the proof
of Lemma \ref{l:Beals}, involves dealing with large $ h$-dependent
sets. It is then convenient to have a global point view, which
involves suitable extensions of locally defined canonical transformations
and relations.

\subsubsection{Local symplectomorphisms}
We start with a simple fact 
about symplectomorphisms of open sets in $ T^* \RR^d $.
\begin{prop}
\label{p:cl}
Let $U_0$
and $ U_1 $ be 
open neighbourhoods of $(0,0)$ and
${  \kappa} : U_0 \to U_1 $  a symplectomorphism such  
that $\kappa (0,0) = (0,0)$.
Suppose also that $ U_0 $ is {\em star shaped} with respect to the 
origin, that is, if $ \rho \in U_0 $, then 
$ t \rho  \in U_0 $ for $ 0 \leq t \leq 1 $. 

Then there exists a continuous,
piecewise smooth family
\[
\{\kappa_t\}_{0 \le t \le 1}
\]
of symplectomorphisms ${  \kappa}_t: U_0 \to U_t = \kappa_t(U_0)$, such that
\begin{equation}
\label{eq:pcl}
 \begin{split}
& {\rm (i)} \ \ \kappa_t(0,0) = (0,0)\,, \quad 0 \le t \le 1\,, 
\\
& {\rm (ii)} \ \ \kappa_1 = \kappa \,,  
\quad \kappa_0 = {\rm{id}}_{U_0} \,, \\
& {\rm (iii)} \ \ \frac {d}{dt} {  \kappa}_t = ({  \kappa}_t)_* H_ 
{q_t} \,,
\qquad 0 \leq t \leq 1\,, 
\end{split}
\end{equation}
where $ \{ q_t \}_{0 \le t \le 1}$ is a continuous, piecewise smooth
family of $ \CI (U_0 ) $ functions.
\end{prop}
The last condition means that $\kappa_t$ is generated by the (time dependent)
Hamiltonian vector field $H_{q_t}$ associated with the Hamiltonian
$q_t$\footnote{This generation can also be expressed through the more
usual form $\frac{d\kappa_t}{dt}=H_{q'_t}$, where $q'_t=q_t\circ \kappa_t^{-1}$.}.

\medskip
\noindent
{\em Sketch of the proof:} Since the group of linear symplectic
transformations is connected, we only need to deform $\kappa$ to a linear
transformation. That is done by taking
$$ 
\tilde \kappa_t ( \rho ) \defeq \kappa ( t \rho ) / t \,,\quad t\in[0,1]\,,
$$
which requires the condition that
$ U_0 $ is star shaped. It satisfies $\tilde\kappa_0 = d\kappa(0,0)$.
\stopthm

As a consequence we have the possibility to globalize a locally defined symplectomorphism:
\begin{prop}
\label{p:cl1}
Let $U_0$ and $ U_1 $ be
open precompact sets in $ T^* \RR^d $ and let
${  \kappa} : U_0 \to U_1 $ be a symplectomorphism which 
extends to $ \widetilde U_0 \Supset U_0 $, an open 
star shaped set. Then $ \kappa $ extends to a symplectomorphism
\[ 
\tilde \kappa \; : \; T^* \RR^d \rightarrow T^* \RR^d \,, 
\]
which is equal to the identity outside of a compact
set. $\tilde\kappa$ 
can be deformed to the identity with $ q_t $'s
in (iii) of \eqref{eq:pcl} supported in a fixed compact set.
\end{prop}
\begin{proof}
Let $ \tilde \kappa^0 $ be the extension of 
$ \kappa $ to $ \widetilde U_0 $. 
By Proposition \ref{p:cl} we can deform $ \tilde \kappa^0 $ to the identity,
with $ ( d/dt ) \tilde \kappa^0_t = (\tilde \kappa_t^0)_* H_{\tilde
  q_t} $,  $\tilde q_t\in\CI(\tilde U_0)$.
If we replace $ \tilde q_t $ by $ q_t = \chi \tilde q_t $ where 
$ \chi \in \CIc( \widetilde U_0 ) $, 
and $ \chi = 1 $ in 
$ U_0 $, the family of symplectomorphisms
of $ T^* \RR^d $ generated by $(q_t)_{0\leq t\leq 1}$ satisfies
\[  \kappa_t \rest_{U_0} = \tilde \kappa^0_t \rest_{U_0} \,, \ \ 
\kappa_0 = {\rm{id}} \,, \ \ \kappa_1 \rest_{U_0 } = \kappa \,, 
\ \ \tilde \kappa_t \rest_{\complement \, \widetilde U_0 } = 
{\rm{id}}_{\complement \, \widetilde U_0 } \,. \]
Then $ \tilde \kappa = 
 \kappa_1 $ provides the desired extension of $ \kappa $,
and the family $ \kappa_t $ the deformation with compactly supported
$ q_t$'s.
\end{proof}

The proposition means that, as long as we have some geometric freedom
in extending our symplectomorphisms, we can consider local symplectomorphisms
as restrictions of global ones which are isotopic to the identity 
with compactly supported Hamiltonians. We denote the latter class by 
$ \mathcal K $:
\[ {\mathcal K} \defeq \{ \kappa \; : \; T^* \RR^d \rightarrow T^* \RR^d \,, 
  \quad \text{\eqref{eq:pcl} holds with 
compactly supported $ q_t $.} \} \,. \]
We note that, except for $ d= 1, 2 $, it is not known whether every 
$ \kappa $ which is equal to the identity outside a compact set 
is in $ \cK $. 

For $ \kappa \in \cK $ we now define a class of (semiclassical) Fourier
integral operators associated with the graph of $ \kappa $. It
fits in the Heisenberg picture of quantum mechanics -- see 
\cite[\S 8.1]{Met} for a microlocal version and 
\cite[\S \S 10.1,10.2]{EZB} for a detailed presentation on the
semiclassical setting.

\medskip

\begin{defn}{\em Let $\kappa\in\cK$ as above, and let
$ 0 \leq \delta < 1/2 $. The operator $U:\cS(\RR^d)\to \cS'(\RR^d)$
belongs to the class of $h$-Fourier integral operators
\[  I_\delta ( \RR^d \times \RR^d , C' ) \,, \ \ 
C' = \{ ( x, \xi;  y , - \eta) \; : \; ( x, \xi) = \kappa ( y , \eta )   \}
\,, \]
if and only if there exists $ U_0 \in \Psi_{\delta}( \RR^d ) $, 
such that $ U=  U(1) $, where
\begin{gather}
\label{eq:u0}
\begin{gathered}
h D_t U ( t) +  U ( t) Q (t) = 0 \,, \quad Q ( t ) = \Op ( q_ t ) \,,  \quad
U( 0 ) = U_0 \,.
\end{gathered}
\end{gather}
Here the time dependent Hamiltonian $ q_t \in \CIc ( T^* \RR^d ) $
satisfies {\rm (iii)} of \eqref{eq:pcl}.}
\end{defn}
We will write
\[  
I_{ 0+ } ( \RR^d \times \RR^d , C' ) \defeq
\bigcap_{\delta > 0 }  I_\delta ( \RR^d \times \RR^d , C' )  \,. 
\]
For $ A \in I_{0+} ( \RR^d \times \RR^d , C' ) $ we define its
$h$-wavefront set
\be\label{e:WFh}
\WFh ( A) \defeq \{(x,\xi;y,-\eta) \in C' \; : \; (y,\eta) \in  
\WFh ( U_0 )  \}\subset T^*\RR^d\times T^*\RR^d \,, 
\ee
where $ \WFh ( U_0 ) $ is defined in \eqref{eq:WFA}.

We recall Egorov's theorem in this setting --- 
see \cite[Appendix a]{He-Sj1} or \cite[Theorem 10.10]{EZB}.

\begin{prop}
\label{p:Egorov}
Suppose that $ U = U ( 1 ) $ for  \eqref{eq:u0} with $ U_0 = I $, and
that $ A = \Op ( a ) $, $ a \in S ( T^*\RR^d ) $. 
Then 
\[   U^{-1} A U = B \,, \ \ B = \Op ( b) \,, \ \ b - \kappa^*a \in h^2 S ( T^*\RR^d ) \,. \]
An analogous result holds for $ a \in S^{m,k} ( T^* \RR^d ) $.

More generally, if $ T \in I_\delta ( \RR^d \times \RR^d , C' ) $ and
$A$ as above, then 
\[ 
A T = T B + T_1 B_1 \,, 
\]
where $ B  = \Op ( b ) $ is as above, 
$ T_1 \in I_\delta ( \RR^d \times \RR^d , C' ) $,
and $ B_1 \in h^{1-2\delta}\Psi_\delta ( \RR^d ) $.
\end{prop}

\noindent
{\bf Remark.}
The additional term $ T_1 B_1 $ is necessary, due to the fact that $T$
may not be elliptic.

\medskip

The main result of this subsection is an extension of this Egorov
property to operators $A$ in more exotic symbol classes.
\begin{prop}
\label{p:fiom}
Suppose that $ \kappa \in \cK $ and $ T \in I_\delta 
( \RR^d \times \RR^d ; C' ) $, where $ C $ is the graph of $ \kappa$. 
Take $ A = e^{G^w ( x, h D )} \Op ( a_0 ) $, where $  a_0 \in \ST (1) $ and $ G $ 
satisfies \eqref{eq:ord1}. Then 
\begin{gather}
\label{eq:fiom}
\begin{gathered}
A T = T B + T_1 B_1 \,, \\ 
B = e^{ ( \kappa^* G)^w ( x, hD ) } \Op (b_0 ) \,, \ \ b_0 - \kappa^* a_0 
\in h^{\frac12 } 
\tilde h ^{\frac32} 
\ST ( 1 ) \,, \\
B_1 = h^{\frac 12 ( 1 - 2 \delta ) } \tilde h^{\frac 12} 
 \Op ( b_1 ) \,, \ \ b_1 \in \ST ( e^{\kappa^* G} ) \,, \ \
T_1 \in I_\delta ( \RR^d \times \RR^d , C' ) \,. 
\end{gathered}
\end{gather}
\end{prop}

The proof is based on two lemmas. The first one is essentially 
Proposition \ref{p:fiom} with $ G = 1$, $ \delta = 0 $, and 
$ T $ invertible:
\begin{lem}
\label{l:p1}
Suppose $ U = U ( 1 ) $, where $ U(t) $ solves \eqref{eq:u0} with 
$ U(0)= I $. 
Then for $ A = \Op ( a ) $, $ a \in \ST ( 1 ) $, 
\begin{equation}
\label{eq:lp1}
U^{-1} A U = B \,,  \ \ B = \Op ( b ) \,, \ \ b - \kappa^* a \in h^{\frac12} 
\tilde h^{\frac32}  
\ST ( 1 ) \,. \end{equation}
\end{lem}
\begin{proof}
We will use the $ \ST $ variant of Beals's lemma given in Lemma \ref{l:Beals}
above. Let $ \ell_j $ be as in \eqref{eq:Beals}, $ j = 1, \cdots, N $, 
and denote
\[ \Op(\ell_j ( t ) ) \defeq U ( t ) \Op ( \ell_j ) U ( t )^{-1}  \,. \]
We first claim that $ \ell_j ( t ) $ are, to leading order, 
linear outside of a compact set:
\begin{equation}
\label{eq:ellt}
 \ell_j ( t ) = (\kappa_t^{-1} )^* \ell_j + h^2 r_t \,, \ \ r_t \in S ( 1 ) \,.
\end{equation}
To prove it, we see that the evolution \eqref{eq:u0} gives
\begin{equation}
\label{eq:st1}
 h D_t \ell_j ( t )^w  = [ \widetilde Q (t) , \ell_j ( t )^w ] \defeq L_j ( t ) \,, \end{equation}
where
\[ \widetilde Q ( t ) \defeq  U ( t ) \Op ( q_t )  U ( t )^{-1} = 
\Op ( \tilde q_t  + {\mathcal O}_{ S ( 1 ) } ( h^2) ) \,, \ \  \tilde q_t = (\kappa_t^{-1})^* q_t \,.  \]
Then, using Lemma \ref{l:Sjnew}, 
\[ L_j ( t ) - 
(h/i) \Op (H_{ \tilde q_t } \ell_j ( t )  ) 
\in h^3 S(1) \,. \]
Since $ (d/dt ) \kappa_t = H_{\tilde q_t } $,  \eqref{eq:st1} becomes
\[ \partial_t ( \kappa_t^* \ell_j ( t ) ) = {\mathcal O}_{S(1)} ( h^2 ) \,,\]
which implies \eqref{eq:ellt}.

For $ A $ in the statement of the lemma let us define $ A( t) $ as 
\[ A ( t) \defeq U( t)^{-1} A U( t )  \,, \ \ 
hD_t A( t ) = [ Q ( t) , A ( t) ]\,, \ \ A ( 0 ) = A \,. \]
We want to show that $ A ( t ) = \Op ( a ( t ) ) $ where $ a ( t) \in
\ST $. 
We will prove that using \eqref{eq:Beals}:
\begin{equation}
\label{eq:adlt}
\begin{split} &   \ad_{\Op (\ell_1)} \circ \cdots \ad_{\Op (\ell_N ) } A( t)  =  
U ( t )^{-1}  \left( \ad_{\Op (\ell_1 ( t ) )} \circ \cdots 
\ad_{\Op (\ell_N ( t) ) }  A \right) U (  t) \,. 
\end{split} \end{equation}
In view of \eqref{eq:ellt} and the assumptions on $ A $ we 
see that the right hand side is 
\[  {\mathcal O} ( h^{N/2} \tilde h^{N/2} ) \, : \,  L^2 ( \RR^d ) 
\, \longrightarrow \, L^2 ( \RR^d ) \,, \]
and thus $ A ( t) = \Op ( a ( t) ) $, 
$ a ( t) \in \ST $. One has $B=A(1)$.

We now need to show that $ a ( t ) - \kappa_t^* a \in h^{\frac 12}
\tilde h ^{\frac 32} \ST $.  For that define 
\[
{\widetilde A}(t) := \Op({\mathbf \kappa}_t^*a).
\]
As in the proof of Egorov's theorem, we calculate
\begin{eqnarray*}
hD_t{\widetilde A}(t) &= &\frac {h}{i} \Op\left( \frac {d}{dt}
{\kappa}_t^*a\right) = \frac {h}{i}
\Op(H_{q_t} { \kappa}_t^* a) \\ &= &\frac {h}{i}
\Op(\{q_t, { \kappa}_t^*a\})
= [Q(t),{\widetilde A}(t)] + h^{\frac32} \tilde h^{\frac32} E (t),
\end{eqnarray*}
and Lemma \ref{l:Sjnew} shows that 
$E(t) = \Op ( e (t ) ) $, where $e ( t )  \in \ST $.
A calculation then shows that
\[ \begin{split}
hD_t(U(t){\widetilde A}(t){U}(t)^{-1}) &= 
U(t)E(t){U}(t)^{-1} \,, 
\end{split} \]
and consequently, 
\begin{equation}
\label{B-int}
A ( t) - \widetilde {A}(t) = i h^{\frac12} \tilde h^{\frac32}  \int_0^t U(t)^{-1}U(s) E (s)
{U}(s)^{-1} U ( t )  ds \,.
\end{equation}
We have already shown that conjugation by $ U (s) $ preserves
the class $ \ST $, hence the integral above is in 
$ \ST $.  
Then  $ A ( t ) - \tilde A ( t ) \in 
 h^{\frac12} \tilde h^{\frac 32} \ST $.
\end{proof}

\medskip

\noindent
{\bf Remark.} The use of Weyl quantization is essential for getting an
error of size $ h^{\frac12} \tilde 
h^{\frac32} $. To see this we take $ U $ to be 
a metaplectic transformation -- see for instance \cite[Appendix to Chapter 7]{DiSj}.
The rescaling \eqref{eq:resc} simply changes it to the 
same metaplectic tranformation, $ \tilde U $,
with $ \tilde h $ as the new Planck constant.
Then 
\[ \widetilde U^{-1} \tilde a^w ( \tilde x , \tilde h D_{\tilde x} ) \widetilde U 
= \left(\tilde \kappa^* \tilde a \right)^w
 ( \tilde x , \tilde h D_{\tilde x } ) \,, \]
that is we have no error term. Had we used right quantization, $ a (
x, h D ) $, we would have
acquired error terms of size $ \tilde h $, which could not be 
eliminated after rescaling back. For the invariance of
the $ \ST $ calculus see \cite[\S 3.3]{WuZ}.

The arguments of the previous lemma can be extended to encompass the
weight function $G$ of \eqref{eq:G=logm}, which may not be in $\ST$
(indeed, $G$ may be unbounded), but which has bounded derivatives.
\begin{lem}
\label{l:fioG}
Let $ U $ be as in Lemma \ref{l:p1}. For $ G $ satisfying 
\eqref{eq:ord1} we define
\[  {\mathcal G}_1  \defeq U^{-1} G^w ( x , h D ) U \; : \; {\mathcal S} 
\longrightarrow {\mathcal S} \,. \]
Then
\[ {\mathcal G}_1 = G_1^w ( x, h D) \,, \ \ G_1 - \kappa^* G \in 
h^{\frac12} \tilde h^{\frac 32} \ST ( 1) \,. \]
\end{lem}
\begin{proof} 
Since $ U : {\mathcal S } \rightarrow {\mathcal S} $, 
the operator $ {\mathcal G}_1 $ maps $ {\mathcal S} $ to $ \mathcal S $.
We now proceed as in the proof of Lemma \ref{l:p1}, noting that
\[  
\partial^{\alpha }  ( \kappa_t^* G ) \in ( h / \tilde h )^{-|\alpha |/2)} 
\ST \,, \quad |\alpha| \geq 1 \,,\quad\text{uniformly for }t\in
[0,1]\,.
\]
Lemma \ref{l:Sjnew} shows that only terms involving derivatives
appear in the expansions, hence the same arguments apply.
\end{proof}
We now combine these various lemmas.

\medskip

\noindent
{\em Proof of Proposition \ref{p:fiom}:} Suppose that 
$ T \in I_\delta ( \RR^d \times \RR^d , C' ) $, meaning that $ T
=U_0U$, where $ U = U( 1 ) $ satisfies \eqref{eq:u0} with $ U(0)=I$. 
By Egorov we may write $T= U U_1 $, with
$ U_1 \in \Psi_{\delta} $. 
Then, in the notation of Lemma \ref{l:fioG},
\[ A T =  U ( U^{-1} A U ) U_1 = U  \exp {\mathcal G}_1 U^{-1} A_0 U U_1
= T B + U  B_1 \,, \]
where, using Lemma \ref{l:fioG},
\[  B =  \exp { \mathcal G}_1 U^{-1} A_0 U  = \exp ( ( \kappa^* G)^w ) B_0 
\,, \ \ B_0 \in \PT \,, \]
and, using \eqref{eq:FF}, 
\[ \begin{split} B_1 & = [ \exp ( ( \kappa^* G)^w ) , U_1 ] B_0 + 
\exp ( ( \kappa^* G)^w )  [ B_0 , U_1 ] \\
& = \exp ( ( \kappa^* G)^w )  \left( U_1 - \exp (- ( \kappa^* G)^w ) U_1
\exp ( ( \kappa^* G)^w ) + [ B_0, U_1 ] \right) \\ 
& \in  h^{ \frac12 ( 1 - 2 \delta ) } \tilde h ^{\frac12} 
\exp ( ( \kappa^* G)^w ) \PT \,. 
\end{split} 
\]
\stopthm

\subsubsection{Lagrangian relations}
We are now ready to extend the above semiglobal theory 
(``semi'' because of our special class of symplectomorphisms $ \cK $) into 
a construction of $h$-Fourier integral operators associated with an arbitrary smooth
Lagrangian relation on $F\subset T^*Y\times T^*Y$ (as introduced in \S
\ref{2gr}). This construction will naturally be done by splitting $F$
into local symplectomorphisms defined on (small) star shaped sets.

Eventually, we want consider the full setup of \S \ref{2gr}, that is
taking $F$ the disjoin union of $F_{ij}\subset U_i\times U_j$, 
defining our Fourier integral operator $ T \in I_{\delta} ( Y \times Y , F' ) $ as
a matrix of operators,
\[  T = ( T_{ij} )_{ 1\leq i, j \leq J } \,, \ \
T_{ij} \in I_{\delta} ( Y_i \times Y_j , F_{ij}' ) \,,\]
and finally take $I_{0+}(Y\times Y,F') = \bigcup_{\delta>0}I_{0+}(Y\times Y,F')$.

\medskip

To avoid too cumbersome notations, we will omit the indices $i,j$, and
consider a single Lagrangian relation
$F\subset U'\times U$, where $U\Subset T^*Y\Subset T^*\RR^d$,
$U'\Subset T^*Y'\subset T^*\RR^d$ two
open sets, and define the
classes $ I_{\delta} ( Y' \times Y ) $.

Fix some small $\vareps>0$. On $U$ we introduce two open covers of $U$,
$$
U \subset \bigcup_{\ell=1}^{L} U_{\ell}\,,\quad U_{\ell}\Subset\tU_{\ell}\,,
$$
such that each $\tU_{\ell}$ is star shaped around one of its
points, and has a diameter $\leq \vareps$.
We also introduce a
smooth partition of unity $(\chi_{\ell})_{\ell=1,\ldots,L}$
associated with the cover $(U_{\ell})$:
\be\label{e:partition}
\sum_{\ell}\chi_{\ell}(\rho) = 1,\qquad \rho\in {\rm
  neigh}(U),\quad \chi_{\ell}\in\CIc(U_{\ell},[0,1])\,.
\ee
$F$ can be seen as a
canonical map defined on the departure subset $\pi_R(F)\subset
U$, with range $\pi_L(F)\subset U'$. Let us call
$\tF_{\ell}=F\rest_{\tU_{\ell}}$ its restriction to $\tU_{\ell}$.

The set of interior indices $\ell$ such
that $\tU_{\ell}\subset \pi_R(F)$ will
be denoted by ${\mathcal L} $.

For each interior index $\ell$, the symplectomorphism 
$\tF_{\ell}$ is the extension of $F_{\ell}=F\restriction_{U_{\ell}}$,
so we may apply Proposition~\ref{p:cl}, and produce a 
global symplectomorphism $\tkappa_{\ell}\in\cK$, which coincides with $F_{\ell}$ on the
set $U_{\ell}$. 
The previous section
provides the family of $h$-Fourier integral operators 
$I_{\delta}(\RR^d\times \RR^d, C_{\ell}')$, where $C_{\ell}$ is
the graph of $\tkappa_{\ell}$. For each interior index
$\ell$ we consider a Fourier integral operator $\tT_{\ell}\in I_{\delta}(\RR^d\times
\RR^d,C_{\ell}')$, and use the
partition of unity \eqref{e:partition} to define 
$$
T_{\ell}\defeq \tT_{\ell}\, \chi_{\ell}^w(x,hD)\,.
$$
Due to the support properties of $\chi_{\ell}$, the operator
$T_{\ell}$ is actually associated with the restriction
$F_{\ell}$ of $\tkappa_{\ell}$ on $U_{\ell}$. The sum
$$
T^{\RR}\defeq\sum_{\ell\in {\mathcal L} } T_{\ell}
$$
defines a Fourier integral operator on $\RR^d$, microlocalized inside
$\pi_L(F)\times \pi_R(F)$, which we call $I_{\delta}(\RR^d\times\RR^d,F')$.

Finally, since we want the wavefunctions to be defined on
the open sets $Y$, $Y'$ rather than on the whole of $\RR^d$, we
use cutoffs $\Psi\in \CIc(Y,[0,1])$, $\Psi\in \CIc(Y,[0,1])$ such that $\Psi(x)=1$ on
$\pi(U)$, $\Psi'(x)=1$ on $\pi(U')$.

We will say that $T:{\mathcal D}' ( Y) \rightarrow \CI (
\overline Y' )$ belongs to the class 
$$
I_{\delta} ( Y' \times Y , F' )
$$
iff
\[
\Psi'\,T\,\Psi=\Psi'\,T^{\RR}\,\Psi\quad \text{for some}\quad
T^{\RR}\in I_{\delta}(\RR^d\times\RR^d,F'),
\]
and 
\[ T - \Psi'\,T\,\Psi = 
\Oo( h^\infty ) \; : \; {\mathcal D}' ( Y) \rightarrow \CI(
\overline Y' ) \,. 
\]
We notice that $\pi_R(\WFh(T))$ is
automatically contained in the support of 
$\sum_{\ell\in {\mathcal L}} \chi_{\ell}$, a strict subset of $\pi_R(F)$: in
this sense, the
above definition of $I_{\delta} ( Y' \times Y , F' )$ depends on the partition of unity
\eqref{e:partition}. However, for any subset of $W\Subset F$, one can always
choose a cover $(\tU_{\ell})$ such that $\bigcup_{j\in {\mathcal L}}\tU_{\ell}\supset
\pi_R(W)$.
In particular, the assumption $\cT\cap\partial F=\emptyset$ we made in \S\ref{2gr}
shows that such a subset $W$ may contain the trapped set $\cT$.

\subsubsection{Conjugating global Fourier integral operators by weights}

By linearity one can generalize Prop.~\ref{p:fiom} to a
Fourier integral operator $T\in I_{0+}(Y'\times Y,F')$, and by linearity to the full setup of
\S\ref{2gr}. 
The observable $a_0$ and weight $G$
are now functions on  $T^*Y'$ or on $T^*\RR^d$. 

If $G$ is supported inside $\pi_L(F)\subset T^*Y'$, then $F^*G=G\circ F$ is a smooth
function on $\pi_R(F)$, which can be smoothly extended (by zero)
outside. In each $U_{\ell}$ we apply Prop.~\ref{p:fiom} to
$T_{\ell}$, and obtain on the right hand side terms of the form 
$$
T_{\ell}\,e^{\tkappa_{\ell}^*G }\,\Op(b_0) + T_1\,
B_1\,,\quad T_1\in I_{\delta}(\RR^d\times\RR^d,F'),\quad 
B_1\in h^{1/2-\delta}\th^{1/2}\PT(e^{\tkappa_{\ell}^*G})\,,
$$
and $\WFh(T_1)\subset \WFh(T)$.

Since $\pi_R(\WFh(T_{\ell}))\Subset U_{\ell}$,
Lemma~\ref{l:G-support} shows that only the part of
$\tkappa_{\ell}^*G$ inside $U_{\ell}$ is relevant to the above
operator, that is a part where $\tkappa_{\ell}\equiv F$. Therefore we have
$$
T_{\ell}\,e^{(\tkappa_{\ell}^*G)^w} =
T_{\ell}\,e^{(F^*G)^w }+\Oo(h^\infty)\,,\qquad
B_1\in h^{1/2-\delta}\th^{1/2}\PT(e^{F^*G}) +\Oo(h^\infty)\,.
$$ 
This proves the generalization of Prop.~\ref{p:fiom} to the 
setting of the relation $F\subset T^*Y'\times T^*Y$, in case $\supp G\subset \pi_L(F)$.

\medskip

In case $G$ is not supported on $\pi_L(F)$, the notation $e^{(F^*G)^w}$
still makes sense microlocally inside $\pi_R(F)$. 
Indeed, take
$\chi,\tilde\chi\in\CIc(\pi_L(F))$, 
$\chi\equiv 1$ near $\pi_L(\WFh(T))$, $\tilde
\chi\equiv 1$ near $\supp\chi$.
Lemma~\ref{l:G-support}, with $V=\pi_L(F)$
implies that
$$
e^{G^w}\,T=e^{\tilde{G}^w}\,T+\Oo(h^\infty)\,,\qquad
\tilde{G}\defeq \tilde\chi\,G.
$$
The above generalization of Prop.~\ref{p:fiom} then shows that
$$
e^{\tilde{G}^w}\,a_0^w\,T+\Oo(h^\infty)
= T\,e^{F^*\tilde{G}^w}\,b_0^w + T_1\,B_1\,.
$$
The weight $F^*\tilde{G}$ is only relevant on $\pi_R(\WFh(T))\Subset
\pi_R(F)$, so it makes sense to write the first term on the above right
hand side as
$$
T\,e^{F^*G^w}\,b_0^w\defeq T\,e^{F^*\tilde{G}^w}\,b_0^w\,,
$$
emphasizing that this operator does not depend (modulo
$\Oo_{\cS'\to \cS}(h^\infty)$) of the way we have truncated $G$ into $\tilde{G}$.

For the same reason, the symbol class $\ST ( e^{F^* G} )$ makes sense if we
assume that the symbols are essentially supported inside
$\pi_R(F)$. We have just proved the following generalization of Prop.~\ref{p:fiom}:
\begin{prop}
\label{p:fiom-global}
Take $F$ a Lagrangian relation as described in \S\ref{2gr}, and 
$T\in I_\delta(Y\times Y,F')$ as defined above. Take $G\in\CIc(T^*Y)$
a weight function
satisfying \eqref{eq:ord1}, and $ A = e^{G^w ( x, h D )} \Op ( a_0 )
$, where the symbol $  a_0 \in \ST (1) $, $\esss a_0 \Subset \pi_L(F)$.

Then the following Egorov property holds:
\begin{gather}
\label{eq:fiom-global}
\begin{gathered}
A T = T B + T_1 B_1 \,, \\ 
B = e^{ ( F^* G)^w ( x, hD ) } \Op (b_0 ) \,, \quad b_0 - F^* a_0 
\in h^{\frac12 }\th ^{\frac32} \ST ( 1 ) \,, \quad \esss b_0 \Subset \pi_R(F)\,,\\
B_1 = h^{\frac 12 ( 1 - 2 \delta ) } \th^{\frac 12} 
 \Op ( b_1 ) \,, \quad b_1 \in \ST ( e^{F^* G} ) \,, \quad
T_1 \in I_\delta ( Y\times Y , F' ) \,. 
\end{gathered}
\end{gather}
\end{prop}

This proposition is the main result of our preliminary section on
exotic symbols and weights.
Our task in the next section will be to construct an explicit weight
$G$, adapted to the hyperbolic Lagrangian relation $F$.

\section{Construction of  escape functions}
\label{cef}

Escape functions are used to conjugate our quantum map (or
monodromy operator), so that the conjugated operator has nicer microlocal
properties than the original one, even though it has the same
spectrum. More precisely, an escape function $G(x,\xi)$ should have
the property to strictly increase along the dynamics, away from the
trapped set. An escape function $g(x,\xi)$ has already been used in to construct the
monodromy operators associated with the scattering problems in
\cite{NSZ1}: its effect was indeed to damp the monodromy operator by a
factor $\sim h^{N_0}$ outside a fixed neighbourhood of $\cT$.
Our aim in this section is to construct a more refined escape
function, the r\^ole of which is to damp the monodromy operator 
outside a semiclassically small neighbourhood of $\cT$, namely an 
$h^{1/2}$-neighbourhood. For this aim, it is necessary to
use the calculus on symbol classes $\ST(m)$ we have presented in \S\ref{pr}.

Our construction will be made in two steps: first in the vicinity the trapped set $\TT$, following 
\cite[\S 7]{SZ10} (where it was partly based on \cite[\S 5]{SjDuke}),
and then away from the trapped set, following an adaptation of
the arguments of \cite[Appendix]{GeSj}.
For the case of relations $ F $ as in \S \ref{2gr} which arise from
Poincar\'e maps of smooth flows, we could alternatively
use the flow escape functions given in \cite[Proposition 7.7]{SZ10}. 
However, the general presentation for open hyperbolic maps is simpler than that
for flows, and will also apply to the monodromy operators obtained
from the broken geodesic flow of the obstacle scattering problem.

\subsection{Regularized escape function near the trapped set}
\label{ref}

Let $ \TT_\pm $ be the outgoing and incoming tails given by 
\eqref{eq:defT} in the case of the obstacle scattering. For 
an open map $ F $ with properties described in \S \ref{2gr}, these
sets are defined by
$$ 
\TT_\pm = \{ \rho \; : \; F^{\mp n } ( \rho ) \in \, \UU \,, \
\forall \, n \geq 0 \}\,, 
$$
where $ F^{\mp n} = F^{\mp} \circ \cdots \circ F^{\mp } $ denotes the usual
composition of relations. We note that $ \TT_\pm $ are closed
subsets of $ \UU $, and due to the hyperbolicity of the flow, they are
unions of unstable/stable manifolds $W^{\pm}(\rho)$, $\rho\in\cT$.

\medskip

\noindent{\bf Remark} Before entering the construction, let us consider the simplest model
of hyperbolic map, namely the linear dilation 
$(x,\xi)\mapsto (\Lambda x,\Lambda^{-1}\xi)$ on $T^*\RR$, with
$\Lambda>1$. 
In that case, the trapped set is reduced to one point, the origin, and
the sets $\TT_{\pm}$ coincide with the position and
momentum axes. In this case, a simple escape function is given by
\be\label{e:escape-1}
G(x,\xi)=x^2-\xi^2 = d(\rho,\TT_-)^2 - d(\rho,\TT_+)^2\,,
\ee
and it can be used (after some modification) to analyse scattering
flows with a single hyperbolic periodic orbit \cite{Ge,GeSj}.

\medskip

The construction of $G(x,\xi)$ in the case of a more complex, but
still hyperbolic, trapped set, inspires itself from the expression
\eqref{e:escape-1} \cite{SjDuke}. Our first Lemma is a construction of
two functions related with, respectively, the outgoing and incoming
tails. It is a straightforward adaptation of \cite[Prop. 7.4]{SZ10}.
For a moment we will use a small parameter $\eps>0$, which will
eventually be taken equal to $h/\th$.
\begin{lem} 
\label{l:S70}
Let $ \tVV $ be a small neighbourhood of $ \TT $ and 
$ \tF : \tVV \rightarrow \tF ( \tVV ) $ be 
the symplectomorphic restriction of $ F $. Then, there exists $C_0>0$
and a neighbourhood $\VV\Subset\tVV$ of the trapped set, such that the
following holds.

For any
small $ \epsilon > 0 $ there
exist functions $ \hph_\pm \in \CI ( \VV\cup\tF(\VV) ; [\eps, \infty) ) $ such that
\begin{equation}
\label{eq:es1'}
\begin{split}
& \hph_\pm  ( \rho )  \sim d( \rho  , \TT_\pm )^2 + \epsilon
 \,, \\
& \pm (\hph_\pm ( \rho )  - \hph_\pm ( \tF ( \rho ) ) ) + C_0 \epsilon
\sim  \hph_\pm ( \rho )
\,, \quad \rho\in\cV\,,\\
& \partial^\alpha
 \hph_\pm (\rho ) =\Oo ( \hph_\pm ( \rho )^{ 1
- |\alpha|/2} )\,,  \\
&  \hph_+ ( \rho ) + \hph_- ( \rho )
\sim d( \rho ,  \TT)^2 +  \epsilon \,.
\end{split}
\end{equation}
Here and below, $ a \sim b $ means that there exists a constant $ C\geq 1 $ (independendent
of $ \epsilon $) such that $ b/C \leq a \leq C b $.
\end{lem}

To prove this lemma we need two preliminary results. 
\begin{lem}
\label{l:whitney}
Suppose $ \Gamma \subset \RR^m $ is a closed set. For any $ \epsilon > 0 $
there exists $ \varphi  \in \CI ( \RR^m )$, such that
\be\label{e:Whitney} \varphi  \geq \epsilon \,, \ \
 \varphi \sim \epsilon +  d ( \bullet, \Gamma)^2 
  \,, \qquad
\partial^\alpha \varphi  = \Oo ( \varphi  ^{1-|\alpha|/2 } ) \,,
\ee
where the estimates are uniform on $ \RR^m $.
\end{lem}
\begin{proof} For reader's convenience we 
recall the proof (see \cite[Lemma 7.2]{SZ10}) based on 
a Whitney covering argument (see \cite[Example 1.4.8, Lemma
1.4.9]{Hor1}).
For $ \delta \ll 1 $  choose a maximal sequence  $ x_j \in \RR^m \setminus \Gamma $
such that $ d ( x_j, x_i ) \geq \delta d ( x_i, \Gamma ) $ (here $ d $
is the Euclidean distance).  We claim that 
\begin{equation}
\label{eq:18}  \bigcup_j B ( x_j , d( x_j , \Gamma ) / 8 ) = \RR^m \setminus
\Gamma  \,. \end{equation}
In fact, if $ x $ is not in the sequence then, for
some $ j$,
$   d ( x, x_j ) < \delta d ( x_j, \Gamma )  $  or $ d( x,
x_j ) < \delta d ( x , \Gamma )$.
In the first case $ x \in B ( x_j , d ( x_j , \Gamma )/ 8 ) $ if $
\delta < 1/8 $. In the second case,  
$  d ( x, x_j ) <  \delta ( d ( x_j , \Gamma ) + d ( x , x_j ) )
$ which means that $ x  \in B ( x_j , d ( x_j , \Gamma )/8 ) $, 
if $ \delta / ( 1 - \delta ) < 1/8 $.  Hence \eqref{eq:18} holds
provided $ \delta < 1/9 $. 

We now claim that 
every $  x \in \RR^m \setminus \Gamma $ lies in at most $ N_0   =
N_0 ( \delta , m ) $ balls $ B ( x_j , d( x_j , \Gamma ) / 2 ) $.
To see this consider $ x $ and $ i \neq j $ such that $ d ( x , x_j )
\leq d ( x_j, \Gamma ) /2 $ and $ d ( x, x_i ) \leq d ( x_j , \Gamma )
/ 2 $.  Then simple applications of the triangle inequality show that
\[  d ( x_i, x_j ) \geq 2 \delta d (
 x, \Gamma ) / 3 \,, \ \   d ( x_j
, \Gamma ) \leq d ( x, \Gamma ) / 2 
\,. \]
Hence 
\begin{gather*}  B ( x_i , \delta d ( x ,  \Gamma)/3 ) \cap 
 B ( x_j , \delta d ( x ,  \Gamma)/3 )  = \emptyset \,, \\
B (  x_\ell, \delta d ( x ,  \Gamma)/3 )   \subset  B ( x , 4d (x,
\Gamma )/3 )  \,, \ \ \ell = i, j \,.\end{gather*}
Comparison of 
volumes shows that the maximal number of such $ \ell's $ is $ (4/\delta)^m  $. 

Let $ \chi \in \CIc ( \RR^m ; [0,1]) $ be supported in $ B (0 , 1/4 ) $,
and be identically one in $ B ( 0 , 1/8)$. We define
\[ \varphi_\epsilon ( x ) \stackrel{\rm{def}}{=} \epsilon
+ \sum_{ d( x_j , \Gamma ) > \sqrt \epsilon } d ( x_j , \Gamma )^2
\chi \left( \frac{  x - x_j  } { d ( x_j , \Gamma ) + \sqrt \epsilon }
\right) \]
We first note that the number non-zero terms in the sum is uniformly bounded
by $ N_0 $. In fact, $ d ( x_j , \Gamma) + \sqrt \epsilon < 2 d ( x_j , 
\Gamma) $, and hence if $ \chi( ( x - x_j ) / ( d ( x_j , \Gamma ) +     
\sqrt \epsilon ) ) \neq 0 $ then
$$ 1/4 \geq |x-x_j| / ( d ( x_j , \Gamma ) + \sqrt \epsilon )
\geq (1/2) |x - x_j | / d( x_j , \Gamma )  \,,$$
and $ x \in B ( x_j , d ( x_j , d ( x_j , \Gamma) )  / 2 ) $.
This shows that $ \varphi_\epsilon ( x )                                 
\leq 2 N_0 (\epsilon +  d ( x, \Gamma )^2 ) $, and
\[ \partial^\alpha \varphi_\epsilon ( x) =
{\mathcal O} ( ( d ( x, \Gamma )^2 + \epsilon )^{1 - |\alpha|/2 } )
\,,\\
\]
uniformly on compact sets.

To see the lower bound on $ \varphi_\epsilon $
we first consider the case when $ d ( x , \Gamma ) \leq C \sqrt \epsilon\
 $.
\[ \varphi_\epsilon ( x ) \geq \epsilon \geq ( \epsilon + d ( x , \Gamma\
)^2)
 / C'  \,.\]
If $ d ( x , \Gamma ) > C \sqrt \epsilon $ then for at least one
$ j $, $ \chi ( ( x - x_j ) / (d ( x_j , \Gamma ) + \sqrt \epsilon ) )   
= 1 $ (since the balls $ B( x_j , d ( x_j , \Gamma ) /8 ) $
cover the complement of $ \Gamma $, and $ \chi ( t ) = 1 $ if $ |t| \leq\
 1/8                                                                     
$). Thus
$$ \varphi_\epsilon ( x ) \geq \epsilon + d ( x_j , \Gamma ) ^2
\geq ( \epsilon + d ( x , \Gamma)^2 )/ C \,, $$
which concludes the proof. 
\end{proof}

The second preliminary result 
is essentially standard in the dynamical systems literature, resulting
from the hyperbolicity of the map $f$ on $\cT$.
\begin{lem}
\label{l:stan}
For $\tVV$ a small enough neighbourhood of $\TT$, 
there exist $0<\theta_{1}<1$
and $C>0$ such that, for any $ K \geq 0 $ and 
$\rho\in\tVV$ such that $\tF^{k}(\rho)$
remains in $\tVV$ for all $0\leq k \leq K$, we have 
\[
d(\tF^{k}(\rho),\TT_{+})\leq C\,\theta_{1}^{k}\,d(\rho,\TT_{+}),\quad 0\leq k\leq K.\]
The same property holds in the backwards evolution with $ \TT_+ $
replaced by $\TT_- $. 
\end{lem}
\begin{proof} 
We want to use the fact that the map $\tF$ is strictly contracting in
the direction transerse to $\TT_+$ (unstable manifold). To state this
contractivity it is convenient to choose coordinate charts
adapted to the dynamics, and containing the points $\tF^k(\rho)$. 
\begin{figure}
\begin{center}
\includegraphics[angle=-90,width=0.8\textwidth]{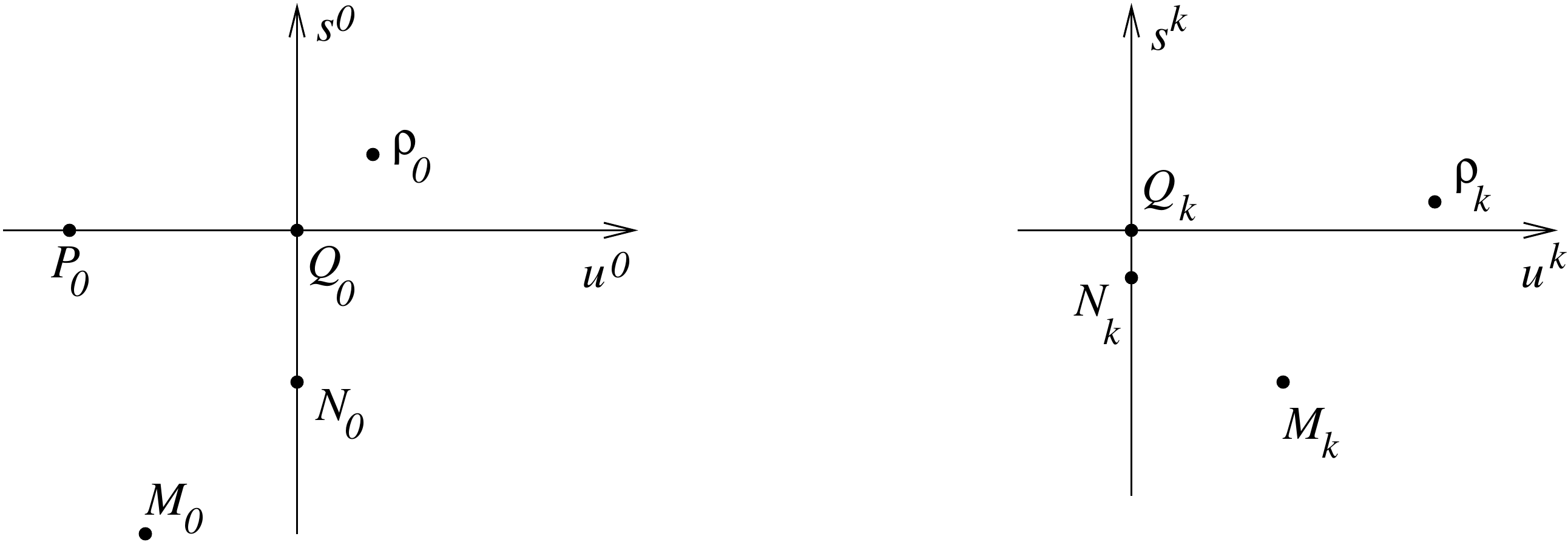}
\caption{\label{f:Lemma43}}
\end{center}
\end{figure}
Let us assume that $\tVV$ is an $\vareps$-neighbourhood
of $\cT$, with $\vareps>0$ small.
By assumption, each $\rho_{k}=\widetilde{F}^{k}(\rho_{0})$, 
$ \rho_0 = \rho$, 
lies in an $\vareps$-size neighbourhood 
of some $M_{k}\in\TT$, $ k \leq K $. 
As a result, the sequence $(M_{k})$
satisfies $d(F(M_{k}),M_{k+1})\leq C\vareps$: it is a $C\vareps$-pseudoorbit.
From the {\em shadowing lemma} \cite[\S 18.1]{KaHa}, there
exists an associated orbit 
\[ 
N_{k}=F(N_{k-1})\in\TT\,, \quad d(N_{k},M_{k})\leq\delta \,, 
\]
with $\delta$ small if $C\vareps$
is small. Hence, $d(N_{k},\rho_{k})\leq\delta+\vareps$.

Besides, the distance $d(\rho_0,\TT_+)$ is equal to the distance
between $\rho_0$ and a certain local unstable leaf $W_{loc}^+(P_0)$,
with $P_{0}\in\TT\cap\tVV$. We will consider the point $Q_{0}=W_{\rm{loc}}^{-}(N_{0})\cap
W_{\rm{loc}}^{+}(P_{0})\in\TT$ and its images $Q_k=\tF(Q_0)$ to
construct our
coordinate charts $(u^{k},s^{k})$, such that the local stable and unstable manifolds 
$ W_{\rm{loc}}^\pm ( Q_k ) $ (see \cite[\S 6.2]{KaHa}) are given by 
\[W_{\rm{loc}}^{-}(Q_{k})=\left\{ (0,s^{k})\right\} \,, \quad
 W_{\rm{loc}}^{+}(Q_{k})=\left\{ (u^{k},0)\right\} \,.
\] 
From the uniform transversality of stable/unstable manifolds, 
these coordinates can be chosen such that for the Euclidean norm we have
\[ \| u^k \|^2  + \| s^k \|^2 \sim d ( Q_k , \rho^k ) ^2\,,\]
uniformly for $ 0\leq k\leq K $. We also have $\|u^k\|,\ \|s^k\|\leq
C(\vareps+\delta)$.
See Fig.~\ref{f:Lemma43} for a
schematic representation.
In this coordinates, the point $\rho_{k}=(u^{k},s^{k})$ is mapped into 
\be
\tF(u^{k},s^{k})=(A_{k}u^{k}+\alpha_{k}(u^{k},s^{k}),{}^{t}A_{k}^{-1}s^{k}+\beta_{k}(u^{k},s^{k}))\label{eq:coordinates}\end{equation}
with $\alpha_k,\beta_k$ smooth functions, $\alpha_{k}(0,s)=\beta_{k}(u,0)=0$,
$d\alpha_{k}(0,0)=d\beta_{k}(0,0)=0$, and the contraction property  $\norm{ A_{k}^{-1}} \leq\nu <1$.

This contraction implies that 
\[
\norm{s^{k}} \leq (\nu+C(\delta+\vareps))^{k}\,
  \norm{s^{0}} ,\quad 0\leq k\leq K\,. 
\]
We can choose $\vareps,\delta$ small enough such that $\theta_{1}\defeq\nu+C(\delta+\vareps)<1$.
Finally, $d(\rho_{k},W_{\rm{loc}}^{+}(Q_{k}))\sim \norm{s^{k}} $
satisfies
\[
d(\rho_{k},\TT_{+})\leq d(\rho_{k},W_{\rm{loc}}^{+}(Q_{k}))\leq
C\,\theta_{1}^{k}\, d(\rho_{0},\mathcal{T}_{+})\,.
\]
\end{proof} 

\medskip

\noindent
{\em Proof of Lemma \ref{l:S70}.} 
We adapt the proof of \cite[Proposition 7.4]{SZ10}
to the setting of a discrete dynamical system.
Let $ \varphi_\pm $ be the functions provided by Lemma
\ref{l:whitney}, respectively for 
$ \Gamma = \TT_{\pm} $. As above we call $\tF$ the restriction of $F$
on $\tVV$, and similarly call $\tF^{-1}=F^{-1}\rest_{\tVV}$. 
For $ K\geq 1 $ to be determined below, we consider
the following neighbourhood of $\TT$:
\be\label{e:VV}
\VV=\bigcap_{k= -K-1}^{K+1} \tF^{k}(\tVV)\,,\quad\text{and define}\
\VV'\defeq \VV\cup \tF(\VV)\,.
\ee
In words, $\VV$ is the set of points $\rho\in\tVV$, whose orbit remains
in $\tVV$ in the time interval $[-K-1,K+1]$. 
The following functions are then well-defined on $\VV'$:
\[ 
\hph_\pm ( \rho ) \defeq \sum_{k=0}^K \varphi_\pm ( \tF^{\pm k } ( \rho )) \,.
\]
Lemma \ref{l:stan} shows that if $ \tVV $ (and therefore $\VV$) is small enough,
then there exist $ \theta_1\in(0,1) $ and $ C > 1 $, 
such that 
\be\label{eq:distpm}  
d ( \tF^{\pm k} ( \rho ) , \TT_\pm ) \leq 
C \theta_1 ^k d ( \rho, \TT_\pm )  \,, \qquad 0 \leq k\leq K \,.
\ee
It thus follows that 
\begin{gather}
\label{eq:gath}
\begin{gathered} \hph_\pm \; \sim \; 
\varphi_{\pm} ( \rho ) \; \sim \; d ( \rho, \TT_\pm )^2 +  \epsilon \,,
\end{gathered}
\end{gather}
with implicit constants independent of $K$.
This establishes the first statement in \eqref{eq:es1'}. 
To obtain the second statement we see that, for any $\rho\in\VV$,
\[ \begin{split}
& \hph_+ ( \rho )  - \hph_+  ( \tF ( \rho ) ) 
=   \varphi_+ ( \rho ) -  \varphi_+ (
\widetilde  F ^{K+1} ( \rho )) \\
& \hph_-  ( \tF ( \rho ) ) - \hph_- ( \rho ) 
= \varphi_- ( \tF ( \rho )) - \varphi_- ( \tF^{-K} ( \rho )) 
\,.
\end{split}
\]
In view of \eqref{eq:gath}, we find
\[ 
\begin{split} \varphi_+ ( \rho ) -  \varphi_+ 
( \tF ^{K+1} ( \rho )) 
& = \varphi_+ ( \rho ) + \Oo ( d ( \tF^{K+1} ( \rho ) , \TT_+ )^2 
+ \epsilon ) \\
& = 
\varphi_+ ( \rho ) + \Oo ( \theta_1^K d ( \rho  , \TT_+ )^2 
+ \epsilon  )  \\
& = \varphi_+ ( \rho ) ( 1 + \Oo ( \theta_1^K ) ) 
 + \Oo( \epsilon ) \,,
\end{split}\]
and similarly for $ \varphi_- $. Taking $K$ large enough so that $\Oo
( \theta_1^K )\leq 1/2$, we obtain, for some $ C_0>0 $, the required estimates:
\[ 
\pm (\hph_\pm ( \rho )  - \hph_\pm ( \tF ( \rho ) ) + C_0 \epsilon
\sim  \hph_\pm ( \rho ) \,.
\]
The estimate 
\[ \partial^\alpha
 \hph_\pm (\rho ) = {\mathcal O} ( \hph_\pm ( \rho )^{ 1
- |\alpha|/2} )\,,  \]
follows from the properties of $ \varphi_\pm $ stated in Lemma
\ref{l:whitney}. 
It remains to show that
\begin{equation}
\label{eq:distTT}   \hph_+ ( \rho ) + \hph_- ( \rho )
\sim d( \rho ,  \TT)^2 +  \epsilon \,. 
\end{equation}
This results \cite{SjDuke,SZ10} from the uniform transversality of the
stable and unstable manifolds (near the trapped set). 
Indeed, 
this transversality implies that, for any two
nearby points $\rho_1,\rho_-\in\cT$ and $\rho$ near them, we have 
\begin{equation}
\label{eq:trans} d ( \rho ,  W^+_{\rm{loc}}(\rho_1)  \cap 
W^{-}_{\rm{loc}}( \rho_2 )) ^2
\sim d ( \rho , W^+_{\rm{loc}} ( \rho_1 ) ) ^2 + 
d ( \rho , W^-_{\rm{loc}}( \rho_2 ) ) ^2  \,.
\end{equation}
Besides, since $\TT_+$ ($\TT_-$) is a union of local unstable
(stable, respectively) manifolds, 
for any $\rho$ near $\cT$ the distance $d(\rho,\TT_{\pm})$ is equal to  $d ( \rho ,
W^+_{\rm{loc}}( \rho_{\pm} ) )$ for some nearby points $\rho_{\pm}$.
We thus get 
\be
d( \rho ,  \TT_+)^2 +  d( \rho ,
 \TT_-)^2 \sim d ( \rho ,  W^+_{\rm{loc}}(\rho_+)  \cap 
W^{-}_{\rm{loc}}( \rho_- )) ^2 \geq d(\rho, \cT)^2\,.
\ee
On the other hand, $d(\rho,\cT)=d(\rho,\rho_0)$ for some
$\rho_0\in\cT$, so that 
\[\begin{split}
d(\rho,\cT)^2 &= d ( \rho ,  W^+_{\rm{loc}}(\rho_0)  \cap 
W^{-}_{\rm{loc}}( \rho_0 ))^2 \\
&\sim d ( \rho ,  W^+_{\rm{loc}}(\rho_0) )^2
+ d ( \rho ,  W^-_{\rm{loc}}(\rho_0) )^2\geq d(\rho,\TT_+) + d(\rho,\TT_-)\,.
\end{split}\]
We have thus proven the transversality result
\be\label{e:transversality}
 d( \rho ,  \TT)^2 \sim  d( \rho ,  \TT_+)^2 +  d( \rho ,
 \TT_-)^2\,, \qquad \rho\ \text{near }\cT\,.
\ee
The statement \eqref{eq:distTT} then directly follows from \eqref{eq:gath}.
\stopthm

From the properties of Lemma~\ref{l:S70}, and in view of the model
\eqref{e:escape-1},
it seems tempting to take
the escape function of the form $\hph_+(\rho)-\hph_-(\rho)$. Yet,
since we want $e^G$ to be an order function, $G$ cannot grow too fast
at infinity.
For this reason, following \cite[\S 7]{SZ10} we use a logarithmic flattening
to construct our escape function:
\begin{lem}
\label{l:S71}
Let $ \hph_\pm $ be the functions given by Lemma \ref{l:S70}.
For some $ M \gg 1 $ independent of 
$ \epsilon $, let us define the function
\begin{equation}
\label{eq:es2}
\hG \defeq
\log (  M \epsilon + \hph_- ) - \log (  M \epsilon + \hph_+ )
\end{equation}
on the neighbourhood $\VV'$ of the trapped set defined in Eq.~\eqref{e:VV}.

Then there exists $ C_1 > 0 $ such that
\begin{gather}
\label{eq:es4}
\begin{gathered}
\hG=\Oo(\log(1/\eps)),\quad \partial_\rho^\alpha \hG =
 \Oo (   \min (
\widehat \varphi_+ , \hph_- )^{-\frac{|\alpha|}2
} )  = \Oo (  \eps^{ -\frac{|\alpha|}2
} ) \,, \quad  |\alpha | \geq 1 \,, \\
\partial_\rho^\alpha ( \hG (\tF ( \rho ) ) - \hG ( \rho )) =
 \Oo(   \min (\hph_+ , \hph_- )^{-\frac{|\alpha|}2
} )  = \Oo(  \eps^{ -\frac{|\alpha|}2
} ) \,, \quad |\alpha | \geq 0 \,, \quad\rho\in \VV\,,\\
\rho \in \VV ,\  d ( \rho, \TT )^2 \geq C_1 \eps\ \ 
\Longrightarrow \  \hG (\tF ( \rho ) ) - \hG ( \rho ) 
\geq 1 / C_1  \,. 
\end{gathered}
\end{gather}
\end{lem}
\begin{proof}
Only the last property in
\eqref{eq:es4} needs to be checked, the others following directly from
Lemma \ref{l:S70}. For this aim we compute
\be
\label{eq:st2}
\hG ( \tF ( \rho )) - \hG ( \rho )  = 
\log\Big( 1 + \frac{\hph_- ( \tF (\rho)) - \hph_- (\rho)} {M \eps + \hph_- ( \rho ) } \Big) 
+\log\Big( 1 + \frac{ \hph_+(\rho) - \hph_+ (\tF (\rho) ) }
{M \eps + \hph_+ ( \tF (\rho ) ) }\Big)  \,.
\ee
Using \eqref{e:transversality}, the condition
$ d ( \TT , \rho)^2 \geq C_1 \eps $ implies that 
$ d ( \TT_+ , \rho )^2 \geq C_2 \eps $ or 
$ d ( \TT_- , \rho )^2 \geq C_2 \eps $, and $ C_2 $ can be 
taken as large as we wish if $ C_1 $ is chosen large enough. 

Let us take care of the first term in \eqref{eq:st2}. For this we need
to bound from below the ratio
\be\label{e:ratio}
R_-(\rho)\defeq\frac{\hph_- ( \tF (\rho)) - \hph_- (\rho)} {M \eps + \hph_- ( \rho ) }\,.
\ee
Let us call $C_3\geq 1$ a uniform constant for the equivalences in
\eqref{eq:es1'}. The second equivalence shows that
$$
\hph_- ( \tF (\rho)) - \hph_- (\rho) \geq \hph_- (\rho)/C_3 - C_0 \eps\,.
$$
Since the function $ x \mapsto \frac{x/C_3 - C_0}{ x + M} $ 
is increasing for $x\geq 0$, the ratio \eqref{e:ratio} satisfies
$R_-(\rho)\geq -C_0/M$.
If we take $M$ large enough, we ensure that 
$$
\log(1+R_-(\rho))\geq -2C_0/M\,.
$$
Furthermore, 
in the region where $d(\rho,\TT_-)^2\geq C_2\eps$, the first
statement in \eqref{eq:es1'} shows that $\hph_-(\rho)\geq
\frac{C_2\eps}{C_3}$, so that
$$
R_-(\rho)\geq \frac{C_2/C_3^2 - C_0}{C_2/C_3+M}\defeq C_4\,.
$$
If we take $C_1$ (and thus $C_2$) large enough, $C_4$ is
nonnegative, and $\log(1+R_-(\rho))\geq C_4/2$. By increasing $M$ and
$C_1$ if necessary, we can assume that $C_4>6C_0/M$.

The same inequalities hold for the second term in 
\eqref{eq:st2} (the condition is now $ d ( \rho , \TT_+ )^2 \geq C_2
\eps$).

Finally, if $ d ( \rho , \TT )^2 \geq C_1 \epsilon$, we find the inequality
\[ \hG ( \tF ( \rho )) - \hG ( \rho ) \geq \frac{C_4}{2} -
\frac{2C_0}{M}>\frac{C_0}{M} \,. 
\] 
Increasing $C_1$ if necessary, the right hand side is $\geq 1/C_1$.
\end{proof}

\subsection{Final construction of the escape function}
\label{efa}

We now set up the
escape function away from the trapped set. We recall that $\tD$ is the departure, resp. arrival sets of the open relation $F$.
\begin{lem}\label{lem:G_0} 
Let $\cW_2$ be an arbitrary small neighbourhood of
the trapped set, and $\cW_3\Subset \tD$ large enough (in
particular, we require that $\cW_3\Supset \supp a_M$, where $a_M$ is the function in
Definition~\ref{def:HQMO}).
Then, there exists $\gz\in\CIc(T^*Y)$, and a neighbourhood
$\cW_1\Subset \cW_2$ of the trapped set, with the following
properties:
\begin{gather}\label{e:G_0}
\begin{gathered}
\forall \rho\in \cW_1,\quad \gz(\rho)\equiv 0,\quad \\
\forall \rho\in \cW_3,\quad \gz (F(\rho)) - \gz (\rho) \geq 0,\\
\forall \rho\in \cW_3\setminus \cW_2,\quad \gz (F(\rho)) - \gz (\rho) \geq 1\,.
\end{gathered}
\end{gather} 
\end{lem}
In \cite{NSZ1} such a
function $\gz$ was obtained as the restriction (on the
Poincar\'e section) of an escape function for the scattering flow, the latter being constructed in \cite[Appendix]{GeSj}. 

The proof of this Lemma for a general open map satisfying the
assumptions of \S\ref{2gr} will be given in the
appendix.
It is an adaptation of the construction of an escape function near the
outgoing tail performed in \cite{DatVas10}
and \cite{VasZw00}.

Now we want to glue our escape function $\hG$ constructed in
Lemma~\ref{l:S71}, defined in the small neighbourhood $\VV'$ of
$\TT$, with the escape function $\gz$ defined away from the trapped
set: the final escape function $G$ will
be a globally defined function 
on $ T^*Y $. One crucial thing is to check that this function is 
the logarithm of an order function for the $ \ST $ class --- see
\eqref{eq:ord} and \eqref{eq:ord1}. 

The following construction is directly inspired by \cite[Prop. 7.7]{SZ10}
\begin{prop}
\label{p:es2}
Let $\VV$, $\hG$ be as in Lemma~\ref{l:S71}, and choose the
neighbourhoods $\cW_i$ such that
$\TT 
\Subset\cW_2\Subset \cV \Subset\cW_3\Subset \pi_R(F)$.
 
Take $ \chi\in \CIc ( \VV') $ equal to $ 1 $
in $\cW_2\cup F(\cW_2)\Subset \VV'$. 
Construct an escape function $\gz$ as in
Lemma~\ref{lem:G_0}, and define
\[ 
G \defeq \chi \hG + C_5\log(1/\eps)\, \gz\, \in \CIc( T^* Y ). 
\]
Then, provided  $C_5$ is chosen large enough, the function 
$G$ satisfies the following estimates:
\begin{gather}
\label{eq:es3}
\begin{gathered}
 |G(\rho)|  \leq C_6\,\log ( 1 / \epsilon ) \,, \quad
\partial^\alpha G  = 
\Oo( \eps^{-|\alpha|/2} ) \,, \quad |\alpha | \geq 1 \,, 
  \\
\rho\in \cW_2\;\Longrightarrow \; \quad G ( F ( \rho)  ) - G ( \rho )  \geq - C_7\,,
\\
\rho \in \cW_2\,, \quad
  d ( \rho, \TT)^2 \geq C_1 \eps
\;  \Longrightarrow \;  G ( F ( \rho)  ) - G ( \rho )  \geq 1 / C_1
\,,\\
\rho\in \cW_3\setminus \cW_2\;\Longrightarrow G(F(\rho)) - G (\rho)
\geq C_8\,\log(1/\eps)\,.
\end{gathered}
\end{gather}
In addition we have
\be\label{eq:orderf}
 \frac{\exp G ( \rho ) }{ \exp G( \mu ) } \leq C_9 \left \langle
\frac{ \rho - \mu }{\sqrt \epsilon } \right \rangle^{N_1} \,,
\ee
for some constants $ C_9 $ and $ N_1 $.
\end{prop}
Eventually we will apply this construction with the small parameter
\[  
\epsilon = \frac h \th \,. 
\]
In particular, the
condition \eqref{eq:orderf} shows that $ \exp G $ is an order function in 
the sense of \eqref{eq:ord} and \eqref{eq:ord1}. 



\medskip
\noindent
{\em Proof of Proposition \ref{p:es2}:}
The first three lines of \eqref{eq:es3} are
obvious, since
$\chi\equiv 1$ on $\cW_2\cup F(\cW_2)$, and $\gz\circ F-\gz\geq 0$.
To check the fourth line, we notice that outside $\cW_2$, we have 
$$
G ( F ( \rho)  ) - G ( \rho )=\Oo(\log(1/\eps)) + C_5 \log(1/\eps)
(\gz\circ F(\rho)-\gz(\rho))\geq \Oo(\log(1/\eps)) + C_5 \log(1/\eps)\,.
$$
If $C_5$ is chosen large enough, the right hand side is bounded from
below by $C_8 \log(1/\eps)$ for some $C_8>0$.

We then need to check \eqref{eq:orderf}. We first check it
for the function $\hG$: we want to show
\begin{equation}
\label{eq:ordph} \frac{ \hph_{\pm } ( \rho) + M\epsilon } {\hph_\pm ( \mu)
+ M\epsilon } \leq 
C_1 \left \langle
\frac{ \rho - \mu }{\sqrt \epsilon } \right \rangle^2 \,,
\end{equation}
with $C_1$ depending on $ M $. 
Since $  \hph_\pm + M \eps \sim \hph_\pm $, this is the same as 
\[ 
 \frac{ \hph_{\pm } (\rho ) } {\hph_\pm ( \mu )
 } \leq 
\tilde C_1 \Big\langle
\frac{ \rho - \mu }{\sqrt \eps } \Big \rangle^2 \,, \]
and that follows from the triangle inequality and the properties of 
$ \hph_\pm $:
\[ \begin{split} 
\hph_\pm ( \rho )  & \leq C ( d ( \rho , \TT_\pm )^2 + \eps ) 
\leq C ( d ( \mu, \TT_\pm)^2  + | \mu - \rho|^2 + 
\eps ) \\
& \leq C' ( \hph_\pm ( \mu )  + | \mu - \rho|^2 
) = C' ( \hph_\pm ( \mu ) 
 + \eps \la ( \rho - \mu )/ \sqrt \eps \ra^2 ) \\
& \leq 2 C'   \hph_\pm ( \mu )   
  \la ( \rho - \mu )/ \sqrt \eps \ra^2 \,.
\end{split}
\]
Inserting the definition of $ \widehat G $, \eqref{eq:ordph} gives
\[ | \hG ( \rho ) - \hG ( \mu ) | \leq C + 2  \log 
\la ( \rho - \mu ) / \sqrt \eps \ra \,.\]
The estimate for 
$ G $ is essentially the same:
\[ \begin{split} | G ( \rho ) - G ( \mu ) | & \leq  
|  \chi ( \rho ) \widehat G ( \rho ) - 
 \chi ( \mu ) \widehat G ( \mu ) | + C_5\log(1/\eps)\,|\gz(\rho)-\gz(\mu)|\\
& \leq 
C | \rho - \mu | \log ( 1 / \eps ) + C \log \langle ( \rho - \mu ) /
\sqrt \epsilon \rangle + C \\
& \leq C' \log \langle ( \rho - \mu ) / \sqrt \epsilon \rangle  + C' 
\,.
 \end{split} \]
The last estimate follows from 
\[   x \log \frac 1 \epsilon  \leq
C \, { \log \left\langle \frac x {\sqrt \epsilon} \right 
\rangle + C} \,, \qquad 0 \leq x \leq 1 
\,.  \]
\stopthm

The function $G(x,\xi)$ we have constructed will be used to twist the
monodromy operator $M(z,h)$, before injecting it inside a Grushin
problem. We first recall the structure (and strategy) of Grushin
problems.

\section{The Grushin problem}
\label{gp}

In this section we will construct a well posed Grushin problem for the operator 
$ I - M $, where $ M=M(z,h) $ is an abstract hyperbolic open quantum map (or monodromy 
operator) as defined in \S \ref{2gr}. We will treat successively the
untruncated operators $M$, and then the operators $\tM$ truncated by
the finite rank projector $\Pi_h$. The second case applies to the 
monodromy operators constructed as 
effective Hamiltonians for open chaotic 
systems \cite{NSZ1}, and to the operators constructed in \S\ref{assc}
to deal with obstacle scattering.

\subsection{Refresher on Grushin problems}
\label{rgp}

We recall some linear algebra facts related to the 
Schur complement formula. For 
any invertible square matrix decomposed into $4$ blocks, we have
\[  
\begin{bmatrix} a & b \\ c & d \end{bmatrix}^{-1} 
=  \begin{bmatrix} \alpha & \beta \\ \gamma & \delta \end{bmatrix} 
\ \Longrightarrow \ a^{-1} = \alpha - \beta \delta^{-1} \gamma \,, 
\]
provided that $ \delta^{-1} $ exists. As reviewed in \cite{SZ9} 
this formula can be applied to {\em Grushin problems} 
\[   \begin{bmatrix} P & R_- \\
R_+ &  \ 0 \end{bmatrix} \; : \; {\mathcal H}_1 \oplus {\mathcal H}_- 
\longrightarrow {\mathcal H}_2 \oplus {\mathcal H}_+ \,, \]
where $ P $ is the operator under investigation and $ R_\pm $ are suitably 
chosen.
When this matrix of operators is invertible we say that the 
Grushin problem is {\em well posed}.
If $  \dim \cH_- = \dim \cH_+ < \infty $, and $ P = P ( z ) $, it is
customary to write
\[  
\begin{bmatrix}  P (z)   & R_-  \\
R_+ &  0 \end{bmatrix}^{-1} = 
 \begin{bmatrix} E ( z ) & E_+ ( z ) \\
E_- (z )  &  E_{-+} ( z )  \end{bmatrix}\,, 
\]
and the invertibility of $ P( z):\cH_1\to\cH_2 $ is equivalent to the invertibility of the
finite dimensional matrix $ E_{-+} ( z ) $. 
For this reason, the latter will be called 
an {\em effective Hamiltonian}.

This connection is made more precise by the following standard result \cite[Proposition 4.1]{SZ9}:

\begin{prop}
\label{p:4}
Suppose that $ P = P ( z ) $ is a family of Fredholm operators 
depending holomorphically on $ z \in \Omega $,
where $ \Omega \subset \CC $ is a simply connected open set. 
Suppose also that
the operators $ R_\pm = 
R_\pm ( z ) $ are of finite rank, depend holomorphically on $ z \in \Omega $,
and the corresponding Grushin problem is well posed for $ z \in \Omega $.
Then for any smooth positively oriented $ \gamma  = \partial \Gamma $, 
$ \Gamma \Subset \Omega $, on 
which $ P ( z )^{-1} $ exists, the operator 
$ \int_\gamma \partial_z P ( z ) P ( z ) ^{-1} 
dz $ is of trace class and 
we have
\begin{equation}
\label{eq:p4}
\begin{split}
\frac 1 { 2 \pi i } \tr \int_\gamma  P ( z )^{-1} \partial_z P ( z ) \,
 dz & = 
\frac 1 { 2 \pi i } \tr \int_\gamma  E_{-+} (z)^{-1}  \partial_z E_{-+} ( z ) \, 
dz \\
& = \# \{ z \in \Gamma   : \det E_{-+} ( z ) = 0 \}  \,,
\end{split} 
\end{equation}
where the zeros are counted according to their multiplicities.
\end{prop}

\subsection{A well posed Grushin problem}
\label{wpp}
Let $ Y \Subset \bigsqcup_{j=1}^J \RR^d $, $ \UU \Subset T^*Y $,
$ F \subset \UU \times \UU $, be a hyperbolic Lagrangian relation, and
$ M = M(z,h) \in I_{0+} ( Y \times Y , F') $
be an associated hyperbolic quantum monodromy operator as in
Definition~\ref{def:HQMO}. In this section it will be more convenient
to see $M$ as a Fourier integral operator acting on $L^2(\RR^d)^J$. Let $ G $ be the escape
function constructed in Proposition \ref{p:es2}, and $G^w(x,hD)$ the
corresponding pseudodifferential operator.

We will construct a well posed Grushin problem for the operator 
$ P ( z ) =  I - M_{tG}( z )  $, where
\be
\label{eq:conjM}
  M_{tG}(z)  \defeq e^{- t G^w ( x , h D ) }\, M(z)\, e^{ t G^w ( x ,
    h D ) } \,,\quad t>0\,.
\ee
For this aim we will need a
finite dimensional subspace of $L^2(\RR^d)^J$ microlocally covering an $ (h/\tilde h)^{\frac12} $
neighbourhood of the trapped set $\cT$. We will construct that subspace by
using an auxiliary pseudodifferential operator. 
\begin{prop}
\label{p:q}
Let $ \Gamma \Subset T^*\RR^d $ be a compact set. For the order function
\be\label{eq:Qel0}
m ( x, \xi ) = h/ \tilde h + d ( ( x, \xi) , \Gamma )^2 \,, 
\ee
there exists
$ q \in \ST ( m ) $, so that
\be
\label{eq:Qel}  
q ( x , \xi )  \sim m(x,\xi) 
\,,  \qquad 
 \partial^\alpha q = {\mathcal O } ( q^{1 - |\alpha|/2 } ) \,, 
\ee
and such that, for $\th$ small enough, the operator $ Q \defeq q^w ( x , h D ) $ satisfies $ Q = Q^* \geq h/ 2 
\tilde h $.
\end{prop}
\begin{proof}
Let $ \varphi $ be as in Lemma \ref{l:whitney} with 
$ \epsilon = h / \th$.
We will take $ q = \varphi $, $Q=\Op(q)$. From the reality of $q$,
this operator is symmetric on $L^2(\RR^d)$. The estimates \eqref{eq:Qel} are
automatic, as well as the uniform bound $q(\rho)\geq h/\th$.
Taking into account the compactness of
$\Gamma$, the estimates \eqref{e:Whitney} easily imply that $q\in \ST(m)$.

To prove the lower bound on $Q$, we use the rescaling \eqref{eq:resc}:
\[  \tvarphi (\trho ) = (  \th /h ) 
\varphi \big( (  h / \th )^{\frac 12} \trho \big)
\,, \qquad  \widetilde \Gamma = ( \th / h )^{\frac 12} \Gamma 
\,.  \]
Our aim is to show that $\Opt(\tv)\geq 1/2$ for $\th$ small
enough. This is a form of G{\aa}rding inequality, but for an
unbounded symbol. To prove it, we draw
from \eqref{e:Whitney}
\[ \frac{ \tv ( \trh_1) } { \tv ( \trh_2) } \leq C
\frac{1 +  d ( \trh_1 , \tilde \Gamma )^2  } 
{ 1 + d ( \trh_2 , \tilde \Gamma )^2 } \leq C ( 1 + d ( \trh_1 , \trh_2 ) ^2 ) 
 \,, \]
which shows that $ \tv $ is an order function in the sense of \eqref{eq:orderm},
and $ \tv \in \tS ( \tv ) $. Similarly, the uniform bound $\tv\geq 1$ implies
that 
\[  (\tv - \lambda)^{-1}    \in \widetilde{S} ( 1 /\tv ) \,, \quad
\lambda < \frac 2 3  \,, \]
and hence for $\th$ small enough, $ \Opt(\tv) - \lambda $ is
invertible on $L^2(\RR^d)$, uniformly for $ \lambda \leq 1/2 $
As a result, for $\th$ small enough we have the bound
\[ \Opt ( \tv) \geq 1/2\,.\]


Recalling that 
\[  \Op( \varphi)  =   ( h / \th ) \, U_{ h , \th}^{-1} \,
 \Opt ( \tv ) \, U_{ h , \th} \,, \]
for $\th$ small enough we obtain the stated lower bound on $Q=\Op(q)=\Op(\varphi)$.
\end{proof}

The above construction can be straightforwardly extended to the
case where $ \Gamma = \TT \Subset \UU \Subset (T^*\RR^d)^J$: the
symbol $q=(q_j)_{j=1,\ldots,J}$ is then a vector of symbols $q_j\in \ST(m)$,
and similarly $Q=(Q_j)_{j=1,\ldots,J}$ is an operator on
$L^2(\RR^d)^J$.

Let $ K \gg 1 $ be a constant to be chosen later. We define the following 
finite dimensional Hilbert space and the corresponding orthogonal projection:
\be\label{e:spaceV}
 V \defeq 
 \bigoplus_{j=1}^J V_j,\quad V_j \defeq \bbbone_{ Q_j \leq K h/\th } L^2 ( \RR^d ) \,,
 \qquad
\Pi_V \defeq {\rm diag}(\bbbone_{ Q_j \leq K h/\th }) : L^2 ( \RR^d )^J 
\stackrel{\perp} {\longrightarrow} V \,.
\ee
Due to the ellipticity of $ Q $ away from $ \TT $,
see \eqref{eq:Qel0}, we have
\[  \dim V < \infty \,. \]
Lemma \ref{l:dim} below will give a more precise statement.

The operators $ R_\pm $ to inject in the Grushin problem are then defined as follows:
\begin{gather}
\label{eq:2Rpm}
\begin{gathered}
R_+ =  \Pi_V  \; : \; L^2 ( \RR^d )^J \longrightarrow 
 V  \,, \qquad
R_- \; : \; V \hookrightarrow L^2 ( \RR^d )^J \,.
\end{gathered}
\end{gather}

Before stating the result about the well posed Grushin problem for 
the operator $ M_{tG} $, we prove a crucial lemma based on the 
analysis in previous sections:
\begin{lem}
\label{l:prG}
Let $ M_{t G } ( z ) $ be given by \eqref{eq:conjM}, $ z \in \Omega (
h ) $, where $ \Omega ( h ) $ is given in \eqref{eq:omegah}. Then for $ \Pi_V $ given 
above, and any $ \varepsilon > 0 $, there exists $ t = t ( \varepsilon ) $, 
sufficiently large, and $ \th = \th (t, \vareps ) $, 
sufficiently small (independently of $ h \rightarrow 0 $) so that,
provided $N_0$ from \eqref{eq:Mlo} and $C_6$ from \eqref{e:C_6} satisfy
$N_0 > 4 C_6 t$, then for $h>0$ small enough one has
\begin{equation}
\label{eq:prG}
   \| ( I - \Pi_V ) M_{tG}  \|_{L^2( \RR^d )^J \rightarrow L^2 ( \RR^d )^J } < \varepsilon  \,, \ \ 
  \| M_{tG} ( I - \Pi_V ) \|_{L^2( \RR^d )^J \rightarrow L^2 ( \RR^d )^J } < \varepsilon \,.
\end{equation} 
\end{lem}
\begin{proof}
Since $ M ( z ) $ is uniformly bounded for $ z \in \Omega ( h ) $ --
see \eqref{e:decayM} -- we will work with a fixed $ z $ and 
normalize $ \| M ( z )  \|_{L^2 \rightarrow L^2 } $ to be $ 1 $. That
is done purely for notational convenience.

First, we can choose $ \psi \in \CIc ( [0, 3K/4 ), [0,1] ) $,
$\psi\rest_{[0,K/2]}\equiv 1$, so that 
\[  
\Pi_ V \, \psi  ( \th  Q  /h ) 
= \psi  ( {\th Q } /h )  \,, \qquad
\psi  ( {\th  Q } /h )   \in \PT \,,
\] 
and hence, 
\[ \begin{split} \|  M_{t G} ( I - \Pi_V ) \|_{L^2 \rightarrow L^2}  &= 
\|  M_{t G} ( I - \psi ( \tilde h Q / h ) )  ( I - \Pi_V ) 
\|_{L^2 \rightarrow L^2} \\
& \leq 
\| M_{tG }  ( I - \psi ( \tilde h Q / h ) ) 
\|_{L^2 \rightarrow L^2} \,. \end{split} \]
All we need to prove then is  \eqref{eq:prG} with $ \Pi_V $ 
replaced by the smooth cutoff $ \psi( \th Q / h ) $, which now puts the 
problem in the setting of \S \ref{fio}.

Before applying the weights, we split $M=MA_M + M(1-A_M)$, using the
cutoff $A_M$ of Definition~\ref{def:HQMO}. We then apply
Proposition~\ref{p:fiom-global} to the Fourier integral operator
 $T=MA_M$, with the function $
a_0\equiv 1 $ 
and replacing $G$ with $-tG$. We then get
\begin{equation}
\label{eq:fioM} 
\begin{split}
M_{tG}  = (MA_M)_{tG} + (M(I-A_M))_{tG} & 
= MA_M\, e^{ - t (F^*G)^w ( x, h D ) }\, e^{t G ^w ( x, h D ) } \\ 
& 
+ M_1 \,h^{\frac12( 1 - \delta) } \th ^{\frac 12} \PT ( e^{ t (G - F^* G )} )\\
& \ \ \ + \Oo_{L^2 \to L^2 } ( h^{N_0 - 4\,C_6 t } ) \,,
\end{split}
\end{equation}
where $ M_1 \in I_{0+} (  Y \times Y , F' ) $,
$\WFh(M_1)\subset F(\supp a_M)\times \supp a_M$.
 
The error term $\Oo_{L^2 \to L^2 } ( h^{N_0 - 4\,C_6 t } ) $
corresponds to $\norm{(M(I-A_M))_{tG}}$, it comes from \eqref{eq:Mlo}
together with the bound
\be\label{e:C_6}
e^{\pm t G^w ( x , h D ) } = \Oo_{L^2 \to L^2 } ( h^{-t C_6
  (1+\Oo(\th))} ) = \Oo_{L^2 \to L^2 } ( h^{- 2 t C_6} )\,,\quad \th\
\text{small enough}\,,
\ee
due to the first property in  \eqref{eq:es3}. 

By contrast, the second line in \eqref{eq:es3} shows that
$\exp(-t ( F^* G - G )(\rho)) \leq e^{t\,C_7}$ for $\rho\in\cW_3$.
We have assumed that  $\cW_3\Supset \supp a_M$: this implies
that for any $t\geq 0$, we have 
$$
\| A_M\,e^{ - t (F^*G)^w}\, e^{t G^w} \|_{L^2\to L^2},\ \ \ \| M_1 \Op(e^{ t (G - F^* G )}) \|_{L^2\to L^2}\leq C\,e^{2C_7 t} \quad \text{uniformly in $h $.}
$$ 
This is a crucial application of Proposition~\ref{p:bbc} and the 
properties of the escape function.

It remains to estimate the norm of 
\begin{gather*}    
MA_M\,e^{ - t  (F^*G)^w }\,e^{t G^w } ( I - \psi ( \th Q / h )) = 
M\,b_t ( x , h D )\,, \quad  b_t \in \ST(1) \,, \\
b_t = a_M\,e^{-t ( F^* G  - G )} ( 1 - \psi ( \th q  / h ) ) + \th\, \ST(1) \,, 
\end{gather*}
where $ q $ is as in \eqref{eq:Qel} and $a_M$ is the symbol of $A_M$.

Fixing $ \varepsilon > 0 $, we first choose $ t $ large enough so that
$$
e^{-t/C_1}<\vareps/4\,
$$ 
where $C_1$ is the constant appearing in \eqref{eq:es3}. 

At this point, we can select the constant $K>1$ in the definition
\eqref{e:spaceV}. In view of the estimates \eqref{eq:Qel}, we 
choose it large enough, so that 
\begin{align}
q(\rho)\geq \frac K 2(h/\th)\,&\Longrightarrow d(\rho,\TT)^2\geq
2\,C_1\, (h/\th)\,,\nonumber\\
\rho\in \pi_R(F),\ q\circ F(\rho)\geq \frac K 2(h/\th)\,&\Longrightarrow d(\rho,\TT)^2\geq
2\,C_1\, (h/\th)\,.\label{e:largeK}
\end{align}
As a consequence, all points $ \rho\in\supp(1-\psi(\th q/h))$ satisfy
$d(\rho,\TT)^2\geq 2\,C_1 (h/\th)$, and therefore 
$$
\forall \rho\in \cW_3,\quad
e^{ - t ( F^*G - G )(\rho) }(1-\psi(\th q(\rho)/h))\leq
\,e^{-t/C_1}<\vareps /4\,.
$$
It then follows, see 
Lemma~\ref{l:stan1}, that 
\[  
\| M\, b^w_t ( x , h D ) \|_{L^2\to L^2} \leq  \frac \varepsilon 2  + C(t) \th\,. 
\]
Altogether, we have obtained 
\[  
\begin{split} 
\|  M_{t G} ( I - \Pi_V ) \|_{L^2 \rightarrow L^2}  & \leq
\| M b_t ( x, h D ) \|_{L^2 \rightarrow L^2}   + \Oo ( h^{\frac 12}\th ^{\frac 12} ) + \Oo(h^{N_0 - 4 C_6 t})\\
& \leq \varepsilon/2 + \Oo_t ( \th) + \Oo ( h^{\frac 12}\tilde h ^{\frac 12} ) + \Oo(h^{N_0 - 4 C_6 t}) \,. 
\end{split} \]
The assumption
$N_0 > 4 C_6 t$ ensures that,
once we take $ \th < \th_0 ( t , \varepsilon ) $ and take $h>0$ small
enough, the above right hand side is $<\vareps$.

A similar proof provides the estimate for $ ( I - \Pi_V ) M_{tG} $,
replacing $\Pi_V$ by $\psi(\th Q/h)$, 
using Proposition~\ref{p:fiom-global} to bring $\psi(\th Q/h)$ to the
right of $M$, and using the assumption \eqref{e:largeK} to bound from
above $(1-\psi(\th q\circ F/h)) e^{-t(F^*G - G)}$.
\end{proof}

The invertibility of the Grushin problem is now a matter of linear algebra:
\begin{thm}
\label{t:2gr}
Suppose that $ M=M(z) $ is a hyperbolic quantum monodromy operator
in the sense of Definition~\ref{def:HQMO}, and $ M_{tG} $ the
conjugated operator \eqref{eq:conjM}.
We use the auxiliary operators of \eqref{eq:2Rpm} to define the
Grushin problem
\be\label{eq:gr2gr} 
\cM_{tG} \defeq \left( \begin{array}{ll} 
I - M_{tG} & R_- \\
\ \ R_+ &  \ 0 \end{array} \right) \; : \; L^2 ( \RR^d )^J \oplus V
\longrightarrow L^2 ( \RR^d )^J \oplus V\,, 
\ee
If $ t $ is large enough, $ \th $ is small enough, and $ N_0 $ from \eqref{eq:Mlo},
$C_6$ from \eqref{e:C_6} satisfy $ N_0 > 4 C_6 t $, then for
$h$ small enough the above Grushin problem
is invertible. Its inverse, $\cE_{tG} $, is
uniformly bounded as $ h \rightarrow 0 $, and the effective
Hamiltonian reads
\begin{equation}
\label{eq:emp}   E_{-+} =  - ( I_V - \Pi_V M_{tG} ) 
+ \sum_{k=1}^\infty \Pi_V   M_{tG}  [ ( I - \Pi_V )  M_{tG} ]^k 
\; : \; V \longrightarrow V \,,
\end{equation}
where the convergence of the series is guaranteed by \eqref{eq:prG}.
\end{thm}
\begin{proof}
We first construct an approximate inverse,
\[  
\cE^0_{tG} \defeq \left( \begin{array}{ll} \Id - \Pi_V & 
\ \ \ \ \ \ \ R_- 
\\ \ \ \ \ \Pi_V & \ - ( I - \Pi_V M_{tG} ) \end{array} \right)\,,\]
for which we check that
\begin{equation}
\label{eq:apE}  \cM_{tG} \, {\mathcal E}_{tG}^0 = \Id_{ L^2 ( \RR^d )^J \oplus V } -
\left( \begin{array}{ll} M_{tG} ( I - \Pi_V ) & ( I - \Pi_V ) M_{tG} \\
\ \ \  \  \ 0 & \ \ \ \ \ 0 \end{array} \right) \,. 
\end{equation}
The theorem now follows from Lemma \ref{l:prG} and the Neumann series 
inversion:
\begin{gather*}
  \cM_{tG} \, \cE_{tG}= \Id_{ L^2 ( \RR^d )^J \oplus V  } \,, \\ 
 \cE_{tG}  = {\mathcal E}_{tG}^0 
\left( \begin{array}{ll} \Id +  R  & ( \Id + R  ) 
( \Id - \Pi_V ) M_{tG } \\
 \ \ \ 0 & \ \ \ \ \ \ \ \ \Id_V \end{array} \right) \,,  \ \ 
R \defeq \sum_{k=1}  [ M_{tG}  ( \Id - \Pi_V )] ^k \,.
\end{gather*}
We finally show that the operator $\Pi_VM_{tG}$, and thus the whole
inverse $\cE_{tG}$, is bounded uniformly
in $h$. Consider a cutoff $\psi_1\in \CIc(T^*Y) $ supported
inside a small neighbourhood of $\cT$, equal to unity in a smaller
neighbourhood of $\cT$.
In particular, using the notations of Proposition~\ref{p:es2} we
assume that $\supp F^*\psi_1 \Subset \cW_2$. Since $\Pi_V$ is
microlocalized in a semiclassically thin neighbourhood of $\cT$, we
have
$$
\Pi_V \psi^w_1(x,hD) = \Pi_V + \Oo_{L^2\to L^2}(h^\infty)\,.
$$ 
We are then lead to estimate the norm of the operator $\psi^w_1
M_{tG}$. Using the decomposition \eqref{eq:fioM} and the fact that
$(F^*G-G)(\rho)\geq -C_7$ for $\rho\in\cW_2$, we obtain for $\th$, $h$ small enough:
\be\label{e:boundM}
\forall t>0,\quad
\| \psi_1^w\,M_{tG}\|_{L^2\to L^2}\leq\,e^{3C_7t}\Longrightarrow \| \Pi_V\,M_{tG}\|_{L^2\to L^2}\leq C'\,e^{3C_7t}\,.
\ee
\end{proof}

This achieves the reduction of the monodromy operator $M(z,h)$ to the finite rank operator $E_{-+}(z)$.
In the next section we perform the same task, starting from a
monodromy operator $\tM(z,h)$ which is already of finite rank.

\subsection{Quantum monodromy operators acting on finite dimensional
spaces}
\label{qfd}
So far we have been considering quantum maps or monodromy operators
given by smooth $h$-Fourier integral operators of infinite 
rank. The quantum monodromy operator constructed in \cite{NSZ1} 
and providing an effective Hamiltonian for operators
$ (i/h) P -z $, $ z \in D ( 0 , R ) $, 
was given by the restriction of such a Fourier integral 
operator to a finite dimensional space $W\subset L^2(\RR^d)^J$
microlocalized on some bounded neighbourhood of the trapped set.
\begin{gather}
\label{eq:tEmp}
\begin{gathered}
 W \subset L^2 ( \RR^d )^J \,, \quad \dim W < \infty \,, 
\quad \Pi_W   :  L^2 ( \RR^d )^J  \stackrel{\perp}{\longrightarrow} W \,,
\\
 M_W = \Pi_W M \Pi_W + \Oo_{W \rightarrow W }( h^{N_0}  )  \,, 
\\\end{gathered}
\end{gather}
where $ M\in I_{0+}(Y\times Y,F')$ is a smooth Fourier integral operator. Compared with the
notations in Definition~\ref{def:HQMO}, we have
$\Pi_W=\Pi_h$ and $M_W=\tM + \Oo_{W\to W}(h^{N_0})$. 

Let us now consider the Grushin problem for $(\Id_W-M_W)$.
\begin{thm}
\label{t:2gr2}
Suppose that $ M_W $ is given by \eqref{eq:tEmp}, with 
$M$ a monodromy operator as defined in Definition~\ref{def:HQMO}. The space
$V$, and the auxiliary operators $ R_\pm $ are as in
Theorem \ref{t:2gr}. We construct a weight $ G $ as in
Prop.~\ref{p:es2}, and such that $\Pi_W\equiv \Id $
near $\supp G$.


If $ t $ is large enough, $ \th $ is small enough, and $ N_0 $
satisfies $ N_0 > 4 C_6 t $, with $C_6$ from \eqref{e:C_6}, then the operator 
\[ 
\tcM_{tG}\defeq \left( \begin{array}{ll} \ \ \ \  \ \ \ I_W - M_W   & 
\Pi_W e^{ t G^w ( x, h D ) } R_- \\
\ & \ \\
\ \ R_+ e^{- tG^w ( x , h D ) } \Pi_W  & \ \ \   \ \ \ \  \
0 \end{array} \right)
 \; : \; W \oplus V \; \longrightarrow \; 
W \oplus V\,, \]
is invertible, with the inverse
\[
 \widetilde{\cE}_{tG} = 
\left( \begin{array}{ll} \ \widetilde{E} & \ \widetilde{E}_+ \\ 
\widetilde{E}_- & \widetilde{E}_{-+} \end{array} \right) = 
\Oo ( h^{-4tC_6 } ) \; : \; W \oplus V \; \longrightarrow \;
W \oplus V \,. 
\]
Furthermore, the effective Hamiltonian is uniformly bounded: 
\[  
 \| \widetilde{E}_{-+}\|_{V\to V}=\Oo_{\vareps}(1)\,. 
\]
\end{thm}


\begin{proof}
We start by proving three estimates showing that the projector $\Pi_W$
does not  interfere too much with the Grushin problem. 
Firstly, in the Definition~\ref{def:HQMO}
we have assumed that $\Pi_W$ is equal to the identity, microlocally near the support of $a_M$:
this has for consequence the estimate
\be\label{e:PiW-A_M}
A_M  = \Pi_W A_M + \Oo_{L^2 \to L^2 } ( h^{\infty}) 
= A_M \Pi_W + \Oo_{L^2 \to L^2 } ( h^{\infty} ) \,.
\ee
Secondly,
let us notice that in Proposition~\ref{p:es2} we
required the weight to satisfy $\supp a_M\Subset \cW_3\Subset \supp
G$. Since we also know that $\Pi_W\equiv \Id$ some neighbourhood $\cW_4$
of $\supp a_M$, it is indeed possible construct the weight $G$ such that
$\cW_3\Subset \supp G \Subset \cW_4$. This has for consequence the
estimate
\be\label{e:PiW-G}
\forall t\geq 0,\qquad
e^{-tG^w ( x, h D )}\, \Pi_W\, e^{ tG^w ( x, h D ) } = \Pi_W + \Oo_{L^2 \to L^2 } ( h^{\infty}) \,.
\ee
(the fact that we are dealing with $ \ST $ symbol classes does
not affect the result, see for instance \cite[Theorem 4.24]{EZB}).
Thirdly, the definition of the subspace $V$ in \eqref{e:spaceV} shows
that the projector $\Pi_V$ is microlocalized inside an $h^{1/2}$
neighbourhood of $\cT$, while $\Pi_W\equiv \Id$ in a fixed
neighbourhood. This induces the estimate
\be\label{e:PiW-PiV}
 \Pi_V \Pi_W = \Pi_V    + \Oo_{L^2 \to L^2 } ( h^{\infty}) \,.
\ee
We now want to solve the problem
$$
\tcM _{tG}  \left( \begin{array}{l} u \\ u_- \end{array} \right) = 
 \left( \begin{array}{l} v \\ v_+ \end{array} \right)  \,.
$$
We first consider the approximate solution
$$
\begin{pmatrix} u^0\\u_-^0\end{pmatrix} = \cE^0_{tG} 
\begin{pmatrix} v\\v_+\end{pmatrix},\qquad
\cE_{tG}^0\defeq \begin{pmatrix}e^{tG^w}& 0\\0&\Id_V\end{pmatrix}\cE_{tG} 
\begin{pmatrix}e^{-tG^w}& 0\\0&\Id_V\end{pmatrix}\,,
$$
where 
$ \cE_{tG} $ is the inverse of the Grushin problem in 
Theorem \ref{t:2gr}. In particular,
\[  ( I - M ) u^0 + e^{t G^w ( x , h D ) } R_- u_-^0 = v \,. \]
The estimate \eqref{e:PiW-A_M} implies that 
\[   \Pi_W M \Pi_W   =  M + \Oo_{ L^2 \to L^2 }  ( h^{N_0} ) \,,\]
and hence ($  v  = \Pi_W v $ as $ v $ is assumed to be in $ W $),
\be
\label{eq:te1} (\Id_W - M_W ) \Pi_W u^0 + \Pi_W e^{ t G^w ( x, h D) } u_-^0 = v 
+ \Oo_W ( h^{N_0} \| u^0 \|_{L^2}  ) \,. 
\end{equation}
Since $ R_+ e^{ -t G^w ( x , h D ) } u^0 = v_+ $, the definition of $
R_+ = \Pi_V $ and the estimates \eqref{e:PiW-G}, \eqref{e:PiW-PiV} imply that
\be\label{eq:te2}
\begin{split} 
  R_+ e^{-t G^w ( x, h D) } \Pi_W u^0  & = \Pi_V \left( e^{-t G^w (x , h D ) } 
\Pi_W e^{ t G^w ( x , h D ) } \right) e^{-t G^w ( x, h D ) } u^0 \\
& = 
v_+ + {\mathcal O}_V ( h^{N_0} \| e^{ - t G^w ( x, h D ) } u^0\|_{L^2 } ) \,. 
\end{split}
\ee
Since $ \cE_{tG} $ is uniformly bounded, we obtain the bound
\[  
\begin{split} \|e^{-tG^w} u^0 \|_{L^2} & 
\leq \norm{\cE_{tG}}\,\big(\norm{e^{-tG^w}v} + \norm{v_+} \big)\\
& \leq C \, (h^{-2 C_6 t} \,\| v \|_W + \|v_+\|_V ) \\
\Longrightarrow 
\|u^0 \|_{L^2}  & \leq 
C \| e^{  t G^w  } \|_{L^2 \rightarrow L^2 } \|e^{-tG^w} u^0 \|_{L^2}\\
& \leq C  ( h^{-4 C_6 t} \| v \|_W + h^{-2 C_6 t} \|v_+\|_V ) \,.
\end{split}
\]
This bound, together with \eqref{eq:te1} and \eqref{eq:te2}, gives
\[ \tcM _{tG}  \begin{pmatrix} \Pi_W u^0 \\ u_-^0 \end{pmatrix} = 
\begin{pmatrix} \Id_{W} + \Oo(h^{N_0 - 4 C_6 t })& \Oo(h^{N_0 - 2 C_6 t })\\\Oo(h^{N_0 - 2 C_6 t }) &
\Id_V +  \Oo(h^{N_0 })\end{pmatrix}
\begin{pmatrix} v \\ v_+ \end{pmatrix} \,. 
\]
The assumption $ N_0 > 4C_6 t $ implies that the operator on the right hand side
can be inverted for $ h $ small enough, and proves the existence of 
$$ 
\widetilde{\cE}_{tG}
= \begin{pmatrix}\Pi_W&0\\0&\Id_V\end{pmatrix}\cE_{tG}^0 \begin{pmatrix}
  \Id_{W} + \Oo(h^{N_0 - 4 C_6 t })& \Oo(h^{N_0 - 2 C_6 t })\\ \Oo(h^{N_0 - 2 C_6 t }) &
\Id_V +  \Oo(h^{N_0 })\end{pmatrix}\,.
$$ 
From this expression, we deduce the estimate for $\|\widetilde{\cE}_{tG}\|$, as well as the uniform boundedness of 
\be\label{e:E-tildeE}
\widetilde{E}_{-+}= E_{-+} + \Oo_{V\to V}(h^{N_0 - 4C_6 t})\,.
\ee
\end{proof}

\noindent{\bf Remarks.} 1)  The projector $ \Pi_W $ is typically 
obtained by taking $ \Pi_W = \diag(\Pi_{W,j})$, $\Pi_{W,j}=\bbbone_{Q_{0,j} \leq C }$ where $
Q_{0,j}=\Op(q_{0,j})\in \Psi(\RR^d) $ is
elliptic. This way,
$\Pi_W$ is a microlocal projector associated with the compact region 
$$
\cW_5 \defeq \sqcup_{j=1}^J\{ \rho \in T^*\RR^d \; : \; q_{0,j} ( \rho) \leq C \}\,.
$$
Then the estimates \eqref{e:PiW-A_M}, \eqref{e:PiW-G} hold if
$ \supp a_M  \Subset \supp G  
 \Subset\, \cW_5$.

\medskip

\noindent
2) The requirement that $ N_0 > 4C_6 t $, where the constant $ C_6 $ depends on $ G $, and
$ t= t(\vareps) $ has to be chosen large enough, seems awkward in the abstract
setting ($N_0$ is the power appearing in \eqref{eq:Mlo}). In 
practice, when constructing the monodromy operator $M$ we can take $N_0$
arbitrary large, independently of the weight $G$ (see
\cite{NSZ1}, or the application presented in \S \ref{assc}).

\subsection{Upper bounds on the number of resonances}
\label{upbd}

Let us first recall the definition of the {\em box} or {\em
Minkowski} dimension of a compact subset $ \Gamma \Subset \RR^{k} $:
\be\label{eq:dimm}
\dim_M\Gamma = 2 \mu_0 \defeq k  - \sup\{ \gamma \; : \; \limsup_{ \epsilon \rightarrow 0 }
\epsilon^{-\gamma} \vol ( \{ \rho \in \RR^k \; : \; d ( \rho , \Gamma ) <
\epsilon \} ) < \infty \} \,.
\ee
The set $\Gamma$ is said to be of pure dimension if
\[  \limsup_{ \epsilon \rightarrow 0 }
\epsilon^{-k + 2\mu_0}
\vol ( \{ \rho \in \RR^k \; : \; d ( \rho , \Gamma ) <
\epsilon \} ) < \infty \,. \]
In other words, for $ \epsilon $ small
\[  \vol ( \{ \rho \in \RR^k \; : \; d ( \rho , \Gamma ) <
\epsilon \} ) \leq C \epsilon^{k - 2\mu} \,, \ \ \mu > \mu_0 \,, \]
with $ \mu $ replaceable by $ \mu _0 $ when $ \Gamma $ is of pure
dimension.

In the case of a compact set $\Gamma=\sqcup \Gamma_j\Subset \sqcup_{j=1}^J\RR^k$, its 
Minkowski dimension is simply
\be\label{e:dimJ}
\dim_M \Gamma=\max_{j=1,\ldots,J}\dim_M \Gamma_j\,.
\ee

The following lemma expresses the intuitive idea that a domain in $T^*\RR^d$ of
symplectic volume $ v $ can support at most $ h^{-d} \, v $ quantum states.
\begin{lem}
\label{l:dim} Let $\Gamma$ be a compact subset of $T^*\RR^d$, of Minkowski
dimension $2\mu_0$. Using the operator  $ Q $ constructed in
Proposition~\ref{p:q}, we take $K\gg 1$ and define the subspace 
\[ 
V \defeq  \bbbone_{ Q \leq K h/\tilde h } L^2 ( \RR^{d } )  \,. 
\]

Then, for any $ \mu > \mu_0 $ there exists $ C = C_\mu $, such that
\begin{equation}
\label{eq:ldim}
\dim V \leq C  \th ^{-d } \left( \frac{ \th } h \right)^{\mu } \,.
\end{equation}
When $ \Gamma $ is of pure dimension we can take $ \mu = \mu_0 $ in 
\eqref{eq:ldim}.
\end{lem}
\begin{proof}
Since the order function $ m (\rho ) \rightarrow \infty $ as $ |\rho|
\rightarrow \infty $ and
$ Q \in \PT (m ) $, the selfadjoint operator $ Q $ has 
a discrete spectrum, hence $ V$ is finite dimensional, 
and 
\[ \dim V = \# \{ \lambda \leq K h/ \tilde h \; : \; 
\lambda \in \Spec ( Q ) \} \,. 
\]
%
%


The usual min-max arguments --- see for instance
\cite{SjDuke} or \cite[Theorem C.11]{EZB} --- show that
 $ \dim V \leq N $ if there exists $\delta>0$ and an operator $ A $ of 
rank less than or equal to $ N $, such that
\begin{equation}
\label{eq:minmax}
  \langle Q u , u \rangle + \Re \langle A u , u \rangle  \geq ( K h /
  \tilde h +\delta )\| u\|^2 \,, \qquad u \in \CIc ( \RR^d ) \,. 
\end{equation}
To construct $ A $, take 
$ a = \psi ( \tilde h q / h ) $, where $ \psi \in \CIc ( \RR ) $,
$ \psi ( t ) \equiv 1 $ for $ |t| \leq 3 K $ and $ \psi ( t ) = 0 $ 
for $ |t| \geq 4 K$.
At the symbolic level, $ a \in \ST $ and $a(\rho)=1$ in the region
where $q(\rho)\leq 3Kh/\th$. Taking into account the fact that $q\geq
h/\th$ everywhere, we have
$$
q(\rho) + 2Kh a(\rho)/\th\geq (2K+1)h/\th,\qquad \rho\in T^*\RR^d\,.
$$
The arguments presented in the proof
of Proposition \ref{p:q} show that, at the operator level, we have for
$\th$ small enough
\be\label{eq:Qa}  
Q +  A_0  \geq  
{2 K h} / {\tilde h} \,,\qquad\text{for}\quad  A_0 \defeq 2 K h a^w ( x, h D ) / \th\,.
\ee
This inequality obviously implies \eqref{eq:minmax}, with $A$ replaced
by the (selfadjoint) operator $A_0$. 
Our task is thus to replace $ A_0 $ in \eqref{eq:minmax} by a finite
rank operator.
We do that as in \cite[Proposition 5.10]{SZ10}, by using 
a {\em locally finite} open covering of a neighbourhood of $\TT$:
\begin{gather*}
W_{h , \th} \defeq \{\rho \; : \; 
d ( \rho , \Gamma)^2 \leq 4K h/ \th \} \subset \bigcup_{ k=1}^{ N ( h,
\th  ) } U_k \,,  \ \ 
 {\rm{diam}} \,( U_k ) \leq (h/\th)^{\frac12} \,.
  \end{gather*}
The definition of the box dimension implies that we can choose this covering
to be of cardinalinty
\begin{equation}
\label{eq:dimK}
N ( h , \th ) \leq C_{K,\mu} ( {\th} / h )^{\mu } \,,
\end{equation}
for any $ \mu > \mu_0 $, and for $ \mu = \mu_0 $ if $ \Gamma $ is of 
pure dimension.

To the cover $ \{ U_k \} $, 
we associate a partition of unity on $ W_{h, \th} $,
\[ \sum_{ k=1}^{N ( h , \th ) }
 \chi_k = 1 \ \ \text{ on $ W_{h, \th} $,} \quad
\supp \chi_k \subset U_k \,, \quad  \chi_k \in \ST \,, \]
where all seminorms are assumed to be uniform  with respect to $k$.
The condition on the support of $ \psi $ in the definition of 
$ a $ and the pseudodifferential calculus in Lemma \ref{l:Sjnew} show that
\be\label{e:1-chi}  
\big( \Id - \sum_{k=1}^{N( h , \th ) } \chi_k^w ( x, h D ) \big) 
 a^w ( x, h D ) \in  \tilde h^\infty \ST \,.
\ee
Hence it suffices to show that for each $ k=1,\ldots,N(h,\th) $, there exists an operator $ R_k $ 
such that 
\[  \chi_k^w( x , h D ) a^w( x , h D)  - R_k \in \th^\infty 
\PT \,,  
 \quad \rank( R_k ) \leq C'   \th^{-d} \,,
\] 
with $ C' $ and the implied constants independent of $ k $. We can assume that, for some
$\rho^k=(x^k,\xi^k)\in T^*\RR^d$,
\[ 
U_k \subset B_{\RR^{2d} } ( \rho^k , ( h/\th )^{\frac 12} ) \,.
\]
Then consider the harmonic oscillator shifted to the point $\rho^k$:
\begin{gather*}
H^k \defeq \sum_{i=1}^d ( hD_{x_i}-\xi_i^k )^2 + (x_i-x_i^k)^2 \,. 
\end{gather*} 
 If $ \psi_0 \in \CIc ( \RR ) $, 
$ \psi_0 ( t) = 1 $ for $ |t| \leq  2 $, $ \psi_0 ( t)  = 0 $ for $ |t| \geq 3  $,
then $  \psi_0 ( \th H^k / h  ) \in \PT$, and 
\be\label{R_k}
  \psi_0 ( \th H^k / h ) \chi_k^w ( x, h D )  a^w ( x, h D ) -  
\chi_k^w ( x, h D )  a^w ( x, h D ) = \cR_k,\qquad  \cR_k\in  \th^\infty \PT \,,
\ee
where the implied constants are uniform with respect to $k$.
The properties of the harmonic oscillator (see for 
instance \cite[\S 6.1]{EZB}) 
show that $   \psi_0 ( \th H^k / h ) $ is a finite
rank operator, with rank bounded by $ C_d\tilde h^{-d} $.
Hence for each $k$ we can take 
\[ R_k \defeq  \psi_0 ( \th H^k / h ) \chi_k^w ( x, h D )  a^w ( x, h D
)\,, \]
 and define
\be
\label{eq:Adef} A \defeq 2K ( h / \th) \sum_{ k=1}^{N( h, \tilde h )} 
R_k \,, \qquad  \rank ( A ) \leq C_d \th^{-d} N( h , \tilde h ) \,. 
\ee
The remainder operators $\cR_k$ in \eqref{R_k} satisfy, for any $M>0$, 
\begin{gather*}
\|\cR_k^*\cR_{k'}\|_{L^2\to L^2}\,,\ \ \|\cR_k\cR_{k'}^*\|_{L^2\to
  L^2} \leq C_M\th^M\, \left\la \frac{d(\rho^k,\rho^{k'})}{\th}\right\ra^{-M} \\ 
\quad\text{uniformly for }k,k'=1,\ldots,N(h, \th)\,.
\end{gather*}
Since the supports of the $ \chi_k$'s form a locally finite partition,
each remainder $\cR_k$ effectively interferes with only finitely many other
remainders. One can then invoke the Cotlar-Stein lemma (see \cite[Lemma
7.10]{DiSj})  to show that
$$
\|\sum_{k=1}^{N(h,\th)}\cR_k\|_{L^2\to L^2}=\Oo( \th^\infty)\,.
$$
Using also \eqref{e:1-chi}, we obtain
\[  
A = A_0 + \Oo_{L^2\to L^2}(h\th^\infty) \,.
\]
Consequently, for $\th$ small enough the estimate \eqref{eq:minmax}
holds with $\delta=h/\th$. In view of \eqref{eq:dimK} and
\eqref{eq:Adef} the bound we have obtained on the rank of $ A $ leads to
\eqref{eq:ldim}.
\end{proof}

We now consider a monodromy operator as defined in Definition~\ref{def:HQMO}:
\be
\label{eq:Mz}
 \Omega(h) \ni z \longmapsto M ( z ,h ) \in I^0_+ ( Y \times Y , F' ) \,, 
\ee
where the depends on  $ z $ is holomorphic.

The decay assumption \eqref{e:decayM} implies that there exists $R_0>0$
such that, for $h$ small enough,
\be
\label{eq:Mmer}
z\in\Omega(h),\ \Re z \leq - R_0 \ \Longrightarrow \ 
\| M ( z ,h ) \|_{L^2\to L^2} \leq 1/2  \,.
\ee
Using the analytic Fredholm theory (see for instance
\cite[\S 2]{SZ9}), this implies that
$ ( I - M ( z ,h ) )^{-1} $ is meromorphic in $ \Omega(h) $ with 
poles of finite rank. The multiplicities of the poles are defined by the
usual formula:
\begin{gather}
\label{eq:Mmul}
\begin{gathered}  m_M ( z ) \defeq 
\inf_{\epsilon > 0 } \frac 1 { 2 \pi i }  \tr \oint_{\gamma_\epsilon ( z )  } 
( I - M( \zeta ) )^{-1}\partial_\zeta M ( \zeta )   d\zeta \,, \\
 \gamma_\epsilon ( z )  : t \mapsto z + \epsilon e^{2 \pi it } \,, \ \ t \in 
[ 0 , 2 \pi ) \,, 
\end{gathered} 
\end{gather}
see Lemma \ref{l:simtr} below for the standard justification of 
taking the trace.


\begin{thm}
\label{th:bound}
Suppose that $ \{M ( z,h ),\, z \in \Omega(h)\} $, is a hyperbolic
quantum monodromy operator, or its truncated version $\tM(z,h)$,  in
the sense of Definition~\ref{def:HQMO}, and that $ \TT$ is the 
trapped set for the associated open relation $ F $. 
Let $ 2\mu_0 $ be the Minkowski dimension of $ \TT$, as defined by (\ref{eq:dimm},\ref{e:dimJ}),
with $ k = 2d $ and $ \Gamma = \TT $.

Then for any $R_1>0$ and any $ \mu > \mu_0 $, there exists $K_{\mu,R_0} $, such that
\begin{equation}
\label{eq:Mdim}
\sum_{ z \in D ( 0 , R_1 ) }  m_M ( z ) 
 \leq K_{\mu,R_1}\,  h ^{-\mu } \,,\qquad h\to 0\,.
\end{equation}
When $ \TT $ is of pure dimension we can take $ \mu = \mu_0 $ in 
the above equation.
\end{thm}
\begin{proof} 
Let us treat the case of the untruncated monodromy
operator $M(z,h)$, the case of the truncated one being similar. 
We apply Theorem \ref{t:2gr} to the family $ M ( z,h ) $. Since
the construction of the Grushin problem \eqref{eq:gr2gr} depends only 
on the relation $ F $ and
the estimates \eqref{eq:Mlo}, we see that for any radius $R>0$, the
Grushin problem
\begin{gather}
\label{eq:gr2gr'} 
\begin{gathered} {\mathcal M}_{tG}(z) \defeq \left( \begin{array}{ll} 
I - M_{tG} ( z )  & R_- \\
\ \ R_+ &  \ 0 \end{array} \right) \; : \; L^2 ( \RR^d )^J \oplus V
\longrightarrow L^2 ( \RR^d )^J \oplus V\,,  \\
M_{tG } ( z ) \defeq  e^{ -t G^w ( x, h D ) } M ( z ) e^{ t G^w ( x, h D ) } \,,
\end{gathered}
\end{gather}
is invertible for $t=t(\vareps)>0$ large enough, and the inverse $ {\mathcal E}_{tG} ( z )  $ is
holomorphic in $z\in D(0,R)$, uniformly when $h\to 0$.
Using the standard result (see \cite[Proposition 4.1]{SZ9} for that,
and \cite[Proposition 4.2]{SZ9} for a generalization not requiring holomorphy) we obtain
\[  \begin{split}
m_M ( z ) & = 
\inf_{\epsilon > 0 } \frac 1 { 2 \pi i } \tr \oint_{\gamma_\epsilon ( z )  } 
( I - M_{tG} ( \zeta ) )^{-1} \partial_\zeta M_{tG} ( \zeta )   d\zeta  \\
& = 
\inf_{\epsilon > 0 } \frac 1 { 2 \pi i }  \tr \oint_{\gamma_\epsilon ( z )  } 
E_{-+} ( \zeta )^{-1} \partial_\zeta E_{-+} ( \zeta ) d\zeta \,. \end{split} \]
Since $ E_{-+} ( \zeta ) $ is a matrix with holomorphic coefficients, 
the right hand side is the multiplicity of the zero of $ \det E_{-+}(\zeta) $
at $ z $. 

Once $0<\vareps<1/2$ and the parameter $t=t(\vareps)$ has been selected, the decay assumption
\eqref{e:decayM}, together with the norm estimate \eqref{e:boundM},
show that there exists a radius $R=R(t)>0$ such that, for $h<h_0$,
$$
z\in\Omega(h),\ \Re z\leq -R/4 \Longrightarrow \| \Pi_V\,M_{tG}(z,h)\|_{L^2\to L^2}\leq 1/2\,.
$$
Together with the expression \eqref{eq:emp}, the bound \eqref{eq:prG} and
the assumption $\vareps<1/2$,
this shows that $E_{-+}(-R/4)$ is invertible, with
$\|E_{-+}(-R/4)^{-1} \|$ uniformly bounded.

We may assume that $R\geq 4R_1$, where $R_1$ is the radius in the
statement of the theorem. The bound \eqref{eq:Mdim} then follows from estimating the 
number of those zeros in $ D ( 0 , R/4 ) $. That in turn follows
from Jensen's formula, which has a long tradition in estimating
the number resonances (see \cite{Mel3}). Namely, for any function $f(z)$ holomorphic
in $z\in D(0,R)$ and nonvanishing at $0$, the number $n(r)$ of its
zeros in $D(0,r)$ (counted with multiplicity) satisfies, for any $r<R$,
$$
  \int_0 ^r \frac {n ( x) } x dx = \frac 1 {2 \pi} 
\int_1^{2 \pi} \log | f ( r e^{i \theta} )| d\theta - \log | f ( 0 ) | \,.
$$
Applying this identity to the function $ f ( \zeta ) = \det E_{-+} ( \zeta - R/4 ) $,
we see that 
\[ \begin{split} \sum_{ z \in D ( 0 , R/4 ) } m_M ( z ) & \leq n ( R/2 ) 
\leq \frac{ 1 } { \log (3/2) } \int_{R/2}^{3R/4} \frac {n ( x ) } x dx 
\leq \frac{ 1 } { \log (3/2) } \int_{0}^{3R/4} \frac {n ( x ) } x dx 
\\
& \leq 
\frac 1 {\log (3/2)  }
 \left( \max_{ z \in D ( 0 , R  ) } \log | \det E_{-+} ( z ) | -
 \log |\det E_{-+} ( - R/4 ) | \right) \,.  
\end{split} \]
Since $ \|E_{-+} ( z )\|_{V\to V} $ is uniformly bounded for $z\in D(0,R)$ and
the rank of $ E_{-+} ( z ) $ is bounded by $ \dim V $, Lemma \ref{l:dim}
gives
\[  \max_{ z \in D ( 0 , R  ) } \log | \det E_{-+} ( z ) | 
\leq C_0 \dim V \leq  K  h ^{-\mu } \,, \]
where $ \mu $ is as in \eqref{eq:Mdim}. 
Also, 
\[  \begin{split} -  \log |\det E_{-+} ( -R/4 ) | &  = \log | \det E_{-+} ( -R / 4 ) ^{-1} | \leq  \dim V \,\log \| E_{-+} ( -R / 4 )^{-1} \| 
\\
 & \leq K  h ^{-\mu }  \,, \end{split} \]
where in the last inequality we used the fact that $E_{-+}(-R/4)^{-1}$ is uniformly bounded.
This completes the proof of \eqref{eq:Mdim} in the case of an
untruncated monodromy operator.

In the case of a truncated operator
$\tM(z,h)=M_W(z,h)+\Oo(h^{N_0})$, we apply Theorem~\ref{t:2gr2} and
get an effective Hamiltonian $\widetilde{E}_{-+}(z)$ which is also
uniformly bounded. The estimate \eqref{e:E-tildeE} provides a uniform
estimate for $\widetilde{E}_{-+}(-R/4)^{-1}$, and the rest of the proof is identical.
\end{proof}


\section{Application to scattering by several convex bodies}
\label{assc}

We now apply the abstract formalism to a very concrete setting 
of several convex obstacles. This will prove Theorem \ref{th:1}
stated in \S \ref{int}.

The setting of several convex obstacles has been a very popular
testing ground for quantum chaos since the work of Gaspard and Rice
\cite{GaRi} but the fractal nature of the distribution of resonances
have been missed by the physics community. 
The optimality of the fractal Weyl laws in that setting was tested
numerically using semiclassical zeta function (hence not in the 
true quantum r\'egime) in \cite{LSZ} -- see Fig.~\ref{f:lsz}. 
The mathematical developments of
other aspects of scattering by several convex obstacles can be found 
in the works  
Ikawa, G\'erard, Petkov, Stoyanov, and Burq -- see 
\cite{NBu},\cite{Ge},\cite{Ik},\cite{Pest}, and references given there.

\subsection{Resonances for obstacles with several
connected components}
\label{bvp}
We first present some general aspects of scattering by several obstacles.
This generalizes and simplifies the presentation of 
\cite[\S 6]{Ge}.  

Let $ {\mathcal O}_j \Subset \RR^n $, $ j = 1, \cdots, J $, 
be a collection of connected open sets,
$ \overline {\mathcal O}_k \cap \overline {\mathcal O}_j = \emptyset $,
with smooth boundaries, $ \partial {\mathcal O}_j $. 
Let 
\[  \Omega \defeq \RR^n \setminus \bigcup_{j=1}^J {\mathcal O}_j \,, \]
and let $ \gamma $ be a natural restriction map:
\[ \gamma \; : \;   H^2 ( \Omega )  \longrightarrow H^{\frac32} ( \partial
\Omega ) 
 \,, \ \   (\gamma \,  u  )_j \defeq \gamma_j u \defeq 
 u \rest_{ \partial \Oo_j}  \,, \]
where we interpret  $\gamma $ as a column vector of operators.

Let $ \Delta_\theta \; : \; H^2 ( \RR^n ) \rightarrow L^2 ( \RR^n ) $ be
the complex-scaled Laplacian, in the sense of \cite[\S 3]{SZ1}, 
\[   \Delta_\theta = \left( \sum_{k=1}^n \partial_{z_k}^2 \right) 
\rest_{\Gamma_\theta } \,, \ \ z \in \CC^n \,, \ \ \Gamma_\theta \simeq \RR^n \,, 
\]
with $ \Gamma_\theta \cap B_{\CC^n}  ( 0 , R ) = \RR^n \cap B_{\RR^n}  (0 , R ) $, 
$ {\mathcal O}_j \Subset B_{\RR^n} ( 0 , R ) $, for all $ j $. Here we 
identified the functions on $ \RR^n $ and functions on $ \Gamma_\theta $.

For $ z \in D ( 0 , r ) $ we define a semiclassical differential operator
\begin{equation}
\label{eq:Pz}  P( z ) \defeq \frac i h ( - h^2 \Delta_\theta - 1 ) - z
\,, 
\end{equation}
with the domain given by either $ H^2 ( \RR^n ) $ or $ 
H^2 (  \Omega ) \cap H_0^1 (  \Omega ) $, $ P ( z ) $
is a Fredholm operator and we have two corresponding resolvents:
\[ R_0 (z ) \; : \; L^2 ( \RR^n ) \; \longrightarrow \; H^2 ( \RR^n ) \,, \ \ 
 R (z ) \; : \; L^2 ( \Omega ) \; \longrightarrow \; H^2( \Omega ) \cap 
H_0^1 ( \Omega )  \,. \]
Here and below $ r > 0$ can be taken arbitrary and fixed as long as 
$ h $ is small enough. 

The operator $ R_0 ( z ) $ is analytic in $ z \in D ( 0 , r ) $ 
(the only problem comes from $ i/h + z = 0 $) 
and $ R ( z ) $ is 
meromorphic with singular terms of finite rank -- see \cite[Lemma 3.5]{SZ1} and,
for a concise discussion from the point of view of boundary layer potentials,
\cite{Mel3}. The multiplicity of a pole of $ R ( z ) $ is defined by 
\begin{equation}
\label{eq:multR}    m_R ( z_0 ) = - \frac{1}{ 2 \pi i } \tr_{L^2( \Omega) } 
 \oint_{\gamma_\epsilon ( z )  } 
R ( z )  d z \,, \ \ 
 \gamma_\epsilon ( z_0 )  : t \mapsto z_0 + \epsilon e^{2 \pi it } \,, \ \ t \in 
[ 0 , 2 \pi ) \,, \end{equation}
and   $ \epsilon > 0 $ is sufficiently small.

A direct proof of the meromorphic continuation and a reduction to the boundary uses Poisson operators associated to individual obstacles $ {\mathcal O}_j $:
\begin{gather*}
H_j ( z )  : H^{\frac32} ( \partial \Oo_j )   \; \longrightarrow \;
H^2 ( \RR^n \setminus {\mathcal O}_j ) 
\; \stackrel{\text{extension by $ 0 $}}{\longrightarrow} \; 
L^2 ( \RR^n ) \,, \\ 
\left( P ( z ) 
H_j ( z )  f \right) ( x ) = 0 \,, \ \ \ x \in \RR^n \setminus
 {\mathcal O}_j 
\,, \ \ \ \ \gamma_j H_j ( z ) f = f \,.  \end{gather*}
We then define a row vector of operators:
\begin{equation}
\label{eq:defH}  H ( z ) : H^{\frac32} ( \partial \Omega ) 
 \; \longrightarrow \; 
 L^2 ( \RR^n ) \,, \ \
 H ( z ) \, \vec v  = \sum_{j=1}^J H_j ( z ) v_j \,, 
\end{equation}
We note that $ \gamma H ( z ) \vec v  \in  H^{3/2} ( \partial \Omega )   $ is
well defined.


The family of operators, $ H ( z ) $, is in general meromorphic with poles
of finite rank -- see the proof of Lemma \ref{l:merM} below.

In this notation the monodromy operator $ {\mathcal M } ( z ) $ defined in 
\eqref{eq:MM} is simply given by 
\begin{gather}
\label{eq:bmon}
\begin{gathered}
\Id  - {\mathcal M }  (z ) \defeq  
 \gamma  H  ( z) 
: H^{\frac32 } ( \partial \Omega ) \longrightarrow 
H^{\frac32 } ( \partial \Omega )  \,,
 \\ 
\left( {\mathcal M }  (z ) \right)_{ij} = \left\{ \begin{array}{ll} 
- \gamma_i H_j ( z ) & i \neq j \,, \\
\ \ \ \ 0 & i = j \,. \end{array} \right. 
\end{gathered}
\end{gather}

We first state the following general lemma:
\begin{lem}
\label{l:merM}
For $  z \in \RR + i ( -1/h , 1/h )  $, 
the operator 
\[  ( \Id  - {\mathcal M} ( z ) )^{-1}  \; : \;
H_h^{\frac32}  ( \partial \Omega)  \rightarrow {H_h^{\frac32} ( \partial
  \Omega)  }  \]
 is
meromorphic with poles of finite rank. For $ \Re z < 0 $ 
\[  \| {\mathcal M} ( z )  \|_{ 
} \leq \frac C {h |\Re z |} \,, \]
and consequently for $ \Re z < - \gamma/ h $,  with $ \gamma $sufficiently large, 
$   ( \Id - {\mathcal M} ( z ) )^{-1}  $ is holomorphic and 
\begin{equation}
\label{eq:lmerM} \|  ( \Id - {\mathcal M} ( z ) )^{-1} \|_{ 
H^{\frac32} ( \partial \Omega)  \rightarrow {H^{\frac32} ( \partial
\Omega)}  } \leq C  \,, \ \ 
\Re z < - \gamma/h  \,. 
\end{equation}
\end{lem}
\begin{proof} 
We first consider $ \Re z < 0 $. Let $ R_j ( z ) $ be the resolvent of
the Dirichlet realization of $ P ( z ) $ on $ \RR^n \setminus {\mathcal O}_j$.

Let $ \chi ( x  ) \in \CIc ( \neigh ( \partial \Oo_j ) $ be equal to 
$ 1 $ near $ \partial \Oo_j $ and have support in a small 
neighbourhood of $ \partial \Oo_j $. In particular we can 
assume that 
the (signed) distance, $ d ( \bullet, \partial \Oo_j ) $ is
smooth there. 
Define the extension operator,  $ \gamma_j T^h_j =
\Id $, 
\begin{gather*}  T_j^h \defeq \chi ( x ) \exp 
\left( - d (x, \partial \Oo_j )^2  ( \Id -h^2 \Delta_{\partial \Oo_j }
  ) / h^2 
\right)
\\
 T_j^h  = {\mathcal O} ( h^{\frac 12} ) 
\; : \; H^{\frac32}_h ( \partial \Oo_j ) 
\longrightarrow H^2_h ( \neigh ( \partial \Oo_j ) ) \,.
\end{gather*}
We also note that 
\[ \gamma_k = {\mathcal O} ( h^{-\frac 12}  ) \; : \; H^2_h ( \RR^n \setminus 
\Oo_j ) \longrightarrow H^{\frac 32}_h ( \partial \Oo_k ) \,, \ \ 
\text{uniformly in $ h $,} \]
see Lemma \ref{l:WFr}.

Then 
\begin{equation}
\label{eq:Hj}   H_j ( z ) = T_j^h - R_j ( z ) P ( z ) T_j^h \,,
\end{equation}
and 
\[ \gamma_k H_j ( z ) = \delta_{jk} \Id
- \gamma_k R_j ( z ) P ( z ) T_j^h \,. \]
The basic properties of complex scaling \cite[\S 3]{SZ1} show that
for $ \chi_j \in \CIc ( B( 0 , R )  ) $ which is $ 1 $ near $ {\mathcal O}_j $
(hence supported away from the complex scaling region) we 
have 
\begin{equation}
\label{eq:jth} \begin{split}  \chi ( \Delta_{j,\theta} - \zeta )^{-1} \chi & = 
\chi ( \Delta_j - \zeta )^{-1} \chi \\
& = 
{\mathcal O} ( 1/ \Im \zeta ) : L^2 ( \RR^n \setminus {\mathcal O}_j )
\longrightarrow L^2 ( \RR^n \setminus {\mathcal O}_j ) 
\,,  \ \ \Im \zeta > 0 \,, \end{split} \end{equation}
where $ \Delta_{j,\theta} $ and $ \Delta_j $ are the complex scaled and the 
usual Dirichlet Laplacians on $ \RR^n \setminus {\mathcal O}_j $.
Hence, 
\[  \gamma_k R_j ( z ) P ( z ) T_j^h = {\mathcal O} ( 1 /  | h  \Re  z | ) 
: H_h^{3/2} ( \RR^n \setminus {\mathcal O}_j ) \longrightarrow 
 H_h^{3/2} ( \RR^n \setminus {\mathcal O}_k )  
\,, \ \  \Re  z <  0 \,.  \]
This in turn shows that 
\[ {\mathcal M} ( z ) 
= {\mathcal O} ( 1 / | h  \Re z| ) : 
 {H_h^{\frac32} ( \partial \Omega)}  \longrightarrow  {H_h^ {\frac32} ( \partial \Omega)}  \,, \ \  \Re z <  0 \,, \]
and consequently that \eqref{eq:lmerM} holds. 

We know (see for instance \cite[Lemma 3.2]{SZ1}) that $ R_j ( z )$ 
is meromorphic in $ D ( 0 , r ) $, and using \eqref{eq:Hj}  we conclude
that so is $ \Id - {\mathcal M}( z ) $. Analytic Fredholm theory 
(see for instance \cite[\S 2]{SZ9}) shows that invertibility of $ 
\Id - {\mathcal M } ( z ) $ for $ \Re z <  -\gamma $ guarantees the meromorphy
of its inverse, with poles of finite rank.
\end{proof}

We recall the following standard result which is already behind the 
definition \eqref{eq:multR}:
\begin{lem}
\label{l:simtr}
Let $ H_1 $  and $ H_2 $ be Hilbert spaces 
and let $ z \mapsto A( z ) \in {\mathcal L} ( H_1 , H_2 ) $,
$ z \mapsto 
B ( z ) \in {\mathcal L} ( H_2 , H_1 ) $, $ z \in D \subset \CC $,
be holomorphic families of bounded operators. Suppose that
$ z \mapsto C ( z) \in {\mathcal L} ( H_2, H_2 ) $,
 $ z \in D$,  is a meromorphic family
of bounded operators, with poles of finite rank. Then for any smooth closed curve
$ \gamma \subset D $, the operator
\[  \oint_\gamma A ( z ) B ( z ) C ( z ) dz \,, \]
is of trace class and 
\begin{equation}
\label{eq:cytr} \tr_{H_2}  \oint_\gamma A ( z ) B ( z ) C ( z ) dz 
= \tr_{H_1}  \oint_\gamma B ( z ) C ( z ) A ( z ) dz 
\,. \end{equation}
\end{lem}
\begin{proof}
Without loss of generality we can assume that the winding number
of $ \gamma $ is nonzero with respect to only one pole of $ C ( z ) $, 
$ z_0 \in D $. We can write
\[ C ( z ) = C_0 ( z ) + \sum_{k=1}^K \frac{C_k}{ ( z - z_0 )^k } \,, \]
where $ C_0 ( z ) $ is holomorphic near $ z_0 $, and $ C_k $ are finite
rank operators. Consequently,  
\begin{equation}
\label{eq:calc}  \oint_\gamma A ( z ) B ( z ) C ( z ) dz = 
\oint_\gamma \left( \sum_{k=1}^K \frac{ A ( z ) B ( z ) C_k   }{ ( z - z_0 )^k } 
\right) dz \,,\end{equation}
is a finite rank operator, and 
\[ \begin{split} \tr_{H_2}  \oint_\gamma A ( z ) B ( z ) C ( z ) dz & = 
\oint_\gamma \left( \sum_{k=1}^K \frac{ \tr_{H_2} 
A ( z )  B ( z )  C_k }{ ( z - z_0 )^k } 
\right) dz  \\
& = 
\oint_\gamma \left( \sum_{k=1}^K \frac{\tr_{H_1}  B ( z ) C_k A ( z ) }
{ ( z - z_0 )^k } \right) dz 
\,, \end{split} \]
where we used the cyclicity of the trace: $ \tr ST = \tr TS $ when $ S $ is
of trace class, and $ T $ is bounded. Same calculation as in \eqref{eq:calc}
gives \eqref{eq:cytr}.
\end{proof}

The main result of this section is the following
\begin{prop}
\label{p:merM}
Suppose that the family of operators, $  z \mapsto
H ( z ) $, defined in \eqref{eq:defH} is holomorphic for $ z \in D ( 0 , r_0) 
$, 
Then the resonances, that is the poles of $ R ( z ) $, 
agree with multiplicities with the poles of $ \Id - 
{\mathcal M} ( z ) $:
\begin{equation}
\label{eq:pmerM}   m_R ( z_0) = - \frac{1}{ 2 \pi i } 
\tr_{L^2 ( \partial \Omega )} 
 \oint_{\gamma_\epsilon ( z )  } 
( \Id - {\mathcal M } ( z ) )^{-1} \frac{d}{dz} 
{\mathcal M } ( z ) dz \,, 
 \end{equation}
where $ \gamma_\epsilon ( z_0 )  : t \mapsto z_0 + \epsilon e^{2 \pi it } $, $ t \in 
[ 0 , 2 \pi ) $, for sufficiently small $ \epsilon > 0 $, and the multiplicity 
$ m_R ( z_0 ) $ is defined by \eqref{eq:multR}.
\end{prop}
\begin{proof}
We first recall, for instance from \cite[\S 6]{Ge},
that $ R ( z ) $ can be expressed using the inverse
of $ \Id - {\mathcal M}( z ) $:
\begin{equation}
\label{eq:exRM} R ( z ) = \bbbone_\Omega R_0 ( z) - 
 \bbbone_\Omega  H (z ) (
\Id - {\mathcal M}( z ) )^{-1} \gamma R_0 ( z ) 
 \,,
\end{equation}
where $ R_0 ( z ) $ acts on functions $ L^2 ( \Omega ) \hookrightarrow 
L^2 (\RR^n )$ extended by $ 0 $.
Indeed,
\[ P  \left( \bbbone_\Omega R_0  f - 
 \bbbone_\Omega  H  (
\Id - {\mathcal M} )^{-1} \gamma R_0  f 
\right) = P  \bbbone_\Omega R_0  f = f 
\,, \ \ \text{ in $\Omega$. }  \]
The Dirichlet boundary condition is satisfied as 
\[ \gamma \left( \bbbone_\Omega R_0 a  - 
 \bbbone_\Omega  H  (
\Id - {\mathcal M} )^{-1} \gamma R_0 
\right) = 
\gamma R_0  - \gamma H  ( \gamma H)^{-1} 
\gamma R_0  = 0 \,. \]
Hence by the uniquenss of the outgoing solution (using complex
scaling that is simply the Fredholm property of $ P ( z ) $) 
\eqref{eq:exRM} holds.

Since $ R_0 ( z) $ is
holomorphic in $ z $, \eqref{eq:exRM} and Lemma \ref{l:simtr} show that
\begin{equation}
\label{eq:awf} 
\begin{split}
  m_R ( z ) &  = \frac{1}{ 2 \pi i } \tr \oint_{\gamma_\epsilon ( z )  } 
 \bbbone_\Omega H ( z ) ( \Id 
- {\mathcal M } ( z ) )^{-1} \gamma R_0 ( z )  \,  dz \\
\end{split}
\end{equation}

We want to compare it to the right hand side in \eqref{eq:pmerM} and for that
we will use the following
\begin{lem}
\label{l:derM}
With the notation above we have
\begin{equation}
\label{eq:lderM}
 \frac{d}{dz} {\mathcal M} ( z ) = - \gamma R_0 ( z ) H ( z ) 
+ (I - {\mathcal M} ( z ) ) \diag \left( \gamma_k R_0 ( z ) 
H_k ( z ) \right) \,. 
 \end{equation}
\end{lem}
\begin{proof} 
The definition \eqref{eq:bmon} gives $ ( d/dz) {\mathcal M} ( z ) = -
\gamma H' ( z ) $, and by differentiating 
\[  P ( z ) H_k ( z ) = 0 \,, \ \ \
 \gamma_k H_k ( z ) = \Id \,, \]
we obtain
\[ P ( z ) H_k'( z )  = H_k( z ) \ \    \text{ in $ \RR^n \setminus  \Oo_k $} 
 \,, \ \  \ \ \gamma_k H_k' ( z ) = 0 \,. \]
Arguing as in the case of \eqref{eq:exRM} we obtain
\[ H_k'( z) = \bbbone_{ \RR^n \setminus \Oo_k} R_0 ( z ) H_k ( z ) - 
H_k ( z ) \gamma_k R_0 ( z ) H_k ( z ) \,, \]
and hence 
\[  \frac{d}{dz} {\mathcal M} ( z ) = 
- \gamma R_0 ( z )  H ( z ) + \gamma 
 H ( z ) 
\diag \left( \gamma_k R_0 ( z ) H_k ( z ) \right) \,,
\]
which is the same as \eqref{eq:lderM}.
\end{proof}
Lemmas \ref{l:simtr} and \ref{l:derM}, and the assumption of holomorphy of 
$ H ( z ) $ in $ D ( 0 , r_0 ) $,  show that
\begin{gather}
\label{eq:Sjtr}
\begin{gathered}  - \tr_{L^2 ( \partial \Omega ) } 
 \oint_{\gamma_\epsilon ( z )  } 
( \Id - {\mathcal M } ( z ) )^{-1} \frac{d}{dz} 
{\mathcal M } ( z ) dz 
= \ \\  \tr_{L^2 ( \RR^n ) }  \oint_{\gamma_\epsilon ( z )  } 
 H ( z ) ( \Id - {\mathcal M } ( z ) )^{-1} 
\gamma  R_0 ( z )   dz 
\,, \end{gathered}
\end{gather}
which is awfully close to \eqref{eq:awf}. The difference 
of the right hand sides of \eqref{eq:Sjtr} and \eqref{eq:awf}
is equal to 
\begin{equation}
\label{eq:Sjtr2}
\begin{split} 
& \sum_{j=1}^J  \tr 
\oint_{\gamma_\epsilon ( z )  } 
\bbbone_{\mathcal O_j}  H ( z ) ( \Id - {\mathcal M } ( z ) )^{-1} 
\gamma  R_0 ( z )   dz = 
\\
& \sum_{j=1}^J  \tr 
\oint_{\gamma_\epsilon ( z )  } 
 H ( z ) ( \Id - {\mathcal M } ( z ) )^{-1} 
\gamma  R_0 ( z )  \bbbone_{\mathcal O_j} \,  dz \,. 
\end{split}
\end{equation}
Now observe that 
$$
\bbbone_{\RR^n \setminus \Oo_j} R_0(z) \bbbone_{{\mathcal O}_j }= H_j (z)\gamma _j R_0(z)
\bbbone_{{\mathcal O}_j},
$$
which implies 
$$
\gamma_k R_0(z) \bbbone_{{\mathcal O}_j}= \gamma_k 
H_j (z)\gamma _j R_0(z)
\bbbone_{{\mathcal O}_j}\,.
$$
This in turn shows that
\[ \gamma R_0 ( z ) \bbbone_{{\mathcal O}_j} = 
\gamma H( z ) 
 \pi_j \gamma
 R_0(z) \bbbone_{{\mathcal O}_j } = 
( \Id - {\mathcal M } ( z ) ) 
 \pi_j \gamma
 R_0(z) \bbbone_{{\mathcal O}_j } \,, 
\]
where $\pi_j:{\CC}^J\to {\CC}^J$ is the orthogonal
projection onto ${\CC}e_j$, $e_j$ denoting the $j$th
canonical basis vector in ${\CC}^J$. Hence 
$$
 H(z)  ( \Id - {\mathcal M } ( z ) )^{-1} 
\gamma R_0(z)\bbbone_{{\mathcal O}_j}=H (z)\pi_j
\gamma R_0(z)\bbbone_{{\mathcal O}_j},
$$  This expression is holomorphic
in $z$ since we assumed that $H(z)$ has no poles in the region of
interest. That proves that the trace in \eqref{eq:Sjtr2} vanishes
and completes the proof of \eqref{eq:pmerM}.
\end{proof}

\subsection{Semiclassical structure of the Poisson operator for convex
obstacles}
\label{sep}
We now review the properties of the operator $ H ( z ) $ given in 
\eqref{eq:defH} where it is given in terms of Poisson operators
$ H_j ( z ) $ for individual convex obstacles. These properties
are derived from results on propagation of singularities 
for diffractive boundary value problems (see \cite[\S 24.4]{Hor2} 
and references given there) and
from semiclassical parametrix 
constructions \cite[A.II]{Ge}, \cite[A.2-A.5]{StVo}. They are based
on ideas going back to Keller, Melrose, and Taylor -- see 
\cite{MT} and references given there. 
The main result we need is stated in Proposition \ref{p:mt} below.  

To orient the reader we first present a brief discussion of a model
case and then use the parametrix to prove the general resuls.

\subsubsection{A model case}
We will review this parametrix in a special model case where
it is given by an explicit formula. 
Using Melrose's equivalence of glancing hypersurfaces
\cite{Mel} this model can be used to analyze the general case but 
due to the presence of the boundary that is quite involved 
\cite[\S 7.3, Appendix A]{MSZ} (see also \cite[Chapter 2]{MSZ} for a
concise presentation of diffractive geometry).


The model case (in two dimensions for simplicity) is provided by 
the {\em Friedlander model} \cite{FGF},\cite[\S 21.4]{Hor2}: 
\begin{equation}
\label{eq:modd} P_0 = ( h D_{x_2} )^2 - x_2 + h D_{x_1} \,, \ \ p_0 = 
\xi_2^2 - x_2 + \xi_1 \,,
\end{equation} 
with the boundary $ x_2 = 0 $, the Poisson operator $ H_ 0$.
The surface to which $ H_0 u $ is restricted to can be written as 
$ x_2 = g ( x_1 ) $. The Hamilton flow of $ p_0 $ is
explicitely computed to be 
\[  ( x, \xi ) \longmapsto ( x_1 + t , x_2 - \xi_2^2 + ( t + \xi_2 )^2  ;
\xi_1 , \xi_2 + t ) \,, \]
and the trajectories on the energy surface $ p_0 = 0 $ tangent to the
boundary $ x_2 = 0 $ correspond to $ \xi_1 = 0 $.
The bicharacteristic concavity of a region $ q_0 ( x)  > 0 $  (modelling the concavity of $ \RR^n \setminus \partial \Oo_j $) is given by the condition
$ H_{p_0}^2 q_0 > 0 $: that is automatically satisfied for $ q_0 ( x) = x_2  $ 
and holds for $ q_0 ( x) =  g ( x_1 ) - x_2 $ if 
$ g'' ( x_1 ) > 2 $. 

For 
$ v \in \CIc ( \RR_{x_1} ) $ 
 the problem
\begin{equation}
\label{eq:P0}
  P_0 u ( x ) = 0 \,, \ \ x_2 > 0 \,, \ \  u ( x_1, 0 ) = v ( x_1 ) \,,
\end{equation}
has 
an explicit solution:
\begin{equation}
\label{eq:H0}  u ( x ) = H_0 v ( x ) \defeq \frac{ 1 } { 2 \pi h } 
\int \! \! \! \int \frac { A_+ ( h^{-2/3} ( \xi_1 - x_2 ) ) } { A_+ ( h^{-2/3} \xi_1 ) }
e^{\frac i h ( x_1 - y_1 ) \xi_1 } v ( y_1 ) d y_1 d \xi_1 \,, 
\end{equation}
where $ A_+ $ is the Airy function solving $ A_+''(t) = t A_+( t ) $
and having the following asymptotic behaviour:
\[  A_+ ( t ) \sim \ \begin{cases} \pi^{-\frac12} t^{-\frac14} e^{\frac 23 t^{3/2} } 
& t \rightarrow + \infty  \,, \\
\pi^{-\frac12} (- t)^{-\frac14} e^{\frac 23 i (-t)^{3/2} + i \frac\pi 4  } & t 
\rightarrow - \infty  \,. 
\end{cases} \]
Different asymptotic behaviours corresponds to different classical regions:
\[
\begin{array}{ll}
 \xi_1 < 0\,, & \text{ hyperbolic region: trajectories transversal 
to the boundary,} \\
\xi_1 = 0\,, & \text{ glancing region: trajectories tangent to the boundary,} 
\\
\xi_1 > 0\,, & \text{ elliptic region: trajectories disjoint from the
boundary.} 
\end{array} \]
If $ v $ is microlocally concentrated in the hyperbolic region,
$ \WFh ( v ) \Subset \{ \xi_1 < 0 \} $, then
\[ H_0 v ( x ) = 
\frac{ 1 } { 2 \pi h } 
\int \! \! \! \int e^{ \frac i h \varphi( x , \xi_1) - \frac i h y_1 \xi_1 } 
a ( x, \xi_1 ) v ( y_1 ) d y_1 \, d \xi_1  + { \mathcal O} ( h^\infty ) 
\| v \|_{L^2 }  \,, \]
where 
$ \varphi ( x, \xi_1 ) 
= \frac 23 (
- ( - \xi_1 )^{\frac32} + ( - \xi_1 + x_2)^{\frac32} )  +  x_1  \xi_1 $.
That means that $ H_0 $ 
is microlocally an $h$-Fourier integral operator in the
hyperbolic region, with the canonical relation given by 
\[ \begin{split} 
{\mathcal C}_0  & \defeq \{ ( ( x_1 , x_2 ; \partial_{x_1} \varphi , \partial_{x_2} \varphi ), 
( \partial_{\xi_1} \varphi, \xi_1 ) ) \,, \ \xi_1 < 0 \}  \\
& = \{  ( ( x_1 , x_2 ; \xi_1 ,   ( - \xi_1 + x_2)^{\frac12} ) , 
( x_1 -   ( - \xi_1 + x_2)^{\frac12} + ( -\xi_1)^{\frac12} , \xi_1 ) ) 
\,, \ \xi_1 < 0 \} 
\\ & = \{ (( x_1 , ( x_1 - y_1 + ( - \xi_1)^{\frac12} )^2 + \xi_1 ; \xi_1 , x_1 - y_1 
+ ( - \xi_1)^{\frac12}  ) , ( y_1 , \xi_1  )) \,, \ \xi_1 < 0 \}  \,. 
\end{split} \]
This corresponds to outward trajectories starting at $ ( y_1 , 0 ) $
and explains why
this choice of an Airy function gives the outgoing
solution to \eqref{eq:P0}. 

The propagation of semiclassical wave front sets is given by 
taking the closure of this relation which is smooth for 
$ \xi_1 <  0 $ only: 
\begin{equation} 
\label{eq:WF}
\WFh ( H_0 v ) \cap \{ x_2 > 0 \}= \overline {\mathcal C }_0 ( \WFh ( v ) \cap \{ \xi_1 \leq 0 \} )  \cap \{ x_2 > 0 \}\,. 
\end{equation}
This can be proved using \eqref{eq:H0}.
Strictly speaking the wave front set on the left hand side 
is defined only in the region
$ \{ x_2 > 0 \} $ because of the presence of the bounday $ x_2 = 0 $.

We now consider $ \gamma_1 u ( x_1 ) \defeq u ( x_1 , g ( x_1 )) $
and (putting $ x = x_1 $),
\begin{equation}
\label{eq:gamH} \gamma_1 H_0 v ( x ) = 
\frac{ 1 } { 2 \pi h } 
\int \! \! \! \int \frac { A_+ ( h^{-2/3} ( \eta - g( x) ) ) } 
{ A_+ ( h^{-2/3} \eta ) }
e^{\frac i h ( x - y ) \eta } v ( y ) d y \, d \eta \,. 
\end{equation}
When acting on functions with $ \WF_h ( v ) \Subset \{ (y , \eta) : 
\eta < 0 \} $, we can again use asymptotics of $ A_+ $ and that
shows that, microlocally for $ \eta < -c < 0 $, 
$ \gamma_1 H_0 $ 
is an $h$-Fourier integral operator with a canonical relation 
with a fold
\cite[\S 21.4]{Hor2},\cite[\S 4]{MT}: 
\[  {\mathcal B}_0 \defeq 
\{  ( x , \eta + g'( x ) ( - \eta + g ( x ) )^{\frac12} ; x + 
( - \eta)^{\frac12}  - ( - \eta + g ( x ) )^{\frac12} , \eta ) \,, \ \eta < 0 \} \,.
\]
This means that the map $ f : {\mathcal B}_0 \rightarrow T^* \RR $, 
the projection on the second factor, $ f ( x , \xi, y , \eta ) = 
( y , \eta ) $, at every point at which
\[   g'(x) = 2 ( - \eta + g ( x ) )^{\frac 12}\,, \]
satisfies $ \dim \ker f' = \dim \coker f'= 1 $, and 
has a non-zero Hessian, 
\[  \ker f' \ni X \longmapsto \langle f'' X , X \rangle 
\in \coker f' \,.\]
The Hessian condition is equivalent to $ g'' ( x ) \neq 2 $
which is satisfied as we assumed that $ g''(x ) > 2 $.
This corresponds to the tangency of the trajectory 
$ x_2 = ( x_1 - y + ( -\eta)^{1/2} )^2 + \eta $, to the boundary 
$ x_2 = g ( x_1 ) $. 
Using either an explicit calculation or general results on folds
(see \cite[Theorem 21.4.2 and Appendix C.4]{Hor2}) we can write
\be\label{e:B_0+}  {\mathcal B}_0 = {\mathcal B}_0^+ \cup {\mathcal B}_0^- \,,
\ee
where $ {\mathcal B}_0^\pm $ correspond to trajectories 
entering ($+$) and leaving ($-$) $ x_2 > g( x_1 ) $. Using
\eqref{eq:gamH} one can show a propagation result similar to
\eqref{eq:WF}:
\begin{equation} 
\label{eq:WF1}
\WFh ( \gamma_1 H_0 v ) = \overline {\mathcal B }_0 ( \WFh ( v ) 
\cap \{ \xi_1 \leq 0 \}) \,. 
\end{equation}

\medskip

\subsubsection{Arbitrary convex obstacle}
To handle the general case we introduce the following notation
\begin{gather*}
S^*_{\partial \Oo_k} \RR^n = \{ ( x , \xi) \in T^* \RR^n \; : \;
x \in \partial \Oo_k \,, \ |\xi | = 1 \} \,, \\ 
\ S^* \partial \Oo_k = \{ ( y , \eta ) \in T^* \partial \Oo_j \,, \ \ 
| \eta | = 1 \} \,, \\
  B^* \partial \Oo_k = \{ ( y , \eta ) \in T^* \partial {\Oo_k} \; : \; 
| \eta| \leq 1 \} \,, \ \  \pi_k: S_{\partial \Oo_k} ^*( \RR^n) 
\longrightarrow B^* { \partial {\Oo_k}}  \,. 
\end{gather*} 
where  $ | \bullet | $ is the induced Euclidean metric, and 
 $\pi_k $ is the orthogonal projection.

We first recall a result showing that when considering $ \gamma_k H_j
( z ) $ we can restrict our attention to a neighbourhood
of $ B^* \partial \Oo_k \times B^* \partial \Oo_j $:

\begin{lem}
\label{l:S1}
Suppose that $ \partial \Oo_\ell $, $ \ell = j, k $, $ \Oo_j \cap \Oo_k =
\emptyset $,  are smooth and 
that $ \Oo_j $ is strictly convex.  If $ \chi_\ell \in S^{0,0} ( T^* \partial \Oo_\ell ) $ 
satisfy $ \chi_\ell \equiv  1 $ near $ B^* \partial \Oo_\ell $, 
$ \ell = j, k $, then for $ z \in D ( 0  , r ) $,  
\begin{equation}
\label{eq:S1}
\begin{split}
&  \gamma_k H_j (z ) ( 1 - \chi_j^w ) = {\mathcal O } ( h^\infty ) \; : \;
L^2 ( \partial \Oo_j ) \longrightarrow C^\infty ( \partial \Oo_k ) \,,
\\
&  ( 1 - \chi_j^w ) \gamma_k H_j (z)  = {\mathcal O } ( h^\infty ) \; : \;
L^2 ( \partial \Oo_j ) \longrightarrow C^\infty ( \partial \Oo_k ) \,.
\end{split} 
\end{equation}
\end{lem}
\begin{proof}
The first estimate follows from the parametrix construction in the 
elliptic region -- see \cite[Proposition A.II.9]{Ge} and
\cite[\S A.-A.5]{StVo}. From \cite[(A.24)-(A.26)]{StVo} we see that 
for $ \psi \in \CIc ( \RR^n \setminus \overline { \Oo_j } ) $
\[  \psi( x ) H_j  ( z ) ( 1 - \chi_j^w )  = {\mathcal O}_{L^2
  \rightarrow \CI } ( h^\infty ) \,. \] 
and the first estimate in \eqref{eq:S1} follows.
As a consequence we can extend $ \gamma_k H_j ( z ) : H^{\frac 32}
( \partial \Oo_j ) \rightarrow  H^{\frac 32}
( \partial \Oo_k )  $ to 
\begin{equation}
\label{eq:HL2}  \gamma_k H_j ( z ) \; : \; L^2 ( \partial \Oo_j ) \longrightarrow
L^2 ( \partial \Oo_k ) \,. 
\end{equation}

Once $ \gamma_k H_j ( z ) $ is defined on $ L^2 $, the second part of 
\eqref{eq:S1} follows from the fact that $ ( - h^2 \Delta - 1 ) H_j (
z ) v = 0 $ and hence $ \WFh( \psi H_j ( z ) ) \subset \{ | \xi |  = 1
\} $. We simply apply Lemma \ref{l:WFr}.
\end{proof}

The next proposition establishes boundedness properties:

\begin{prop}
\label{p:S2} 
Suppose that $ \partial  \Oo _\ell $, $ \ell = j , k $ are
smooth and that $ \Oo_j $ is strictly convex. Then 
\begin{equation}
\label{eq:S2} \gamma_k H_j = {\mathcal O} ( 1/h ) \; : \; L^2 ( \partial
\Oo_j ) \longrightarrow L^2(  \partial  \Oo_k ) \,.
\end{equation}
\end{prop}

We cannot quote the results of \cite{Ge} directly since for 
similar estimates in \cite[A.II.1]{Ge} the Lax-Phillips
odd dimensional theory is invoked. As in \cite{Ge} our proof
is based on propagation of singularities for the
time dependent problem, but it uses the more flexible
method due to Vainberg \cite{Vai}. 

\begin{proof}
Let $ H_h^s ( \partial \Oo_\ell ) $ denote semiclassical 
Sobolev spaces defined in \eqref{eq:ssn}.

As in the proof of Lemma \ref{l:merM} we will use the 
resolvent of the Dirichlet Laplacian on $ \RR^n \setminus
\Oo_j $, denoted below by $\Delta_j$.  We also use the extension operator $ T_j^h $ defined there.

Following \eqref{eq:Hj} we write
\[  \gamma_k H_j = \delta_{jk} - \gamma_k R_j ( z ) P ( z ) T_j^h \; :
\; H^{\frac32}_h ( \partial \Oo_j ) \longrightarrow H^{\frac32}_h
  ( \partial \Oo_k ) \,, 
\]
and we need to show that for $ z \in [ - C_0 \log ( 1/ h ) , R_0] + i
[-R_0, R_0]  $ 
the bound is $ {\mathcal O} ( 1/h) $. In view of the discussion above that means
showing that for $ \varphi_\ell \in \CIc ( \neigh ( \partial \Oo_\ell
)) $, 
\be\label{e:est-R_j}  \varphi_k  R_j ( z ) \varphi_j = {\mathcal O} ( 1 ) : L^2 ( \RR^n
\setminus \Oo_j ) \longrightarrow H^2_h ( \RR^n ) \,. 
\ee
We recall that  $ \varphi_k R_j ( z)\varphi_j  = \varphi_k P_j( z)^{-1}
\varphi_l $, where 
\[ P_j ( z ) = ( ( i/h) ( - h^2 \Delta_{j,\theta} - 1 ) -
z )  = {\mathcal O} ( 1/h) : H^2_h ( \RR^n \setminus \Oo_j ) 
\longrightarrow L^2 ( \RR^n \setminus \Oo_j ) \,, \]
and   $ \Delta_{j, \theta }  $ is the complex scaled Dirichlet
Laplacian on $ \RR^n \setminus \Oo_j $  -- see \eqref{eq:Pz}.
The rescaling involved in the definition shows that we need
\begin{equation}
\label{eq:Vai}
\varphi_k ( ( - h^2 \Delta_{j, \theta } - 1 ) - hz/i )^{-1} \varphi_j = 
{\mathcal O}( 1/h ) \; : \;
{L^2 ( \RR^n \setminus  \Oo_j ) 
\longrightarrow H^2_h ( \RR^n  ) }  \,, \end{equation}
which follows from 
\begin{equation}
\label{eq:Sz}
\varphi_k  ( - \Delta_j - \zeta^2 )^{-1}  \varphi_j = 
{\mathcal O}_{L^2\to L^2} ( 1/ | \zeta| ) \,, \ \  \Im \zeta > - C \,, \ \ \Re
\zeta > C \,,
\end{equation}
where, as in \eqref{eq:jth} we have replaced $\Delta_{j,\theta} $ by
the unscaled operator.

To establish \eqref{eq:Sz} we use Vainberg's theory as presented
in \cite[Section 3]{SjN} and \cite[Section 3]{TZ}. For that we need results about the 
wave propagator.

Let $ U_j ( t ) \defeq \sin t \sqrt {-\Delta_j } / \sqrt { - \Delta_j
} $ be the Dirichlet wave propagator. Then for $ \chi \in \CIc ( \RR^n ) $
\begin{equation}
\label{eq:Ujt} \chi U_j ( t ) \chi : L^2 ( \RR^n \setminus \Oo_j )
\longrightarrow  \overline{ \mathcal C }^\infty ( \RR^n \setminus \Oo_j ) \,, \ \ t
\geq T_\chi \,, \end{equation}
where $ \overline{ \mathcal C }^\infty $ denotes the space of extendable smooth
functions.  
This 
follows from the singularities propagate along
reflected and glancing rays and that there are no trapped rays in the case 
of one convex obstacle -- see \cite[\S 24.4]{Hor2}. 

Applying \cite[(3.35)]{SjN} or \cite[Proposition 3.1]{TZ} to \eqref{eq:Ujt} gives 
\eqref{eq:Sz}, and thus completes the proof of \eqref{eq:S2}.
\end{proof}

We now set up some notations concerning the symplectic relations
associated with our obstacle system. 
For $ i \neq j $ we now define the (open) symplectic 
relations $\cB^\pm_{ij}$, analogues of the relations $\cB_0^{\pm}$
\eqref{e:B_0+} in the Friedlander model. 
For $ x \in \partial \Oo_j $ denote by $ \nu_j ( x) $ the
{\em outward} pointing normal vector to $ \partial \Oo_j $ at 
$ x $. Then 
\begin{gather}
\label{eq:Bij}
\begin{gathered}
 ( \rho , \rho') \in {\mathcal B}^\pm_{ij} \subset B^* \partial \Oo_i \times B^* \partial \Oo_j 
 \\ \Longleftrightarrow \\ \exists \; t > 0 \,, \xi \in \SP^{n-1} \,, \
x  \in {\partial \Oo_j} \,, \  x + t \xi  \in 
{\partial \Oo_i}  \,,  \ \langle \nu_j ( x ) , \xi \rangle > 0 \,, \\ 
\pm \langle \nu_i ( x + t \xi ) , \xi \rangle <  0 \,, \
\pi_j ( x, \xi ) = \rho' \,, \ \ \pi_i ( x + t  \xi , \xi) = \rho \,, 
\end{gathered}
\end{gather}
Notice that $\cB^+_{ij}$ is
equal to the billiard relation $F_{ij}$ defined in \eqref{eq:bill}. 

The relations $\cB^{\pm}_{ij}$ are singular at their boundaries
\be \partial \, { {\mathcal B}^\pm_{ij}}  \, = \,
 \overline{  {\mathcal B}_{ij}^\pm } \, \cap \left( 
B^* \partial \Oo_i \times S^* \partial \Oo_j \cup 
S^* \partial \Oo_i \times B^* \partial \Oo_j  \right) \,,
\ee
which corresponds to the glancing rays on $\Oo_j$ or $\Oo_i$. We
will often use the closures of these relations,
$\overline{\cB^{\pm}_{ij}}$, which include the glancing rays.
The inverse relations ($ {\mathcal C}^t \defeq \{ ( \rho, \rho') : ( \rho', \rho )
\in {\mathcal C} \} $) are obtained by reversing the momenta:
\be\label{e:inverse}
( {\mathcal  B}_{ij}^+ )^t = {\mathcal J} \circ {\mathcal B}_{ji}^+ 
\circ {\mathcal J} \,, \ \ {\mathcal J} ( y , \eta ) \defeq
( y , - \eta ) \,, \ \ ( y , \eta ) \in B^* \partial \Oo  \,. 
\ee
If we define
\begin{equation}
\label{eq:deU}  \UU \defeq \neigh(B^* \partial \Oo) =
\bigsqcup_{j=1}^J \neigh(B^* \partial \Oo_j) \,, \end{equation}
we are in the dynamical setting of \S \ref{2gr}.
\begin{figure}
\begin{center}
\includegraphics[width=0.9\textwidth]{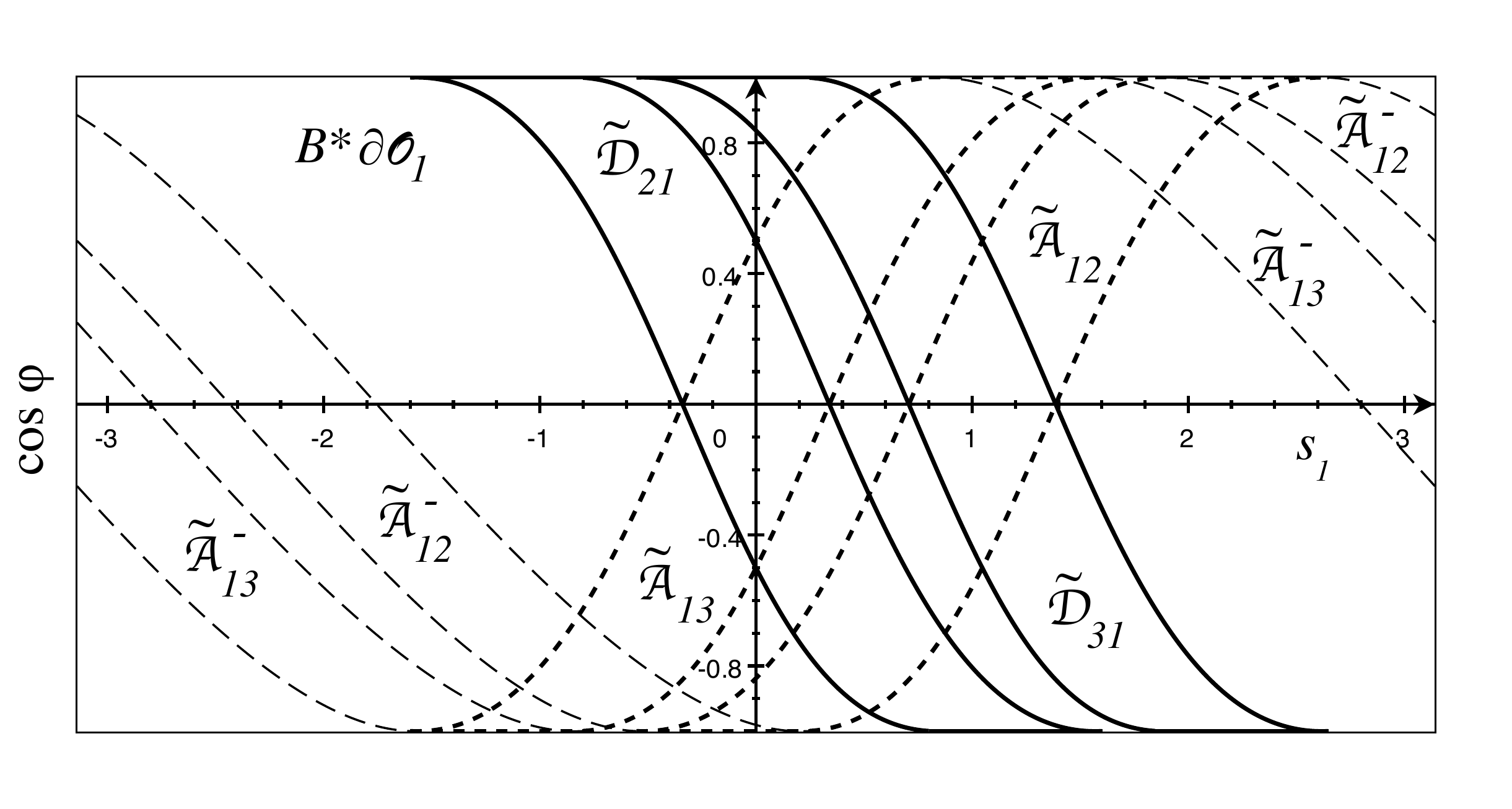}
\caption{\label{f:3disks-AD} (after \cite{Ott}) Partial boundary phase
  space $B^*\Oo_1$ for the symmetric three disk scattering shown in Fig.~\ref{ott}.
We show the boudaries of the departure ($\tD_{i1}$, full lines), arrival
($\tA_{1j}$, dashed lines) and shadow arrival ($\tA_{1j}^-$, long
dashed) sets. }
\end{center}
\end{figure}
Compared with \S\ref{2gr}, we define the arrival and departure sets
from the closure of the relation $F_{ij}$:
\[  \widetilde A_{ij} \defeq  \overline{{\mathcal B}_{ij}^+} (
B^* \partial \Oo_j )\subset B^*\partial\Oo_i \,, 
\ \ \ \ \widetilde D_{ij} \defeq \overline{({\mathcal B}_{ij}^+)^t} ( B^* \partial \Oo_i )\subset B^*\partial\Oo_j \,.
\] 
From \eqref{e:inverse} we check that 
\[  \widetilde D_{ij} =  {\mathcal J} ( \widetilde A_{ji} )   \,. \]
Besides, the {\em shadow} arrival 
sets are given by 
\[  \tA_{ij}^- \defeq \overline{\cB_{ij}^- } 
( B^* \partial \Oo_j ) \,.
\] 
Also, let us call 
\begin{equation}
\label{eq:defE}
\tA_i^{(-)} \defeq \bigcup_{ i \neq j } \tA_{ij}^{(-)} \,, \quad
\tD_i\defeq \bigcup_{ i \neq j } \tD_{ji}\,.
\end{equation}
The subsets of glancing rays are denoted by
$$ \tA_i^{\mathcal{G}}\defeq  { \tA_{i} } \cap 
S^* \partial \Oo_i  =  { \tA_{i}^-} 
\cap S^* \partial \Oo_i,\qquad \quad  \tD_i^{\mathcal{G}}\defeq  { \tD_{i} } \cap 
S^* \partial \Oo_i\,.
$$ 
With this notation we can state the most important result of this 
section:
\begin{prop}
\label{p:mt}
For $i\neq j$, let $ H_j ( z )$ and $ \gamma_i $ be as in \S\ref{bvp}, and 
assume that $ \partial \Oo_k $, $ k = i , j $ are strictly convex. 
For any tempered $v\in L^2(\Oo_j)$, we have
\begin{gather}
\label{eq:WF2} 
\begin{gathered}
  \WFh ( \gamma_i H_j ( z ) v ) = \left( \overline{  {\mathcal B}_{ij}^+ \cup
{\mathcal B}_{ij}^- } \right) ( \WFh ( v ) \cap B^* \partial \Oo_j ) \,,
\end{gathered}
\end{gather}
uniformly for 
\be\label{e:Omega_0}
 z \in \Omega_0 \defeq [ - C_0 \log ( 1/ h ) , R_0 ] + i[ - R_0 , R_0
 ] ,\quad C_0
, R_0 > 0 \quad \text{fixed.}
\ee
If $ Q_k \in \Psi_h^{0,-\infty } ( \partial \Oo_k ) $, $ k = i, j $, 
satisfy 
\be\label{e:condQ}
 \WFh ( Q_i ) \cap {\tA_i^{\mathcal{G}}} = \emptyset \,, \ \ 
\WFh ( Q_j ) \cap \tD_j^{\mathcal{G}} = \emptyset \,, 
\ee
 then 
\begin{equation}
\label{eq:fio1} Q_i \gamma_i H_j ( z ) Q_j \in I^0 ( \partial \Oo_i \times \partial \Oo_j , 
({\mathcal B}_{ij}^+)' ) +  I^0 ( \partial \Oo_i \times \partial \Oo_j , 
({\mathcal B}_{ij}^-)' ) \,. 
\end{equation}
Because of the assumptions on $ Q_k$'s, only compact 
subsets of the open relations 
$ {\mathcal B}_{ij}^\pm $ are involved in the definition of the 
classes $ I^0 $. 

We also have, for some $\tau>0$ and $z$ in the above domain, the norm estimate
\begin{equation}
\label{eq:fio11}
\| Q_i \gamma_i H_j ( z ) Q_j  \|_{ L^2 ( \partial \Oo_j)  \rightarrow
  L^2 ( \partial \Oo_i) }   \leq  C( R_0) \exp (\tau  \Re z )  \,.
\end{equation}
\end{prop}

 Although we will never have to use any detailed analysis near the glancing
points (that is, points where the trajectories are tangent to the 
boundary) it is essential that we know \eqref{eq:WF2} and that requires
the analysis of diffractive effects. In particular, we have to know
that there will not be any propagation along the boundary. 

\begin{proof} 
For $ z \in  \Omega_0 $,  \cite[Theorem A.II.12]{Ge} gives the 
wave front set properties of $ H_j ( z ) $.  In particular it implies
that for $ \varphi_k \in \CIc ( 
\neigh ( \partial \Oo_k ) ) $, $ \varphi_k = 0 $  near $  \Oo_j$, 
\[ \begin{split}   \WFh ( \varphi_k H_j ( z ) v ) =  \big\{  & ( x + t \xi , \xi ) \; : \; t
> 0 \,, \ x  + t \xi \in \supp \varphi_k \,, \  | \xi | = 1
\,, \\
& \ \    ( x , \pi_j ( \xi ) ) \in \WFh ( v )  \cap B^* \partial \Oo_j
\,,  \ \langle \nu_j ( x ) ,
\xi \rangle \geq 0 \big\} \,. \end{split} \]
This and Lemma \ref{l:WFr} immediately give \eqref{eq:WF2}.

The conditions on $ Q_k$'s appearing in \eqref{eq:fio1} mean that
we are cutting off the contributions of rays satisfying 
$ \langle \nu_j ( x ) , \xi \rangle = 0 $, that is with 
$ \pi_j ( \xi ) \in S^* \partial \Oo_j $ (glancing rays on $\Oo_j$), as well as the contributions
of the rays $ \langle \nu_i ( x+t\xi ) , \xi \rangle = 0 $ (glancing
rays on $\Oo_i$).   
This means that the contributions to 
$ Q_i \gamma_i H_j ( z ) Q_j $ only come from the interior of
the hyperbolic regions on the right, $( B^* \partial \Oo_j )^\circ $,
and on the left, $ ( B^* \partial \Oo_i )^\circ $.  The description of 
$ H_j ( z ) $ in the hyperbolic region given in 
\cite[Proposition A.II.3]{Ge} and \cite[\S A.2]{StVo} shows that
it is a sum of zeroth order Fourier integral operators associated to
the relations $ {\mathcal B}^\pm_{ij} $. For $ \Re z < 0 $ the 
forward solution of the eikonal equation 
gives the exponential decay of \eqref{eq:fio11} (where $0<\tau\leq
d_{ji}$, the distance between $\Oo_j$ and $\Oo_i$). 
As pointed out in \cite[\S A.2]{StVo}, this decay is valid for $z$ in the
logarithmic neighbourhood $\Omega_0$.
\end{proof}

\subsection{The microlocal billiard ball map}
\label{damb}
In this section we will show how for several strictly convex
obstacles satisfying Ikawa's condition \eqref{eq:Ik} the operator
$ {\mathcal M} ( z) $ defined in \S \ref{bvp} can be replaced
by an operator satisfying the assumptions of \S \ref{2gr}. This
follows the outline presented in \S \ref{int}.



Namely, we now show that invertibility of $ \Id - \cM ( z ) $ 
can be reduced to invertibility of $ \Id - M ( z ) $ where 
$ M ( z) $ satisfies the assumptions of \S \ref{2gr}. 
That can only be done after introducing a microlocal 
weight function. 

We first consider a general weight. Suppose that 
\[ g_0 \in 
S ( T^* \partial \Oo ; \langle \xi \rangle^{-\infty} ) \defeq
\bigcap_{N\geq 0} 
S ( T^* \partial \Oo ; \langle \xi \rangle^{-N} )  \,, \]
or simply that $ g_0 \in \CIc ( T^* \partial \Oo ) $. 
Then for a fixed but {\em large} $ T $ consider 
\be
\label{eq:g0g}
 g \defeq T \log \left( \frac 1 h \right) \, 
 g_0 \,, \ \ \exp( \pm  g^w ( x , h D  ) )  \in 
\Psi^{ T C_0 , 0 }_{0+} ( \partial \Oo )  \,, \ \ 
C_0 = \max_{T^* \partial \Oo } |g_0|  \,. \ee
We remark that the operators $ \exp ( \pm g^w ( x, h D ) )  $ are
pseudo-local in the sense that
\begin{equation}
\label{eq:WFg}  \WF_h ( e^{ \pm g ^w ( x , h D ) } v ) = \WF_h ( v ) 
\end{equation}
(we see the inclusion from the pseudodifferential nature of $ \exp ( \pm
g^w ) $, and the equality from their invertibility).

For $ k \neq j $, we will write
\be\label{e:M_g}
   (\cM_g)_{kj} (z) \defeq e^{-g^w_k ( x, h D) }\, {\mathcal M}_{kj}(z)\,  e^{g^w_j ( x, h D) }\,. 
\ee
As before, we consider these operators as a matrix acting on $ L^2
( \partial \Oo ) = \bigoplus_{j=1}^J L^2 ( 
\partial \Oo_j ) $.
Using Proposition \ref{p:S2}
we see that $ \cM ( z ) $ is  bounded by 
$ C/ h $ as an operator on $L^2 ( \partial \Oo ) $, uniformly for $ z \in \Omega_0$ (see \eqref{e:Omega_0}). 
Since $ \exp ( \pm g^w ( x, h D ) ) = 
\Oo_{L^2 \rightarrow L^2 } ( h^{-T C_0 })  $ the conjugated
operator satisfies $
\cM_g ( z ) = {\mathcal O}_{L^2 \rightarrow L^2 } ( h^{- 2 T C_0 -1})$.

We want to reduce the invertibility of $ I-\cM_g(z) $ to that
of $I- M_g(z) $, where $ M(z) $ is a Fourier integral operator. That means
eliminating the glancing contributions in $ \cM(z) $
(see \eqref{eq:fio1}).

We start with a simple lemma which shows that Ikawa's condition \eqref{eq:Ik}
eliminates glancing rays and the restrictions to shadows:
\begin{lem}
\label{l:Ik}
Suppose that for some $ j \neq i \neq k$ 
\[  \overline {\mathcal O}_{i } 
\cap {\text{\rm{convex hull}}}( 
\overline {\mathcal O}_j \cup \overline {\mathcal O}_k ) 
= \emptyset \,.  \]
Then, in the notation of Proposition \ref{p:mt}, 
\begin{equation}
\label{eq:lik1}
\overline{  {\mathcal B}^\pm_{ki} } \circ \overline{ {\mathcal B}_{ij}^- } 
 = \emptyset \,.
\end{equation}
\end{lem}
\begin{proof}
Suppose that 
\eqref{eq:lik1} does not hold. Then there exists 
$ x \in \partial \Oo_j $, $ \xi \in \SP^{n-1} $, and $ 
0 < t_1 < t_2 $, such that 
$ x + t_1 \xi \in \overline {\Oo_i } $, and 
$ x + t_2 \xi \in \partial \Oo_k $. But this means that $ \overline {\Oo_i}  $ 
intersects the convex hull of $ \partial \Oo_j $ and $ \partial \Oo_k $.

%
\end{proof}

The trapped set $ \cT $ was defined  using \eqref{eq:defT}, where we
recall that the relation $F=(F_{ij})$ is also given by
$F_{ij}=\cB^+_{ij}$.
Lemma \ref{l:Ik} shows that 
\[    \cT \cap S^* \partial \Oo = \emptyset \,. \]


Lemma \ref{l:Ik} implies that (see Fig.~\ref{f:3disks-AD})
\be
\overline{ {\mathcal B}_{ki}^\pm
}\, ( \tA_i^{\mathcal{G}}) = \emptyset\,,\ \  \ \text{so that}\ \
\tD_i^{\mathcal{G}}\cap \tA_i^{\mathcal{G}}=\emptyset,\ \ \ \text{and similarly}\ \ 
\left( \overline { \mathcal B_{ik}^\pm } \right)^t (\tD_i^{\mathcal{G}}
)=\emptyset\,.
\ee
Since the sets $\tA_i^{\mathcal{G}}$ and $  \widetilde D_i^{\mathcal{G}} $  
are closed and disjoint, we can find small neighbourhoods
\be
U_i^A \defeq \neigh(\tA_i^{\mathcal{G}}) \,,\ \ \ U_i^D \defeq
\neigh(\tD_i^{\mathcal{G}})\,,\ \ \  U_i^D \cap U_i^A = \emptyset\,,
\ee
so that
\be\label{eq:nei}
   \overline{ {\mathcal B}_{ki}^\pm }\, ( U_i^A ) = \emptyset  \,,\qquad
 \left( \overline { \mathcal B_{ik}^\pm } \right)^t (
U_i^D ) = \emptyset \,.
\ee

We can now formulate a Grushin problem which will produce the
desired effective Hamiltonian.
For that let $ \chi_{j,A} , \chi_{j,D}  \in 
\CI ( T^* \Oo_j ; \RR ) $ 
satisfy 
\begin{gather}
\label{eq:conc}
\begin{gathered} 
\chi_{j,A}\rest_{\neigh(\tA_j^{\mathcal G}) }  \equiv 1 \,, \ \ 
\ \ \supp \chi_{j,A } \Subset U_j^A \,, \\
\chi_{j,D}\rest_{\neigh(\tD_j^{\mathcal{G}}) }  \equiv 1 \,, \ \ 
\ \ 
\supp \chi_{j,D } \Subset U_j^D \,.
\end{gathered}
\end{gather}
We let $ \tilde \chi_{j,\bullet} $ have the same properties as 
$ \chi_{j,\bullet } $ with $ \chi_{j,\bullet } = 1 $ on 
$ \supp \tilde \chi_{j ,\bullet } $. 
In view of \eqref{eq:WF2} 
and \eqref{eq:nei}, we have
\begin{equation*}
\begin{split} 
& \cM_{kj}(z)\, \chi_{j,A}^w ( x, h D )  = 
\Oo_{L^2 ( \partial \Oo_j ) \rightarrow \CI ( \partial \Oo_k ) }  
( h^\infty ) \,,  \\
&  \chi_{j,D}^w ( x, h D )\,  \cM_{ji}(z) = 
\Oo_{L^2 ( \partial \Oo_i ) \rightarrow \CI ( \partial \Oo_j ) }  
( h^\infty ) \,,
\end{split}
\end{equation*}
for all $ k \neq j \neq i $, uniformly for $z\in \Omega_0$.
In view of \eqref{eq:WFg}, and using the notations of \eqref{e:M_g},
we have the same properties for the conjugated operator:
\begin{equation}
\label{eq:cru}
\begin{split} 
&  (\cM_g)_{kj}(z)\, \chi_{j,A}
^w ( x, h D ) = 
\Oo_{L^2 ( \partial \Oo_j ) \rightarrow \CI ( \partial \Oo_k ) }  
( h^\infty ) \,,  \\ 
&  \chi_{j,D}^w ( x, h D ) \,
 (\cM_g)_{ji}(z)= 
\Oo_{L^2 ( \partial \Oo_i ) \rightarrow \CI ( \partial \Oo_j ) }  
( h^\infty ) \,.
\end{split}
\end{equation}
For $\bullet=A,D$, let $ \widetilde \Pi_{j,\bullet} $ be an orthogonal finite rank projection on 
$ L^2 ( \partial \Oo_j ) $ such that 
\begin{equation}
\label{eq:prj}  \begin{split}
& \tilde \chi_{j,\bullet}^w \,\widetilde \Pi_{j,\bullet} = 
 \widetilde \Pi_{j,\bullet} \,\tilde \chi_{j,\bullet}^w + {\mathcal O}_{
L^2 \rightarrow \CI }  ( h^\infty ) = \tilde \chi_{j,\bullet}^w + {\mathcal O}_{
L^2 \rightarrow \CI }  ( h^\infty ) \,, \\
&  \chi_{j,\bullet}^w \widetilde \Pi_{j,\bullet} =  \widetilde \Pi_{j,\bullet}  \chi_{j,\bullet}^w  + {\mathcal O}_{
L^2 \rightarrow \CI }  ( h^\infty )  = \widetilde \Pi_{j,\bullet}
+ {\mathcal O}_{L^2 \rightarrow \CI }  ( h^\infty ) \,. 
\end{split}
\end{equation}
Such projections can be found by constructing a real valued
function $ \psi_{j,\bullet}  \in \CIc ( T^* \partial \Oo_j ) $
satisfying $ \psi_{j,\bullet} \equiv 1 $ on $ \supp ( \tilde \chi_{j,\bullet} ) $ and 
$ \psi_{j,\bullet} \equiv 0 $ on $  \supp ( 1 - \chi_{j,\bullet}) $. 
Then $ \widetilde \Pi_{j,\bullet} \defeq \bbbone_{ \psi_{j,\bullet}^w
  ( x, h D ) \geq 1/2  } $ 
provides a desired projection of rank comparable to 
$ h^{1-n} $.

We need one more orthogonal projector $ \Pi_j^\#$, microlocally
projecting in a neighbourhood of $B^*\partial\Oo_j$. Precisely, we
assume that for some cutoff $ \chi_j\in \CIc(T^*\partial\Oo_j) $
with $\chi_j=1$ near $B^*\partial\Oo_j\cup U^A_j\cup U^D_j$, this
projector satisfies
\begin{equation}
\label{eq:pis}  \Pi_j^\# \,\chi_j^w = \chi_j^w \,\Pi_j^\# + \Oo_{L^2
  \rightarrow \CI } ( h^\infty ) = \Pi_j^\# + \Oo_{L^2
  \rightarrow \CI } ( h^\infty )  \,. 
\end{equation}
From Lemma~\ref{l:S1} we easily get the bounds
\be\label{e:M_gPi}
(\cM_g)_{kj}(z)\,(I-\Pi_j^\#) = \Oo_{L^2
  \rightarrow \CI } ( h^\infty )\,,\qquad (I-\Pi_k^\#)\,(\cM_g)_{kj}(z)=\Oo_{L^2
  \rightarrow \CI } ( h^\infty )\,.
\ee

The operator 
\begin{equation} 
\label{eq:Pj} 
P_j \defeq \widetilde \Pi_{j,A} + \widetilde \Pi_{j,D} 
+ \Id - \Pi_j^\#
\,.
\end{equation}
is not a projection but it can be easily modified to 
yield a projection with desired properties:
\begin{lem}
\label{l:proj}
Let $ \gamma $ be positively oriented curve around $ \zeta = 1 $, 
$ \gamma : t \mapsto 1 + \epsilon \exp ( 2 \pi i t ) $, $ \epsilon $
small and fixed.
For $ P_j $ given by \eqref{eq:Pj}  define
\[  \widetilde \Pi_j \defeq \frac 1 { 2 \pi i } \int_\gamma 
( \zeta - P_j )^{-1} d \zeta \,. \]
Then $ \widetilde \Pi_j $ is an orthogonal projection satisfying
\begin{equation}
\label{eq:prjb}  
\widetilde \Pi_j = \tPi_{j,A} + \widetilde \Pi_{j,D}  + \Id
- \Pi_j^\#+ {\mathcal O}_{L^2 \rightarrow \CI }  ( h^\infty ) \,, 
\end{equation}
where $ \widetilde \Pi_{j,\bullet } $, $ \Pi_j^\sharp$  are the
projections in \eqref{eq:prj} and \eqref{eq:pis}.
\end{lem}
\begin{proof}
The operator $ P_j $ 
is not a projection but it is self-adjoint and satisfies
\begin{equation}
\label{eq:alm}
P_j^2 = P_j + {\mathcal O}_{L^2 \rightarrow \CI }  ( h^\infty ) \,, 
\end{equation}
which we check using \eqref{eq:prj}, \eqref{eq:pis} and 
the properties of the functions $ \chi_{j,\bullet} $, $\chi_j$.
Hence its spectrum is contained in $ [ 0 , {\mathcal O} ( h^\infty )]
\cup [ 1 - {\mathcal O} ( h^\infty ),  1 +  
{\mathcal O} ( h^\infty ) ] $. For $ \epsilon $ small enough 
and fixed we can take $ \gamma $ including all spectrum near $ 1 $
(the statements implicitly assume that $ h $ is small enough). 
For $ \zeta \in \gamma $ we write 
\[ (\zeta - P_j )^{-1} = \zeta^{-1} ( \Id + ( \zeta - 1 )^{-1} P_j ) 
\left( \Id - \zeta^{-1} ( \zeta - 1 )^{-1} ( P_j^2 - P_j ) \right)^{-1} \,. \]
The inverse on the right hand side exists in view of 
\eqref{eq:alm} and satisfies 
\[ \left( \Id - \zeta^{-1} ( \zeta - 1 )^{-1} ( P_j^2 - P_j ) \right)^{-1}  
= \Id + {\mathcal O}_{L^2 \rightarrow \CI }  ( h^\infty )  \, \]
uniformly on $ \gamma $. Inserting these two formul{\ae} into 
the integral defining  $ \Pi_j $ gives \eqref{eq:prjb}.
\end{proof}

The next lemma provides the property crucial in the
construction of the Grushin problem:
\begin{lem}
\label{l:cru}
For $ \widetilde \Pi_j $ defined above and any 
$ k \neq j \neq i $ we have
\[  ({\mathcal M}_g)_{kj}(z)\, 
\widetilde   \Pi_j \, 
({\mathcal M}_g) _{ji} (z) = 
{\mathcal O} ( h^\infty ) \; : \; L^2( \partial \Oo_i ) 
\longrightarrow  \CI ( \partial \Oo_k ) \,, \]
uniformly for $z\in \Omega_0$.
\end{lem}
\begin{proof}
To simplify the notation we write $\cM_{ij}$ instead of $(\cM_g)_{ij}(z)$.

Using 
\eqref{eq:prjb} and
\eqref{e:M_gPi}, 
we can write 
\[   {\mathcal M}_{kj}\, \tPi_j \,{\mathcal M}_{ji} 
= {\mathcal M}_{kj} 
(\tPi_{j,A} +  \tPi_{j,D}) {\mathcal M}_{ji} 
+  {\mathcal O}_{ L^2 \rightarrow \CI }  ( h^\infty ) \,. 
\]
From \eqref{eq:prj} we use
\[  \tPi_{j,A}=\chi_{j,A}^w  \,\tPi_{j,A}  +\Oo(h^\infty)\,,\qquad 
\tPi_{j,D}=\tPi_{j,D}\,\chi_{j,D}^w  +\Oo(h^\infty)\,,  \]
and hence, using \eqref{eq:cru}, we complete the proof.
\end{proof}

Now define the following orthogonal projection:
\begin{equation}
\label{eq:Pih}   \Pi_h \defeq \diag ( \Id - \widetilde \Pi_j ) \; : \; 
L^2 ( \partial \Oo) \longrightarrow L^2 (  \partial \Oo)  \,.
\end{equation}
Since each $\Id - \tPi_j$ is microlocalized on a compact
neighbourhood of $B^*\partial \Oo_j$, it has a rank
comparable with $h^{1-n}$, and so does $\Pi_h$.

Using this projection we obtain the main result of this
section:
\begin{thm}
\label{th:obs}
Let $ \Pi_h $ be given by \eqref{eq:Pih} and \eqref{eq:prj},
and $ {\mathcal M}_g ( z , h ) $ be defined by \eqref{eq:bmon}
and \eqref{e:M_g}. 

If $ \Oo_j $ are strictly convex and satisfy the Ikawa condition
\eqref{eq:Ik} then the following Grushin problem is well posed
for $ z \in D ( 0 , C ) $:
\[  \left( \begin{array}{ll} 
\Id - {\mathcal M }_g ( z , h )  & \Pi_h^t \\
\ \ \ \ \ \ \ \Pi_h &  \ 0 \end{array} \right) \; : \; L^2 ( \partial 
\Oo) \oplus W_h
\longrightarrow L^2 ( \partial \Oo  ) \oplus W_h \,, 
\ \ W_h \defeq \Pi_h  L^2 ( \partial \Oo )  \,, \]
where $ \Pi^t_h : W_h \hookrightarrow L^2 ( \partial \Oo ) $.

The effective Hamiltonian is given by 
\[ E_{-+} ( z, h ) =  - \Big( \Id_{W_h}  - \Pi_h \big(M_g ( z, h) + R ( z, h ) \big) \Pi_h  \Big) \,,\]
where 
\begin{equation}
\label{eq:fio3}
 M_g ( z , h ) = \varphi^w {\mathcal M }_g ( z , h )  \varphi^w \in 
I^0_{0+} ( \partial \Oo \times \partial \Oo , F' ) \,, \ \ 
 \varphi^w = \diag ( 1- \tilde \chi_{j,A}^w - \tilde \chi_{j,D}^w ) \,,
\end{equation}
with the relation $ F $ given by \eqref{eq:bill}, $ \tilde 
\chi_{j,\bullet}$ satisfying \eqref{eq:conc}, and 
\[  R ( z, h ) = {\mathcal O}_{L^2 ( \partial \Oo ) \rightarrow
\CI ( \partial \Oo ) } ( h^\infty ) \,. \]
\end{thm}
\begin{proof}
We first observe that Lemma \ref{l:cru} gives
\begin{equation}
\label{eq:fol}
 {\mathcal M }_g ( z , h )  \Pi_h  {\mathcal M }_g ( z , h )  = 
 {\mathcal M }_g ( z , h )^2 + {\mathcal O}_{L^2 ( \partial \Oo ) \rightarrow
\CI ( \partial \Oo ) } ( h^\infty )  \,. 
\end{equation}
The theorem follows from this.
Indeed, \eqref{eq:fol} implies that, in abbreviated notation,
\begin{gather*} 
  \begin{pmatrix}   \Id - {\mathcal M }_g   & \Pi^t \\
\ \  \ \ \Pi &  \ 0 \end{pmatrix}  
 \begin{pmatrix}  (\Id + ( \Id - \Pi) {\mathcal M}_g ) ( \Id - \Pi) 
  &  \Pi^t + (\Id - \Pi) {\mathcal M}_g \Pi^t
 \\
\ \  \ \ \ \Pi ( \Id + \cM_g ( \Id - \Pi) )  &  \  - (\Id_{W} - \Pi \cM_g \Pi)  \end{pmatrix} 
\\ \ = \, \Id_{ L^2 ( \partial \Oo ) \oplus W_h } + 
\Oo_{ L^2 ( \partial \Oo ) \oplus W_h  \rightarrow 
\CI ( \partial \Oo ) \oplus W_h  } ( h^\infty ) \,.
\end{gather*}
The exact inverse in then obtained by a Neumann series inversion. In
view of \eqref{eq:prj} and 
\eqref{eq:fio1} we obtain \eqref{eq:fio3}.
\end{proof}


\medskip

\noindent
{\em Proof of Theorem \ref{th:1}.} To apply Theorem \ref{th:bound}
from \S \ref{upbd} we need to show that we can choose $ g $ so that
$ M_g ( z, h ) $ satisfies the conditions in Definition
\ref{def:HQMO}. The only assumption that needs verification is
\eqref{eq:Mlo}.  For that we take $ g_0 $ in \eqref{eq:g0g} 
given in Lemma \ref{lem:G_0} (proved in the appendix),
with the map being the billiard ball relation, $ F $.

Egorov's theorem (Proposition \ref{p:Egorov})  then shows that 
\eqref{eq:Mlo} holds and we can make $ {\mathcal W} $ as 
close to the trapped set as we wish. 
\qed

\vspace{0.5cm}
\begin{center}
\noindent
{\sc  Appendix: {Construction of an escape function for an open map $F$}}
\end{center}
\vspace{0.4cm}
\renewcommand{\theequation}{A.\arabic{equation}}
\refstepcounter{section}
\renewcommand{\thesection}{A}
\setcounter{equation}{0}

In this appendix we prove Lemma~\ref{lem:G_0}  by explicitly constructing an escape function $\gz$.
The only assumption we need on $F$ is the fact that the trapped set
$\cT\Subset \tD$, $\cT\Subset \tA$, where $\tD$, resp. $\tA$ are the
(closed) departure and arrival sets of $F$. Our construction is
inspired by  similar constructions 
in \cite[Lemma 4.3]{DatVas10}  and \cite{VasZw00}.
Here we will independently construct functions $g_{\pm}$ with good
escape properties away from the incoming and outgoing tails
$\cT_{\mp}$,  respectively.

Let us start with the function $g_+$.
We can take $V_{\pm}\Subset \cU$, open neighbourhoods of the tails $\cT_{\pm}$, such that 
$$
\quad V_+\cap V_-\Subset \cW_2\,,
$$
where $\cW_2$ is the neighbourhood of $\cT$ in the
statement of Lemma \ref{lem:G_0}.
We will first construct a function $g_+\in\CIc(T^*Y)$ such that 
\begin{gather}\label{e:g_+}
\begin{gathered}
\forall \rho\in \cW_3,\quad g_+ (F(\rho)) - g_+ (\rho) \geq 0,\\
\forall \rho\in \cW_3\setminus V_-,\quad g_+ (F(\rho)) - g_+ (\rho) \geq 1\,.
\end{gathered}
\end{gather}
We will perform a local construction, based on a finite set of points  $\rho\in \tD\setminus V_-$.
\begin{figure}
\begin{center}
\includegraphics[angle=-90,width=0.7\textwidth]{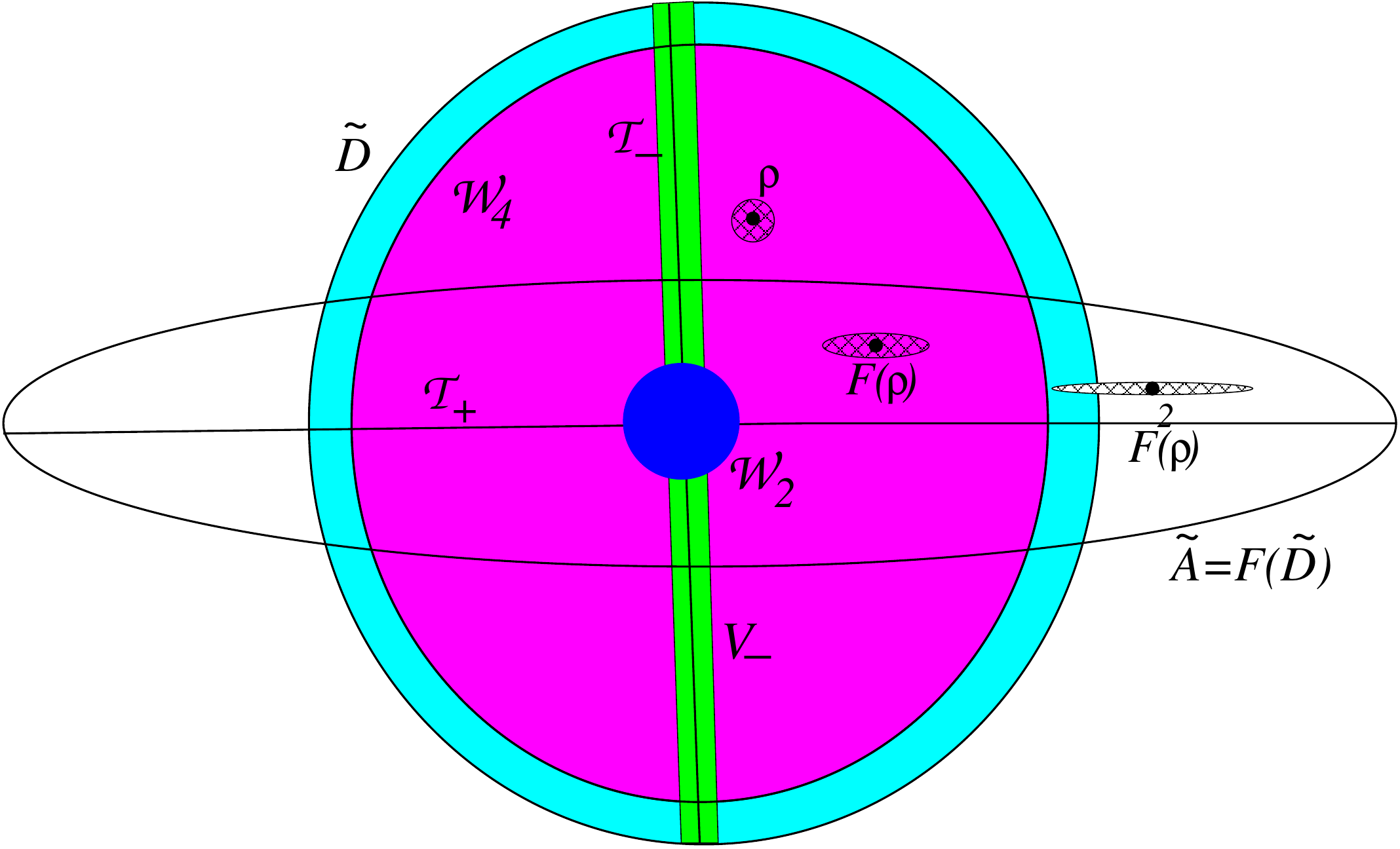}
\caption{Sketch of the construction of an escape function
  $g_{\rho,+}$. The blue (resp. pink) circle denotes the
  neighbourhood $\cW_2$ (resp. $\cW_4$) of the trapped set inside
  $\tD$, the vertical green rectangle is the neighbourhood $V_-$ of
  the incoming tail $\cT_-$. The three ellipses
  with patterns indicate the sets $F^{k}(V_\rho)$ surrounding
  $F^k(\rho)$, $0\leq k\leq 2$. Here the escape time $n_+(\rho)=2$. \label{f:g0-construct}}
\end{center}
\end{figure}
Consider a compact set $\cW_4$ such that $\cW_3\Subset\cW_4\Subset \tD$. Take  any point $\rho\in \cW_4\setminus V_-$. We define its forward escape time by
$$
n_{+}=n_{+}(\rho)\defeq \min\{k\geq 1,\ F^k(\rho)\not\in\cW_4\}\,.
$$
Since $\rho\not\in V_-$, this time is uniformly bounded from
above. Besides, the forward trajectory $\{F^k(\rho),\ 0\leq k\leq n_{+}\}$ is a set of mutually different points, with $F^{n_{+}}(\rho)\in \tA\setminus \cW_4$. 
Since $\tA\setminus\cW_4$ is relatively open, there exists a small
neighbourhood $V_\rho$ of $\rho$, such that the neighbourhoods
$F^k(V_\rho)$, $0\leq k\leq n_{+}-1$, are all inside $\tD\setminus \cT_-$, while  $F^{n_{+}}(V_\rho)\in\tA\setminus \cW_4$ (see Fig.~\ref{f:g0-construct}). 

Take a smooth cutoff $\chi_\rho\in \CIc(V_\rho,[0,1])$, with $\chi_\rho=1$ in a smaller neighbourhood $V'_\rho\Subset V_\rho$ of $\rho$, and consider its push-forwards
$$
\chi_{\rho,k}\defeq \begin{cases}\chi_\rho\circ F^{-k}\quad& \text{on $F^k(V_\rho)$},\\
0&\text{outside $F^k(V_\rho)$},\end{cases}
\quad  0\leq k\leq n_{+}\,.
$$ 
The supports of the $\chi_{\rho,k}$, $0\leq k\leq n_{+}-1$, are all
contained in $\tD\setminus \cT_-$, while $\supp \chi_{\rho,n_{+}}\subset \tA\setminus \cW_4$.
From there, we define
\begin{align*}
\text{the set}\quad W'_{\rho,+} &\defeq \bigcup_{k=0}^{n_{+}-1} F^k(V'_\rho)\subset \tD\,,\\
\text{the function}\quad g_{\rho,+}&\defeq \sum_{k=0}^{n_{+}} (k+1)\chi_{\rho,k}\,.
\end{align*}
The function $g_{\rho,+}$ is smooth, and on $\tD$ it satisfies
$$
g_{\rho,+}\circ F - g_{\rho,+} = \chi_{\rho}\circ F + \sum_{k=0}^{n_{+}-1}\chi_{\rho,k} - (n_{+}+1)\chi_{\rho,n_{+}}\,.
$$ 
The properties of the supports of the $\chi_{\rho,k}$ imply that
$$
g_{\rho,+}\circ F - g_{\rho,+}\geq 0\ \ \text{in $\cW_4$},\qquad 
g_{\rho,+}\circ F - g_{\rho,+}\geq 1\ \ \text{in $W'_{\rho,+}\cap \cW_4$}\,.
$$
Since $\cW_4\setminus V_-$ is a compact set, we may extract a finite set of points $\{\rho_j\in \cW_4\setminus V_-\}_{j=1,\ldots,J}$ such that $\bigcup_{j=1}^J W'_{\rho_j,+}$ is an open cover of $\cW_4\setminus V_-$. The sum
$$
g_+\defeq \sum_{j=1}^J g_{\rho_j,+}
$$
is smooth in $\UU$, and satisfies the properties
\eqref{e:g_+}. Furthermore, for each $\rho_j$ the function
$g_{\rho_j,+}$ vanishes near $\cT_-$, so there exists
$V'_-\Subset V_-$  a neighbourhood of $\cT_-$ such that all
$g_{\rho_j,+}$, and also $g_+$, vanish on $V'_-$.

Applying the same construction in the past direction, we construct a
smooth function $\tilde g_-$ and a neighbourhood $V'_+\Subset V_+$ of $\cT_+$, such that 
\begin{gather*}
\begin{gathered}
\forall \rho\in F(V'_+),\quad \tilde g_-(\rho)\equiv 0,\quad \\
\forall \rho\in F(\cW_3),\quad \tilde g_- (F^{-1}(\rho)) - \tilde g_- (\rho) \geq 0,\\
\forall \rho\in F(\cW_3\setminus V_+),\quad \tilde g_- (F^{-1}(\rho)) - \tilde g_- (\rho) \geq 1\,.
\end{gathered}
\end{gather*}
(notice that the sets $F(V'_+)$, $F(V_+)$ and $\cW_3$ have the
appropriate properties with respect to $F^{-1}:\tA\mapsto \tD$). 
The function $g_-\defeq -\tilde g_-\circ F$ then satisfies 
\begin{gather}
\begin{gathered}
\forall \rho\in V'_+,\quad  g_-(\rho)\equiv 0,\quad \\
\forall \rho\in \cW_3,\quad g_- (F(\rho)) - g_- (\rho) \geq 0,\\
\forall \rho\in \cW_3\setminus V_+,\quad g_- (F(\rho)) - g_- (\rho) \geq 1\,.
\end{gathered}
\end{gather}
The function $g_0\defeq g_+ + g_-$ satisfies conditions in Lemma
\ref{lem:G_0}, and vanish on $\cW_1\defeq V'_+\cap V'_-$.

\bigskip

\noindent
{\sc Acknowledgments.}
We would like to thank the National Science Foundation
for partial support under the grant
DMS-0654436. This article was
completed while the first author was visiting the Institute of Advanced Study in Princeton,
supported by the National Science Foundation under agreement No. 
DMS-0635607. 
The first two authors were also partially supported by
the Agence Nationale de la Recherche under the grant
ANR-09-JCJC-0099-01.
The second author was supported by fellowship FABER of the Conseil
R\'egional de Bourgogne.
Finally, we thank Edward Ott for his permission to include figures from 
\cite{Ott} in our paper.


\end{document}